\documentclass[reqno]{amsart}
\usepackage{amscd}
\usepackage[dvips]{graphics}

\usepackage{graphicx,subfigure,float,caption2,color,epstopdf}

\def\beq{\begin{equation}}
\def\eeq{\end{equation}}
\def\ba{\begin{array}}
\def\ea{\end{array}}

\def \RpN{\mathbb{R}_+^{n+1}}

\def\cal{\mathcal}

\numberwithin{equation}{section}
\newenvironment{abs}{\textbf{Abstract}\mbox{  }}{ }
\newenvironment{key words}{\textbf{Keywords}\mbox{  }}{ }
\newtheorem{theorem}{Theorem}[section]
\newtheorem{definition}[theorem]{\textbf{Definition}}
\newtheorem{corollary}[theorem]{\textbf{Corollary}}
\newtheorem{proposition}[theorem]{\textbf{Proposition}}
\newtheorem{lemma}[theorem]{Lemma}
\renewenvironment{proof}{\noindent{\textbf{Proof.}}}{\hfill$\Box$}
\theoremstyle{remark}
\newtheorem{remark}[theorem]{\textbf{Remark}}
\theoremstyle{plain}

\begin{document}
\title{\textbf Divergent operator with degeneracy and related sharp inequalities}
\author  {Jingbo Dou, Liming Sun, Lei Wang,   and Meijun Zhu}

\address{Jingbo Dou, School of Mathematics and Information Science, Shaanxi Normal University, Xi'an, Shaanxi, 710119, China}

\email{jbdou@snnu.edu.cn}

\address{Liming Sun, Department of Mathematics, University of British Columbia, Vancouver, BC V6T 1Z2, Canada}

\email{lsun@math.ubc.ca}

\address{Lei Wang, Academy of Mathematics and Systems Sciences, Chinese Academy of Sciences, Beijing 100190, P.R. China; School of Mathematical Sciences,  University of Chinese  Academy of Sciences, Beijing 100049,  P.R. China, and Department of Mathematics, The University of Oklahoma, Norman, OK 73019, USA
}

\email{wanglei@amss.ac.cn}

\address{ Meijun Zhu, Department of Mathematics,
The University of Oklahoma, Norman, OK 73019, USA}

\email{mzhu@math.ou.edu}


\maketitle

\noindent
\begin{abs}
In this paper we classify all nonnegative extremal functions to a sharp weighted Sobolev inequality on the upper half  space, which involves a divergent operator with degeneracy on the boundary.  As an application of the results, we derive a sharp Sobolev type inequality involving  Baouendi-Grushin operator, and classify certain extremal functions for all $\tau>0$ and $m\ne2 $ or $ n\ne1$.

\end{abs}

\smallskip

\section{\textbf{Introduction}\label{Section 1}}

The current work is motivated and heavily influenced by the popular work of Caffarelli and Silvestre \cite{CS07},  and by our recent studies on  the extension type operators (see, for example, Dou and Zhu \cite{DZ15}, Dou, Guo and Zhu \cite{DGZ17}, Gluck \cite{G18}, Gluck and Zhu \cite{GZ19} and  Wang and Zhu \cite{WZ19}). Our results, among the other things, partially answer an open questions for years (see Theorem \ref{L-2} and Proposition \ref{symmetrization-z} below).

Throughout the paper, we denote $\mathbb{R}^{n+1}_+=\{(y, t) \in \mathbb{R}^{n+1} \ : \ \ y\in \mathbb{R}^n, \ \  t>0\}$ as the upper half space.

\subsection{A divergent operator}
In  \cite{CS07}, Caffarelli and Silvestre study the following extension problem for $\alpha \in (-1, 1)$:
\begin{equation}\label{equ1.2-1}
	\begin{cases}
	div(t^\alpha  \nabla u)=0,  & \quad \text{in}\quad \mathbb{R}^{n+1}_+,\\
	u(y, 0)=f(y), & \quad\text{on}\quad \partial\mathbb{R}^{n+1}_+.
	\end{cases}
	\end{equation}
A global defined fractional Laplacian operator on $f(y)$ in a good space is given by
$$(-\Delta)^{\frac{1-\alpha }2}f(y)=-C\lim_{t \to 0^+} t^\alpha\frac {\partial u}{\partial t}(y, t)$$for a suitable constant $C$.

For $f(y)$ in a good space,  the weak solution $u(y,t)$ to \eqref{equ1.2-1} can be represented, up to a constant multiplier, as an extension of $f(y)$ via operator ${\cal P}_\alpha$:
$$
u(y, t)= {\cal P}_\alpha(f)(y,t) :=\int_{ \mathbb{R}^{n}}  P_\alpha(y-x, t) f(x)dx,
$$
whose positive kernel is
$$
P_\alpha(y, t)=\frac{ t^{{1-\alpha}}}{(|y|^2+t^2)^{\frac{n+1-\alpha}2}},\ \ y\in\mathbb{R}^n,\ t>0.
$$
See more discussions in the introduction part in Wang and Zhu  \cite{WZ19}  for the related studies of the extension operators involving divergent operator $div(t^\alpha\nabla )$.

\subsection{New nonlinear equations}
Our original interest is to understand the following general equation
\begin{equation}\label{basic-equ}
	-div(t^{\alpha} \nabla u)=f(u, t),  \quad u \ge 0, \quad \text{in}\quad \mathbb{R}^{n+1}_+
		\end{equation}
with or without explicitly given boundary conditions.  For $f(u, t)=0$, as we mentioned above, equation \eqref{basic-equ} was discussed by Caffarelli and Silvestre \cite{CS07} in connecting to the study of fractional Laplacian operators; the Liouville type theorems for this homogeneous equation were obtained recently by Wang and Zhu \cite{WZ19}.  Here, we shall study equation \eqref{basic-equ} for $f(u,t)=t^\beta |u|^{p-1}$.  After a standard scaling argument for an associated Sobolev type inequality on an unbounded domain (see, for example, inequality \eqref{GGN-2} below), we can see that
\begin{equation}\label{cr-p}
p^*=\frac{2n+2\beta+2}{n+\alpha-1}
\end{equation}
is the so called critical exponent.

Denote
$$C^\infty_0(\overline{\mathbb{R}^{n+1}_+})=\{u|_{\overline{\mathbb{R}^{n+1}_+}}: u\in C_0^\infty(\mathbb{R}^{n+1})\}.$$
We have the following  inequality.
\begin{proposition}\label{main-1}
	Assume $n\geq 1$,  $l>-1$, $k >0$ and $ \frac{nl}{n+1}\le k\le l+1$.
	There is a constant $C=C(n,k)>0$ such that for all $u\in C_0^\infty (\overline{\mathbb{R}^{n+1}_+})$,
	\begin{equation}\label{GGN-1}
	(\int_{\mathbb{R}^{n+1}_+ }t^l |u|^{\frac{n+l+1}{n+k}} dydt)^{\frac{n+k}{n+l+1}} \le C \int_{\mathbb{R}^{n+1}_+ }t^k|\nabla u| dy dt.
		\end{equation}
		\end{proposition}
		
Proposition  \ref{main-1} is  a known result. In fact, it is true even for $k \le 0$, see Maz'ya \cite[inequality (2.1.35)] {Maz1985}.  Here we give a direct proof for $k>0$, similar to the original one in Gagliardo \cite{Ga59} and Nirenberg \cite{Ni59}. See Section 2 for more details.

Define the
weighted Sobolev space ${\cal D}_\alpha^{1,p}(\mathbb{R}^{n+1}_+)$ as the completion of the space $C^\infty_0(\overline{\mathbb{R}^{n+1}_+})$ under the norm
\[
\|u\|_{{\cal D}_\alpha^{1,p}(\mathbb{R}^{n+1}_+)}=\big(\int_{\mathbb{R}^{n+1}_+}t^{\alpha}|\nabla u|^pdy dt\big)^\frac1p.
\]
We say $ u \in {\cal D}_{\alpha, loc}^{1,p}(\mathbb{R}^{n+1}_+)$, if $u\in W^{1,1}_{loc}(\mathbb{R}^{n+1}_+)$ and for any compact set $K \subset \overline{\mathbb{R}^{n+1}_+}$,
$$
\int_{K}t^\alpha |\nabla u|^{p}dydt<\infty,\, \ \int_{ K}t^\alpha |u|^{p}dydt< \infty.
$$

Using H\"older inequality, we can derive  the following inequality from \eqref{GGN-1}.

\begin{corollary}\label{cor1-4}
Assume that $n\ge 1$ and $\alpha, \beta$ satisfy
\begin{equation}\label{beta-1}
\alpha>0,\ \beta>-1, \  \frac{n-1}{n+1}\beta\le\alpha\le\beta+2.
\end{equation}
There is a positive constant $C_{n+1, \alpha,\beta}>0$ such that, for all $u\in {\cal D}_{\alpha}^{1,2}(\mathbb{R}^{n+1}_+)$,
\begin{equation}\label{GGN-2}
(\int_{\mathbb{R}^{n+1}_+ }t^\beta |u|^{p^*}dydt)^{\frac{2}{p^*}} \le  C_{n+1, \alpha,\beta} \int_{\mathbb{R}^{n+1}_+ }t^{\alpha}|\nabla u|^2 dy dt.
\end{equation}
\end{corollary}

Sobolev inequalities with monomial weights were also  studied early by Cabre and Ros-Oton \cite[Theorem 1.3]{Cabre2013}.  In particular, for $\alpha=\beta\ge 0$, inequality \eqref{GGN-2} and its sharp form were obtained  by  Cabre and Ros-Oton  \cite[Theorem 1.3]{Cabre2013},  Bakry, Gentil and Ledoux \cite{BGL2013} and Nguyen \cite{Ng2015}, essentially from the classical sharp Sobolev inequality in $\mathbb{R}^{n+1}$.

To study the sharp form of inequality   \eqref{GGN-2} for general $\alpha$ and $ \beta$,  we define
\begin{equation}\label{sharp-C-1}
S_{n+1, \alpha,\beta}:= \inf_{u \in C_0^\infty(\overline{\mathbb{R}^{n+1}_+})\setminus \{0\}}\frac  {\int_{\mathbb{R}^{n+1}_+ }t^\alpha|\nabla u|^2 dy dt}{(\int_{\mathbb{R}^{n+1}_+ }t^\beta |u|^{p^*}dydt)^{\frac{2}{p^*} }}>0.
\end{equation}

Using the concentration compactness principle, we  obtain  the existence of the extremal functions for $\frac {n-1}{n+1} \beta<\alpha<\beta+2$. The case $\alpha = \frac {n-1}{n+1} \beta$ is more complicated, see details in Section 3.
\begin{theorem}\label{sharp-C-2} Assume that $n\ge 1$ and $\alpha, \beta$ satisfy
	\begin{equation}\label{beta-2}
	\alpha>0,\ \beta>-1, \ \frac{n-1}{n+1}\beta<\alpha<\beta+2.
	\end{equation}
	Constant $S_{n+1, \alpha,\beta}$ is achieved by a nonnegative  extremal function $u \in {\cal D}_{\alpha}^{1,2}(\mathbb{R}^{n+1}_+)$.
		\end{theorem}

Let $u\ge 0$ be an extremal function to $ S_{n+1,\alpha,  \beta}$, then up to the multiple of some constant, for any $ \phi \in   {\cal D}_{\alpha}^{1,2}(\mathbb{R}^{n+1}_+)$, it holds
\begin{equation}\label{weak-1}
\int_{\mathbb{R}^{n+1}_+ }t^{\alpha} \nabla u \cdot \nabla \phi dy dt =\int_{\mathbb{R}^{n+1}_+ }t^\beta u^{p^*-1}\phi dydt.
\end{equation}
If we  know that $u \in C^2(\mathbb{R}^{n+1}_+) \cap C^1(\overline{ \mathbb{R}^{n+1}_+})$, then $u$ is a classical solution to the following equation
	\begin{equation}\label{genequ-1}
	\begin{cases}
	-div(t^\alpha \nabla u)=t^{\beta}u^{p^*-1},\quad \;\;& \text{in}~ \mathbb{R}^{n+1}_+,\\
	\lim_{t \to 0^+} t^\alpha\frac{\partial u}{\partial t}=0,\quad \;\;& \text{on}~ \partial \mathbb{R}^{n+1}_+.
	\end{cases}
	\end{equation}

\begin{definition}\label{sol-d-1}  $u \in {\cal D}_{\alpha}^{1,2}(\mathbb{R}^{n+1}_+)$ is said to be a weak solution to \eqref{genequ-1} if  equality \eqref{weak-1} holds for all $ \phi \in {\cal D}_{\alpha}^{1,2}(\mathbb{R}^{n+1}_+).$
\end{definition}

Due to the degeneracy of the operator, we can not show that any weak solution is in $C^1(\overline{ \mathbb{R}^{n+1}_+})$. But we are able to show

\begin{theorem}\label{thm-reg} Let $n\ge 1,$ and $\alpha, \beta$ satisfy
	\begin{equation}\label{beta-3}
	\alpha>0,\ \beta>-1,  \, \frac{n-1}{n+1}\beta\le\alpha<\beta+2.
	\end{equation}
	 Assume that $ u \in {\cal D}_{\alpha}^{1,2}(\mathbb{R}^{n+1}_+)$ is a weak solution to equation \eqref{genequ-1}, then $u \in C^2(\mathbb{R}^{n+1}_+) \cap C^\gamma_{loc} (\overline{ \mathbb{R}^{n+1}_+}) $ for some $\gamma \in (0, 1)$.
\end{theorem}

In the proof of regularity, as a byproduct, we show that a nontrivial nonnegative weak solution to equation \eqref{genequ-1} must be positive in $\overline{\mathbb{R}^{n+1}_+}$. Then we obtain the following Liouville theorem for positive weak solutions to equation \eqref{genequ-1} for $\alpha > 0$.
In two special cases, we obtain the precise form of these solutions, thus  can compute precisely the sharp constant to inequality \eqref{GGN-2}.

\begin{theorem}\label{L-1} Let  $n\ge 1,$ and $\alpha, \beta$ satisfy \eqref{beta-2}.
For positive weak solution $u\in {\cal D}_{\alpha}^{1,2}(\mathbb{R}^{n+1}_+)$ to equation \eqref{genequ-1}, we have, up to the multiple of some constant,
\begin{equation}\label{type-0}
u(y, t)=(\frac{1}{|y-y^o|^2+(t+A)^2})^{\frac{n+\alpha-1}2} \psi(|\frac{(y-y^o, t+A)}{|y-y^o|^2+(t+A)^2}-(0, \frac1{2A})|),
\end{equation}
for some $y^o\in\mathbb{R}^n$, $ A>0$, $\psi(r)>0$  and $\psi\in C^2[0,\frac1{2A})\cap C^0[0,\frac1{2A}]$ satisfying an ordinary differential equation
\begin{equation}\label{ode-0}
	\begin{cases}
		\psi''+(\frac{n}{r}-\frac{2\alpha A r}{\frac{1}{4A^2}-r^2})\psi'-\frac{\alpha(n+\alpha-1)A}{\frac{1}{4A^2}-r^2}\psi
		=-C(\frac{1}{4A^2}-r^2)^{\beta-\alpha} \psi^{\frac{n+2\beta-\alpha+3}{n+\alpha-1}}, \; 0<r <\frac{1}{2A},\\
		\psi(\frac{1}{2A})=A^{\frac{n+\alpha-1}{2}},\ \psi'(0)=0,\   \lim_{r\to (\frac{1}{2A})^{-}}\big(\frac{1}{4A^2}-r^2\big)^\alpha \psi'(r)=0,
	\end{cases}
\end{equation}
for one constant $C>0$ independent of  $A$.  Furthermore, for $n \ge 2$, or for $n=1$ and $\alpha, \beta$ satisfying  an additional assumption:
\begin{equation}\label{constraint}
\frac{1-(1-\alpha)^2}{4}\leq\frac{\alpha(2+\beta)}{(\alpha+\beta+2)^2},
\end{equation} there is only one positive solution to equation \eqref{ode-0}.

Moreover, in following two cases, the solutions can be  explicitly written out.

\noindent  1).  For $\beta=\alpha-1$, if $\alpha>0 $ for $n\geq 2$ or $\alpha\in(0,\frac{1}{2}]\cup[\frac{1+\sqrt{17}}{4},\infty) $ for $n=1$, then up to the multiple of some constant, $u(y, t)$ must be in the form of
\begin{equation}\label{sol-1-0}
u(y,t)= \Big(\frac{A}{(A+t)^2+|y-y^o|^2}\Big)^{\frac{n+\alpha-1}{2}},
\end{equation}
where $ A>0$, $y^o \in \mathbb{R}^{n}$,
and
$$S_{1,\alpha, \alpha-1}=\alpha(n+\alpha-1)
\big[\pi^\frac n2\frac{\Gamma(\alpha)\Gamma({\frac{n}2+\alpha})}{\Gamma(n+2\alpha)}\big]^{\frac{1}{n+\alpha}}.
$$

\noindent  2).   For $\beta=\alpha$, if $\alpha>0$ for $n\geq 2$ or $\alpha\ge\sqrt{2}$ for $n=1$, then up to the multiple of some constant, $u(y, t)$ must be in the form of
\begin{equation}\label{sol-2-0}
u(y,t)= \Big(\frac{A}{A^2+t^2+|y-y^o|^2}\Big)^{\frac{n+\alpha-1}{2}},
\end{equation}
where $A>0$, $y^o \in \mathbb{R}^{n}$, and
$$S_{1,\alpha, \alpha}=(n+\alpha-1)(n+\alpha+1)
\big[\frac{\pi^\frac n2}2\frac{\Gamma(\frac{\alpha+1}2)\Gamma(\frac{n+\alpha+1}2)}{\Gamma({n+\alpha+1})}
\big]^{\frac{2}{n+\alpha+1}}.
$$
\end{theorem}	

\begin{remark}If we know that there is a  nonnegative weak solution $u\in {\cal D}_{\alpha}^{1,2}(\mathbb{R}^{n+1}_+)$ to equation \eqref{genequ-1} with $\alpha= \frac{n-1}{n+1}\beta$, then the same argument holds in proving the regularity of $u$, \eqref{type-0} as well as the uniqueness of positive solutions to equation \eqref{ode-0}.
\end{remark}

\begin{remark} Formula \eqref{type-0} indicates that $u(y, t)$ is ``almost'' a radially symmetric function in the sense that equation \eqref{genequ-1} can be reduced into the  ODE \eqref{ode-0}.
	\end{remark}

Theorem \ref{L-1} part 2) for $\alpha=\beta =0$ follows from the classical result of Caffarelli, Gidas and Spruck \cite{CGS1989}. See Zhu's thesis \cite{Zhu96} for another proof via the method of moving spheres.  Here, we will use the method of moving spheres to prove Theorem \ref{L-1}.  The method of moving spheres enables us to obtain the precise form of positive solutions to equation \eqref{genequ-1} on the boundary $\partial \mathbb{R}^{n+1}_+$. We then transform the equation into a new equation on a ball with constant boundary value, and successfully show that  all solutions to the new equation must be radially symmetric with respect to the center of the ball. For $\alpha$, $\beta$ satisfying the conditions in Theorem \ref{L-1}, we are able to show that the new equation has a unique radially symmetric solution. In two cases: $\beta=\alpha-1$ and $\beta=\alpha$, we can write down the precise unique solution to  the ODE \eqref{ode-0}, which leads to the complete classification of positive solutions.

\subsection{Baouendi-Grushin Operator} As an application of sharp inequality \eqref{GGN-2} and the classification results in Theorem \ref{L-1}, we consider the following critical semilinear equation with Baouendi-Grushin operator
\begin{equation}\label{gru-1}
\Delta_{z}u+(\tau+1)^2|z|^{2 \tau}\Delta_x u=-u^{\frac{Q+2}{Q-2}},  \ \ u>0 \ \ in \ \ \mathbb{R}^{n+m},
\end{equation}
where $\tau \in (0, \infty), n,m\ge1$, $x\in \mathbb{R}^n$, $z \in \mathbb{R}^m$ and $Q=m+n(\tau+1)$  is the homogeneous dimension.  The partial differential operator $\cal{L}:= \Delta_z+(\tau+1)^2|z|^{2 \tau} \Delta_x$ is often called Baouendi-Grushin operator (\cite{Bao1967, Gru1970,Gru1971}).
 For $\tau=0$ or $n=0$,  equation \eqref{gru-1} is the constant scalar curvature equation on $\Bbb{R}^{n+m}$, which is widely studied, and well-understood through the work of Gidas, Ni and Nirenberg \cite{GNN1981} and the work of Caffarelli, Gidas and Spruck \cite{CGS1989} (see, Zhu's thesis \cite{Zhu96} for a simpler proof via the method of moving spheres). For $n\ge 1$ and $\tau>0$, the operator is degenerate on $|z|=0$. In particular, for  $n=1$, $m=2k$ ($k \in \mathbb{N}$) and $\tau=1$, equation \eqref{gru-1} is the constant Webster curvature equation on Heisenberg group $\mathbb{H}=\mathbb{R}\times \mathbb{C}^n $ for solution $u(x, z)$ which is radially symmetric in the variable $z$ . Jerison and Lee \cite{JL1987, JL1988} was able to classify positive solutions with decay at infinity to this equation. See also Garafalo and Vassilev  \cite{GV2001} for further generalization. For $\tau=\frac12$, equation \eqref{gru-1} is also  related to the transonic flow problem, see, for example, Wang \cite{Wang2003}.

Moreover, equation  \eqref{gru-1} is also related to the following weighted Sobolev inequality.
Let $\cal{D}_\tau^1(\mathbb{R}^{n+m})$ be the Hilbert space as the completion of $C_0^\infty (\mathbb{R}^{n+m})$ under the norm
\begin{equation}\label{norm-8-29}
\|u(x, z)\|_{\cal{D}_\tau^1(\mathbb{R}^{n+m})}=(\int_{\mathbb{R}^{n+m}}(|\nabla_z u|^2+(\tau+1)^{2}|z|^{2\tau }|\nabla_x u|^2 )dxdz)^{\frac 12},
\end{equation}
where $x \in \mathbb{R}^n$, $z \in \mathbb{R}^m.$

\begin{proposition}\label{bec-1-gen}    For  $\tau > 0$, there is an optimal  positive constant $S_\tau(n,m)$ such that for all $u(x, z)\in \cal{D}_{\tau}^1(\mathbb{R}^{n+m})$,
\begin{equation}\label{GGN-3-3}
\big(\int_{\mathbb{R}^{n+m}}  |u|^{\frac{2Q}{Q-2}}dxdz\big)^\frac{Q-2}{Q} \le  S^{-1}_\tau(n,m) \int_{\mathbb{R}^{n+m}}(|\nabla_z u|^2+(\tau+1)^{2}|z|^{2\tau }|\nabla_x u|^2 )dxdz,
\end{equation} where $x\in \mathbb{R}^n$, $z\in \mathbb{R}^m$.
\end{proposition}

For $\tau>0$, the above weighted Sobolev inequality \eqref{GGN-3-3} is known for many years. For example, it can be derived  from  a representation formula for Baouendi-Grushin  operator in Franchi, Guti\'{e}rrez and Wheeden \cite{FGW1994}  and a Hardy-Littlewood-Sobolev inequality due to Folland-Stein \cite{FS74}, and is written down precisely in R. Monti and D. Morbidelli \cite[inequality (1.3)]{MM2006}. See, also \cite{FL1984} and \cite{M2006}. Using inequality \eqref{GGN-2}, we will give a self-contained and direct proof of inequality \eqref{GGN-3-3} for function $u(x, z)$ which is radially symmetric in the variable $z$ in Section 6. More precisely, we will prove the following.

 Let $\cal{D}_{\tau,z}^1(\mathbb{R}^{n+m})$ be the Hilbert space as the completion of $\{u\in C_0^\infty (\mathbb{R}^{n+m})\,|\, u~$ $ \text{is~radially symmetric in the variable}~  z \in \mathbb{R}^m \}$ under the norm given by \eqref{norm-8-29}.

\begin{proposition}\label{bec-1}    For   $\tau \ge 0$, there is an optimal  positive constant $S_{\tau, z}(n,m)$ such that for all $u(x, z)\in \cal{D}_{\tau,z}^1(\mathbb{R}^{n+m})$,
\begin{equation}\label{GGN-3}
\big(\int_{\mathbb{R}^{n+m}}  |u|^{\frac{2Q}{Q-2}}dxdz\big)^\frac{Q-2}{Q} \le  S^{-1}_{\tau,z}(n,m) \int_{\mathbb{R}^{n+m}}(|\nabla_z u|^2+(\tau+1)^{2}|z|^{2\tau }|\nabla_x u|^2 )dxdz,
\end{equation} where $x\in \mathbb{R}^n$, $z\in \mathbb{R}^m$.  Moreover, the equality holds for some extremal functions in $\cal{D}_{\tau,z}^1(\mathbb{R}^{n+m})$.

\end{proposition}

On the other hand, it is a long-standing open problem to find the best constant $S_{\tau,z}(n,m)$ and $S_\tau(n,m)$ for $\tau>0$  in the above theorems. For $S_{\tau,z}(n,m)$, positive answer is known only in case $\tau=1$: for $ n\ge 1$, $m\ge 1$ and $\tau=1$, $S_{\tau,z}(n,m)$ and  the classification  of extremal functions in $\cal{D}_{\tau,z}^1(\mathbb{R}^{n+m})$ were essentially obtained in the early work of Jerison and Lee \cite{JL1987} in their study of CR Yamabe problem (for $n=1$, $m$ is even), and by Garofalo and Vassilev \cite[Theorem 1.5]{GV2001}. Garofalo and Vassilev used Jerison and Lee's argument. It seems to us that such an argument  only works for $\tau=1$. See also R. Frank and E. Lieb \cite[Theorem 3.1]{FL2012} for a shorter proof of the Jerison-Lee's theorem on the sharp Sobolev inequality on the Heisenberg group.

\smallskip

Here we will obtain the best constant $S_{\tau, z}(n,m)$ from Theorem \ref{L-1}, and will  classify all positive weak solutions to equation \eqref{gru-1} in $\cal{D}_{\tau,z}^1(\mathbb{R}^{n+m})$, for all $\tau>0$ and  $n, \ m  \ge 1$ except the case of  $m=2$ and $ n= 1$. Thus, depending on whether the extremal functions are given explicitly or implicitly, the sharp constant in inequality  \eqref{GGN-3} for all $\tau>0$ can be explicitly computed or estimated (except the case of $m=2$ and $n=1$). However, as we mentioned above,  the case of $m=2$, $n=1$ and  $\tau=1$  is covered by Jerison-Lee's theorem, as well as by the work of Frank and Lieb.
\begin{definition}\label{weak-z}  $u$ is called a weak solution to equation  \eqref{gru-1} in $\cal{D}^1_{\tau.z}(\mathbb{R}^{n+m})$, if $u \in {\cal D}^1_{\tau.z}(\mathbb{R}^{n+m})$ and for any $\phi\in \cal{D}^1_{\tau.z}(\mathbb{R}^{n+m})$,
	\begin{align*}
	\int_{\mathbb{R}^{n+m}}\big(\nabla_z u\cdot\nabla_z \phi+(\tau+1)^2|z|^{2\tau}\nabla_x u\cdot\nabla_x \phi \big)dxdz=\int_{\mathbb{R}^{n+m}} u^{\frac{Q+2}{Q-2}}\phi dxdz.
	\end{align*}
\end{definition}
\begin{theorem}\label{L-2} Assume that $m\neq 2$ or $m=2,\ n\neq 1$.
	
	\noindent 1). For $\tau=1$,  the equality  in \eqref{GGN-3} holds  up to the multiple of some constant for
all $u(x, z)$ given by
  \begin{equation}\label{sol-gru}
u(x,z)=\big ( \frac{A}{|x-x^o|^2+(|z|^{2}+A)^2} \big )^{\frac{2n+m-2}4},
\end{equation}
where $A>0,\ x^o\in\mathbb{R}^{n}$,
and
\[
S_{1,z}(n,m)=m(2n+m-2)\big[\frac{\pi^{\frac{n+m}{2}}\Gamma(\frac{n+m}{2})}{\Gamma(n+m)}\big]^{\frac{2}{2n+m}}.
\]

Moreover,
  if  $u$  is a positive weak solution to equation \eqref{gru-1} in ${\cal D}_{\tau,z}^{1}(\mathbb{R}^{n+m})$, then up to the multiple of some constant, $u(x,z)$ is given by \eqref{sol-gru}.

\noindent 2). For  $\tau > 0$,  the equality  in \eqref{GGN-3}   holds  up to the multiple of some constant for
all $u(x, z)$ given by
  \begin{equation}\label{sol-gru-1}
u(x, z)=(\frac{1}{|x-x^o|^2+(|z|^{\tau+1}+A)^2})^{\frac{Q-2}{2(\tau+1)}} \psi\big(|\frac{(x-x^o, |z|^{\tau+1}+A)}{|x-x^o|^2+(|z|^{\tau+1}+A)^2}-(0, \frac1{2A})|\big),
\end{equation} where $A>0,x^o\in\mathbb{R}^n,\psi>0$ is the unique solution to \eqref{ode-0}.

Moreover,
  if  $u$  is a positive weak solution to equation \eqref{gru-1}  in ${\cal D}_{\tau,z}^{1}(\mathbb{R}^{n+m})$, then  up to the multiple of some constant, $u(x,z)$ is given by \eqref{sol-gru-1}.
\end{theorem}

Unfortunately, the case of $m=2$ and $n=1$ is left open (the main reason is that:  in this case,  condition  \eqref{constraint} is not satisfied, see Section 6 for more details).

It seems to be standard to show  that all extremal functions in $\cal{D}_{\tau,z}^1(\mathbb{R}^{n+m}) $ to the sharp inequality \eqref{GGN-3}  must be $C^2(\mathbb{R}^{n+m})$ functions which satisfy equation \eqref{gru-1}.  It is certainly the case when $\tau=0$. But for $\tau>0$, we have not found a reference to address this point. We shall come back to discuss the  regularity of weak solutions to equation \eqref{gru-1} in our future study.

For general $\tau>0$, if one can find a solution to the  ODE \eqref{ode-0},  the best constant then can be calculated precisely. See Section 5 and 6 for more details.

\medskip

Finally, we believe that $S_{\tau}(n,m)=S_{\tau, z}(n,m)$.   The main difficulty for finding $S_\tau(n,m)$ seems to be the lack of radially symmetric property for the extremal functions of \eqref{GGN-3-3}. Another approach to prove the equality is to establish the rearrangement in $z$ variable. However,
the argument via the rearrangement in $z$ variable has not be completely carried out by us yet(even we claimed that we proved it in a previous version, but we found a gap in the proof). Using Fourier transformation, Beckner \cite{Bec2001} gave the rearrangement for $\frac{2Q}{Q-2}$ even,  then from the hyperbolic  geometry point of view, he obtained the sharp inequality \eqref{GGN-3-3} for  $n=1, m=1$ or $2$, and $\tau=1$.  Combining Fourier transformation and spherically symmetric decreasing rearrangement, we extend Beckner's argument for $\frac{2Q}{Q-2}$ being an integer  and obtain the following partial results.

\begin{proposition}\label{symmetrization-z}
If
$$
\frac{2Q}{Q-2}=\frac{2m+2n(\tau+1)}{m-2+n(\tau+1)}
$$
is a positive integer, then $S_{\tau}(n,m)=S_{\tau, z}(n,m)$.
\end{proposition}

Observe that for $\tau=1$, if $m=n=1$, or $m=2$, $n=1$, or $m=n=2$, or $m=4$, $n=1$, $2Q/(Q-2)$ is an integer. Thus  by Theroem \ref{L-2} and Proposition \ref{symmetrization-z}, except the case $n=1,\ m=2$, we obtain $S_1(n,m)$ for the above cases.

\medskip

The paper is organized as follows:  We first present a direct proof of Proposition \ref{main-1}  in Section 2. In Section 3, we prove the existence of extremal functions for inequality \eqref{GGN-2}. We show that these extremal functions are H\"older continuous up to the boundary in Section 4. In Section 5, we prove the Liouville theorem (Theorem \ref{L-1}). In Section 6 we derive the results related to Baouendi-Grushin operator. The proofs of some technical lemmas are given in the Appendix.

\section{Generalized Gagliardo-Nirenberg inequality}
In this section, we shall derive the  generalized Gagliardo-Nirenberg inequality (Proposition \ref{main-1}) for any $u\in C^\infty_0(\overline{\mathbb{R}^{n+1}_+})$. We thank H. Brezis for sharing his comment on the history of the popular named Gagliardo-Nirenberg inequality. Since we are not able to verify the details first hand, we stick with the common name (the essential idea first appeared in Gagliardo's paper \cite{Ga59}, and shortly after it appeared in Nirenberg's paper \cite{Ni59}).

We first show that the inequality holds for $l=k-1>-1$ (that is: $k=l+1>0$, the upper bound for $k$).

\begin{lemma}\label{lem:main-p=1} Assume $k>0$ and $u\in C^\infty_0(\overline{\mathbb{R}^{n+1}_+})$, then
\begin{align}\label{eq:main-p=1}
     \int_{\RpN}t^{k-1}|u|dydt\leq C(k)\int_{\RpN}t^{k}|\nabla u|dydt.
\end{align}
\end{lemma}

\begin{proof} Without loss of generality, assume $u\ge0$.
Observe that for $k>0$,
$$
\int_0^\infty t^{k-1} u(y, t)dt= -\frac 1{k}\int_{0}^\infty  \frac {\partial u (y, t) }{\partial t}  \cdot t^{k} dt.
$$
Integrating with respect to $y$ on both sides gives the desired inequality.
\end{proof}

 We then follow the proof for the classical Gagliardo-Nirenberg inequality to establish the inequality for $l=\frac{n+1}nk$ (that is: $k=\frac{nl}{n+1}$, the lower bound for $k$).

 \begin{lemma}\label{eq:main-p>1}   Suppose $k \ge 0$ and $u\in C^\infty_0(\overline{\mathbb{R}^{n+1}_+})$, then
\begin{align}\label{eq:main-p=n/n+1}
   \big(\int_{\RpN}t^{\frac{n+1}{n}k}|u|^{\frac{n+1}{n}}dydt\big)^{\frac{n}{n+1}}\leq C(n,k)\int_{\RpN} t^{k}|\nabla u|dydt.
\end{align}
\end{lemma}

\begin{proof} Without loss of generality, assume $u\ge0$. For $k>0$, integration by parts gives
\begin{equation*}
	\begin{split}
z^ku(y,z)=-\int_{z}^\infty \frac{d}{dt}\big[t^{k}u(y,t)\big]dt=&-\int_{z}^{\infty}\big[kt^{k-1}u(y,t)+t^k\partial_t u(y,t)\big]dt \nonumber\\
\leq & C(k)\int_{0}^\infty t^{k}|\nabla u|(y,t)dt,
\end{split}
\end{equation*}
where we have used Lemma \ref{eq:main-p=1}.  Above inequality obviously holds for $k=0$, so is the following inequality:
for $i=1, \cdots, n$, we have
\[u(y,z)\leq \int_{-\infty}^{+\infty}|\nabla u|(y_1,\cdots, y_{i-1}, x_i, y_{i+1},\cdots y_n,z) dx_i.
\]
Therefore,
\begin{align*}
z^{\frac{k}{n}}|u|^{\frac{n+1}{n}}(y,z)\leq C(k)^{\frac1n}\left(\int_{0}^\infty t^{k}|\nabla u|(y,t)dt\right)^{\frac1n}\prod_{i=1}^n\left(\int_{-\infty}^\infty|\nabla u|dx_i\right)^{\frac1n}.
\end{align*}
Integrating both sides with respect to the measure $z^{k}dydz$ and applying the extended H\"{o}lder's inequality with respect to such a measure yield

\begin{eqnarray*}
& &\int_{\mathbb{R}^{n+1}_+} z^{\frac{n+1}{n}k}|u|^{\frac{n+1}{n}}dydz\\
&\leq& C(k)^{\frac{1}{n}}\int_{\mathbb{R}^{n+1}_+} z^k \big(\int_{0}^{\infty} t^k |\nabla u|dt\big)^{\frac{1}{n}} \prod^{n}_{i=1}\big(\int_{-\infty}^{\infty} |\nabla u|dx_i\big)^{\frac{1}{n}} dydz\\
&=&C(n,k)\int_{\mathbb{R}^{n}}\big[\big(\int_0^\infty t^k|\nabla u|dt\big)^{\frac{1}{n}} \int_0^\infty \prod_{i=1}^n\big( \int_{-\infty}^{\infty}z^k|\nabla u|dx_i\big)^{\frac{1}{n}}dz\big] dy\\
&\leq&C(n,k)\big(\int_{\mathbb{R}^{n}}\int_0^\infty t^k |\nabla u|dtdy\big)^{\frac{1}{n}}\big[\int_{\mathbb{R}^n}\big(\int_0^\infty \prod_{i=1}^n \big(\int_{-\infty}^{\infty}z^k|\nabla u|dx_i\big)^{\frac{1}{n}}dz\big)^{\frac{n}{n-1}}dy\big]^{\frac{n-1}{n}}
\end{eqnarray*}
\begin{eqnarray*}
&\leq&C(n,k)\big(\int_{\mathbb{R}^{n+1}_+} z^k|\nabla u|dydz\big)^{\frac{1}{n}}\big[\int_{\mathbb{R}^n}\prod_{i=1}^n\big(\int_{0}^{\infty}\int_{-\infty}^{\infty}z^k|\nabla u|dx_i dz\big)^{\frac{1}{n-1}}dy\big]^{\frac{n-1}{n}}\\
&=&C(n,k)\big(\int_{\mathbb{R}^{n+1}_+} z^k|\nabla u|dydz\big)^{\frac{1}{n}}\cdot\big[\int_{\mathbb{R}^{n-1}}\big(\Big(\int_{0}^{\infty}\int_{-\infty}^{\infty}z^k|\nabla u|dy_1 dz\Big)^{\frac{1}{n-1}}\\
& &\cdot\int_{-\infty}^{\infty}\prod_{i=2}^n\Big(\int_{0}^{\infty}\int_{-\infty}^{\infty}z^k|\nabla u|dx_i dz\Big)^{\frac{1}{n-1}}dy_1\big)dy_2\cdots dy_n\big]^{\frac{n-1}{n}}\\
&\leq&C(n,k)\big(\int_{\mathbb{R}^{n+1}_+} z^k|\nabla u|dydz\big)^{\frac{2}{n}}\\
& &\cdot\big[\int_{\mathbb{R}^{n-1}}\big(\int_{-\infty}^{\infty}\prod_{i=2}^n\Big(\int_0^\infty\int_{-\infty}^\infty z^k|\nabla u|dx_i dz\Big)^{\frac{1}{n-1}}dy_1\big)^{\frac{n-1}{n-2}}dy_2\cdots dy_n\big]^{\frac{n-2}{n}}\\
&\leq&\cdots\\
&\leq&C(n,k)\big(\int_{\mathbb{R}^{n+1}_+} z^k|\nabla u|dydz\big)^{\frac{j+1}{n}}\cdot\big[\int_{\mathbb{R}^{n-j}}\big(\int_{-\infty}^{\infty}\prod_{i=j+1}^n\Big(\int_0^\infty\int_{-\infty}^\infty \\
& &\cdots\int_{-\infty}^\infty z^k|\nabla u|dy_1\cdots dy_{j-1}dx_i dz\Big)^{\frac{1}{n-j}}dy_j\big)^{\frac{n-j}{n-j-1}}dy_{j+1}\cdots dy_n\big]^{\frac{n-j-1}{n}}\\
	&\leq&C(n,k)\big(\int_{\mathbb{R}^{n+1}_+} z^k|\nabla u|dydz\big)^{\frac{n+1}{n}}.
\end{eqnarray*}
In the second inequality, we write $C(n, k)=C(k)^{\frac{1}{n}}$ for convenience. The proof is completed.
\end{proof}

\smallskip

\noindent{\bf Proof of Proposition \ref{main-1}}.
Let $\theta=\frac{(n+1)k-nl}{n+k}$ and $p=\frac{n+l+1}{n+k}$.  Since $k-1\leq l\leq \frac{n+1}{n}k$, we know  $\theta \in [0,1]$. Besides, $\theta$ satisfies $l=(k-1)\theta+\frac{n+1}{n}k(1-\theta)$ and  $p=\theta+\frac{n+1}{n}(1-\theta)$. For $k>0$, using inequalities \eqref{eq:main-p=1} and \eqref{eq:main-p=n/n+1}, we conclude that
\begin{align*}
\int_{\RpN} t^{l}|u|^pdydt\leq& \big(\int_{\RpN} t^{k-1}|u|dydt\big)^{\theta}\big(\int_{\RpN}t^{\frac{n+1}{n}k}|u|^{\frac{n+1}{n}}dydt\big)^{1-\theta}\\
\leq &C(n,k)\big(\int_{\RpN}t^k|\nabla u|dydt\big)^p.
\end{align*}
\hfill$\Box$

\begin{remark}\label{rem2.1}For $k=l=0$, the proof of Lemma 2.2 is the same as that of the classical Gagliardo-Nirenberg inequality, see, for example, Evans book \cite{Evans}. However, for $l >-1$ and $k=0$, our proof does not work, though  we do  know  inequality \eqref{GGN-1} is still true for $k=0$ from Maz'ya \cite[inequality (2.1.35)] {Maz1985}.
\end{remark}
\begin{remark}\label{rem2.2}  If we write $p= \frac{n+1+l}{n+k}$, we show that condition $l\leq \frac{n+1}{n}k$ (that is: $p \le \frac {n+1}{n}$ ) is necessary.
Suppose that Proposition \ref{main-1} is true for some $k$ and $l$. Then for any $\lambda, \ t_0>0$ satisfying $(1-\lambda^{-1}) t_0 \ge 0$, we consider the rescaled functions $u_{\lambda, t_0}(y,t)=u(\lambda^{-1}y, t_0+\lambda^{-1}(t-t_0))$. We have
\begin{align*}
    \big(\int_{\RpN}t^l|u_{\lambda, t_0}|^pdydt\big)^{\frac1p}&=\lambda^{\frac{l+n+1}{p}}\big(\int_{\mathbb{R}^n\times \{z : z>(1-\lambda^{-1})t_0\}}[z-t_0+\lambda^{-1}t_0]^l|u(y,z)|^pdydz\big)^{\frac1p}\\
    \int_{\RpN}t^{k}|\nabla u_{\lambda, t_0}|dydt&=\lambda^{k+n}\int_{\mathbb{R}^n\times \{z : z>(1-\lambda^{-1})t_0\}}[z-t_0+\lambda^{-1}t_0]^k|\nabla u(y,z)|dydz.
\end{align*}
If we plug $u_{\lambda, t_0}$ to \eqref{GGN-1} and let $t_0\to \infty$, then we must have $l/p\leq k$, which is equivalent to $l\leq \frac{n+1}{n}k$, and indicates that $p \le (n+1)/n.$
\end{remark}

\noindent{\bf Proof of Corollary \ref{cor1-4}}. We only consider $u\in C^\infty_0(\overline{\mathbb{R}^{n+1}_+})$, since the general case can be proven by approximation. We divide the proof into two cases: $\alpha+\beta>0$ and $\alpha+\beta\leq 0$.

\textbf{Case $1$.} $\alpha+\beta>0$. Applying Proposition \ref{main-1} to $u^{\frac{2(n+k)}{n+2k-l-1}}$, where $n+2k-l-1>0$, by H\"older inequality, we have
\begin{align*}
\big(\int_{\mathbb{R}^{n+1}_+}t^l|u|^{\frac{2(n+l+1)}{n+2k-l-1}}dydt\big)^{\frac{n+k}{n+l+1}}
\leq&C\int_{\mathbb{R}^{n+1}_+}t^k |u|^{\frac{n+l+1}{n+2k-l-1}}|\nabla u|dydt\\		\leq&C\big(\int_{\mathbb{R}^{n+1}_+}t^l|u|^{\frac{2(n+l+1)}{n+2k-l-1}}dydt\big)^{\frac{1}{2}}\big(\int_{\mathbb{R}^{n+1}_+}t^{2k-l}|\nabla u|^2 dydt\big)^{\frac{1}{2}},
\end{align*}
then
\begin{equation*}
\big(\int_{\mathbb{R}^{n+1}_+}t^l|u|^{\frac{2(n+l+1)}{n+2k-l-1}}dydt\big)^{\frac{n+2k-l-1}{n+l+1}}\leq C\int_{\mathbb{R}^{n+1}_+}t^{2k-l}|\nabla u|^2 dydt.
\end{equation*}
Taking $\alpha=2k-l>1-n$ and $\beta=l>-1$,  we obtain the desired inequality
\begin{equation}\label{+1}
\big(\int_{\mathbb{R}^{n+1}_+}t^\beta|u|^{\frac{2(n+\beta+1)}{n+\alpha-1}}dydt\big)^{\frac{n+\alpha-1}{n+\beta+1}}\leq C\int_{\mathbb{R}^{n+1}_+}t^{\alpha}|\nabla u|^2 dydt
\end{equation}
for $\alpha+\beta>0$ and $ \frac{n-1}{n+1}\beta\leq \alpha\leq\beta+2$. Here, we also have $\alpha>0$.

\textbf{Case $2$.} $\alpha+\beta\leq 0$. Assume that $\alpha,\ \beta$ satisfy $\alpha>0,\ \beta>-1,\  \frac{n-1}{n+1}\beta\leq \alpha\leq\beta+2$ and $\alpha+\beta\leq 0$, then we have $0<\alpha<1,\ -1<\beta\leq 0$. For some $k>0$ and $l=\beta$, applying Proposition \ref{main-1} to $u^{\frac{2(n+k)}{n+\alpha-1}}$,  and using H\"{o}lder inequality, we have
\begin{align}\label{+2}
&\big(\int_{\mathbb{R}^{n+1}_+}t^\beta|u|^{\frac{2(n+\beta+1)}{n+\alpha-1}}dydt\big)^{\frac{n+k}{n+\beta+1}}\nonumber\\
\leq &C\int_{\mathbb{R}^{n+1}_+}t^k |u|^{\frac{n+2k-\alpha+1}{n+\alpha-1}}|\nabla u|dydt\nonumber\\
\leq& C\big(\int_{\mathbb{R}^{n+1}_+}t^{\alpha}|\nabla u|^2 dydt\big)^{\frac{1}{2}}\big(\int_{\mathbb{R}^{n+1}_+}t^{2k-\alpha}| u|^\frac{2(n+2k-\alpha+1)}{n+\alpha-1} dydt\big)^{\frac{1}{2}},
\end{align}
where $k$ needs to satisfy
\begin{align*}
k>0,\ \frac{n}{n+1}\beta\leq k\leq \beta+1.
\end{align*}
Since $-1<\beta\leq 0$, $k$ only needs to  satisfy $0<k\leq \beta+1$. Replacing $\beta$ with $2k-\alpha$ in \eqref{+1}, we have \begin{align}\label{+3}
\big(\int_{\mathbb{R}^{n+1}_+}t^{2k-\alpha}|u|^{\frac{2(n+2k-\alpha+1)}{n+\alpha-1}}dydt\big)^{\frac{n+\alpha-1}{n+2k-\alpha+1}}\leq C\int_{\mathbb{R}^{n+1}_+}t^{\alpha}|\nabla u|^2 dydt,
\end{align}
where $k$ needs to satisfy
\begin{align*}
2k-\alpha>-1,\ k>0,\ \frac{n-1}{n+1}(2k-\alpha)\leq \alpha\leq (2k-\alpha)+2.
\end{align*}
Since $0<\alpha<1$, $k$ only needs to satisfy $n\geq 2,\ 0<k\leq \frac{n}{n-1}\alpha$ or $n=1,\ k>0$. Therefore, choosing $k$ to satisfy $n\geq 2,\ 0<k\leq\min \{\frac{n}{n-1}\alpha, \ \beta+1\}$ or $n=1,\ 0<k\leq \beta+1$, and taking \eqref{+3} back to \eqref{+2}, we get the desired inequality.

In conclusion, inequality \eqref{GGN-2} holds for $ \alpha>0,\ \beta>-1,\ \frac{n-1}{n+1}\beta\leq \alpha\leq \beta+2$.
\hfill$\Box$

\section{Existence of extremal functions}
 In this section, we prove the existence of extremal functions to the sharp form of \eqref{GGN-2} by the  concentration-compactness principle.
The case $\alpha=\frac{n-1}{n+1}\beta$ is more complicated. See Remark \ref{special beta}.

Recall that the weighted Sobolev space ${\cal D}_\alpha^{1,2}(\mathbb{R}^{n+1}_+)$ is  defined as the completion of the space $C^\infty_0(\overline{\mathbb{R}^{n+1}_+})$ endowed with the norm
\[
\|u\|_{{\cal D}_\alpha^{1,2}(\mathbb{R}^{n+1}_+)}=\big(\int_{\mathbb{R}^{n+1}_+}t^{\alpha}|\nabla u|^2 dy dt\big)^\frac12.
\]
And for $1\le p<\infty,\ \beta>-1$, we define
$$L^{p}_\beta(\mathbb{R}^{n+1}_+)=\{u\,:\,\mathbb{R}_+^{n+1}\to \mathbb{R}\,|\,\|u\|^p_{L^{p}_\beta(\mathbb{R}^{n+1}_+)}=\int_{\mathbb{R}^{n+1}_+ }t^\beta |u|^{p}dydt<\infty\},$$
$$L^{p}_{\beta,loc}(\overline{\mathbb{R}^{n+1}_+})=\{u\,:\,\mathbb{R}_+^{n+1}\to \mathbb{R}\,|\,\int_{K }t^\beta |u|^{p}dydt<\infty,\, \forall K\subset\subset\overline{\mathbb{R}^{n+1}_+}\}.$$
 Define $ B_R(x)=\{z\in\mathbb{R}^{n+1}\,|\, |z-x|<R\}$ and $B_R^+(x)=B_R(x)\cap \mathbb{R}^{n+1}_+$.
We denote by $\mathcal{M}(\mathbb{R}^{n+1}_+ )$ the space of positive, bounded measures in $\overline{\mathbb{R}^{n+1}_+} $.
The sharp constant  inequality \eqref{GGN-2}  can also be classified by
\begin{equation*}
{S_{n+1, \alpha,\beta}}:= \inf \{\int_{\mathbb{R}^{n+1}_+}t^{\alpha}|\nabla u|^2 dy dt\,:\, u\in {\cal D}_\alpha^{1,2}(\mathbb{R}^{n+1}_+),\ \|u\|_{L^{p^*}_{\beta}(\mathbb{R}^{n+1}_+)}=1\}.
\end{equation*}
The aim of this section is to show  that $S_{n+1, \alpha,\beta}$ is attained by some functions. For $\lambda>0$ and $z\in\mathbb{R}^{n}$, define
\[
u^{\lambda,z}(y,t)=\lambda^{\frac{n+\alpha-1}2}u(\lambda y+z,\lambda t).
\]
It is easy to verify that
\begin{eqnarray*}
\|u^{\lambda,z}\|_{L^{p^*}_\beta(\mathbb{R}^{n+1}_+)}=\|u\|_{L^{p^*}_\beta(\mathbb{R}^{n+1}_+)}\text{ and}\quad
\|u^{\lambda,z}\|_{{\cal D}_\alpha^{1,2}(\mathbb{R}^{n+1}_+)}=\|u\|_{{\cal D}_\alpha^{1,2}(\mathbb{R}^{n+1}_+)}.
\end{eqnarray*}

\begin{proposition}\label{Min}
Assume $n\ge 1$, and $\alpha, \beta$ satisfy \eqref{beta-2}. Let $\{u_j\}$ be a minimizing sequence of functions for $S_{n+1, \alpha,\beta}$ with $\|u_j\|_{L^{p^*}_{\beta}(\mathbb{R}^{n+1}_+)}=1$, then after passing  to a subsequence, there exists $\lambda_j>0$ and $z_j\in \mathbb{R}^{n}$ such that $u^{\lambda_j,z_j}_j\to u$ in $L_\beta^{p^*}(\mathbb{R}^{n+1}_+ )$. In particular, there exists at least one nonnegative minimizer for  $S_{n+1, \alpha,\beta}$.
\end{proposition}

Apparently, Theorem \ref{sharp-C-2} follows from this proposition immediately. To prove Proposition \ref{Min},
we first  establish the concentration-compactness principle as the procedures in \cite{Willem97} similar to that in  P.L. Lions \cite{Lions1985a,Lions1985b}.
\begin{lemma}\label{CCP-1} Assume $n\ge 1$, and $\alpha, \beta$ satisfy \eqref{beta-3}.  Let $\{u_j\}$ be a bounded sequence  in ${\cal D}_\alpha^{1,2}(\mathbb{R}^{n+1}_+ )$, $\mu,\nu$ be two Radon measures and a function $u\in {\cal D}^{1,2}_\alpha(\mathbb{R}^{n+1}_+ )$, such that\\
$(1)$\ $u_j\rightharpoonup u$ weakly in  ${\cal D}^{1,2}_\alpha(\mathbb{R}^{n+1}_+ )$,\\
$(2)$\ $u_j\rightarrow u$ a.e. in $\overline{\mathbb{R}^{n+1}_+}$,\\
$(3)$\ $\nu_j=t^\beta |u_j-u|^{p^*}dydt\rightharpoonup \nu$ weakly in $\mathcal{M}(\mathbb{R}^{n+1}_+ )$,\\
$(4)$\ $\mu_j=t^{\alpha}|\nabla (u_j-u)|^2 dy dt\rightharpoonup \mu$ weakly in $\mathcal{M}(\mathbb{R}^{n+1}_+ )$.\\
Define
\begin{eqnarray*}
\mu_\infty&=&\lim_{R\to\infty}\varlimsup_{j\to\infty}\int_{\mathbb{R}^{n+1}_+\backslash B_R(0)}t^{\alpha}|\nabla u_j|^2 dy dt,\\
\nu_\infty&=&\lim_{R\to\infty}\varlimsup_{j\to\infty}\int_{\mathbb{R}^{n+1}_+\backslash B_R(0)}t^\beta |u_j|^{p^*}dydt.
\end{eqnarray*}
Then,
\begin{eqnarray*}
&(i)& \|\mu\|\geq S_{n+1, \alpha,\beta}\|\nu\|^{\frac{2}{p^*}},\\
&(ii)& \mu_\infty\ge S_{n+1, \alpha,\beta}\nu_\infty^\frac2{p^*},\\
&(iii)& \varlimsup_{j\rightarrow\infty}\int_{\mathbb{R}^{n+1}_+}t^\alpha |\nabla u_j|^2dydt=\int_{\mathbb{R}^{n+1}_+}t^\alpha |\nabla u|^2dydt+\|\mu\|+\mu_\infty,\\
&(iv)& \varlimsup_{j\rightarrow\infty}\int_{\mathbb{R}^{n+1}_+}t^\beta | u_j|^{p^*}dydt=\int_{\mathbb{R}^{n+1}_+}t^\beta |u|^{p^*}dydt+\|\nu\|+\nu_\infty,
\end{eqnarray*}
where $\|\mu\|=\sup_{u\in C(\overline{\mathbb{R}^{n+1}_+}), \|u\|_{L^\infty}=1}<\mu,u>$.
Moreover, if $u=0$ and $\|\mu\|=S_{n+1, \alpha,\beta}\|\nu\|^{\frac{2}{p^*}}$, then $\mu$ and $\nu$ are concentrated at a single point.
\end{lemma}
\begin{proof} 1). Assume first $u=0$, i.e. $u_j\rightharpoonup 0$ weakly in ${\cal D}^{1,2}_\alpha(\mathbb{R}^{n+1}_+ )$, and we prove $(i)$ and $(ii)$.
	
	1.1). For any $\varphi\in C^\infty_0(\overline{ \mathbb{R}^{n+1}_+})$, by inequality \eqref{GGN-2}, we have
\begin{align}\label{Ex-1}
S_{n+1, \alpha,\beta}\big(\int_{\mathbb{R}^{n+1}_+ }t^\beta |\varphi u_j|^{p^*}dydt \big)^{\frac{2}{p^*}}\le  \int_{\mathbb{R}^{n+1}_+ }t^{\alpha}|\nabla (\varphi u_j)|^2 dy dt,
\end{align}
where \begin{eqnarray}\label{Ex-2}
RHS
&\le&\int_{\mathbb{R}^{n+1}_+ }t^{\alpha}|\varphi|^2|\nabla u_j|^2dydt+2\int_{\mathbb{R}^{n+1}_+}t^\alpha|\varphi| | u_j||\nabla \varphi||\nabla u_j|dydt\nonumber\\
&&+\int_{\mathbb{R}^{n+1}_+ }t^{\alpha}|u_j|^2|\nabla\varphi|^2 dy dt.
\end{eqnarray}
According to Lemma \ref{density} in Appendix, it is easy to verify that $\cal{D}^{1,2}_\alpha(\mathbb{R}^{n+1}_+)\subset {\cal D}^{1,2}_{\alpha,loc}(\mathbb{R}^{n+1}_+)$. Then by the compact embedding lemma (Lemma \ref{CompE}), we have $u_j\rightarrow 0$ in $L_{\alpha, loc}^2(\overline{\mathbb{R}^{n+1}_+})$. Therefore in \eqref{Ex-2}, as $j\to \infty,$
\begin{align*}
	\int_{\mathbb{R}^{n+1}_+}t^\alpha|\varphi| | u_j||\nabla \varphi||\nabla u_j|dydt
\leq C(\varphi)
\big(\int_{\text{supp} \varphi}t^\alpha|\nabla u_j|^2dydt\big)^\frac{1}{2}\big(\int_{\text{supp} \varphi}t^\alpha|u_j|^2dydt\big)^\frac{1}{2}
\to 0,
\end{align*}
and
\begin{eqnarray*}
\int_{\mathbb{R}^{n+1}_+ }t^{\alpha}|u_j|^2|\nabla\varphi|^2 dy dt\leq C(\varphi)\int_{\text{supp} \varphi}t^\alpha |u_j|^2dydt\rightarrow 0.
\end{eqnarray*}
Back to \eqref{Ex-1}, and letting $j\to \infty$, we arrive at
\begin{equation}\label{compa}
S_{n+1, \alpha,\beta}\big(\int_{\mathbb{R}^{n+1}_+ }|\varphi|^{p^*}d\nu\big)^\frac{2}{p^*} \le  \int_{\mathbb{R}^{n+1}_+ }|\varphi|^2 d\mu.
\end{equation}
A limit process shows
\begin{equation}\label{Ex-4}
 S_{n+1, \alpha,\beta}\nu(E)^{\frac2{p^*}}\le\mu(E),
\end{equation}
for any  bounded Borel set $E\subset\overline{\mathbb{R}^{n+1}_+}$, which implies $\|\mu\|\geq S_{n+1, \alpha,\beta}\|\nu\|^{\frac{2}{p^*}}$.

\medskip

1.2).  For any $R>1$, choose cut-off function $\psi_R\in C^1(\overline{ \mathbb{R}^{n+1}_+})$, such that $0\leq \psi_R\leq 1$, $\psi_R(y,t)=1$ for $|(y,t)|\geq R+1$, $\psi_R(y,t)=0$ for $|(y,t)|\leq R$ and $|\nabla \psi_R|\leq C$. By inequality \eqref{GGN-2}, we have
$$\big(\int_{\mathbb{R}^{n+1}_+} t^\beta|\psi_R u_j|^{p^*}dydt\big)^\frac{2}{p^*}\leq S_{n+1, \alpha,\beta}^{-1}\int_{\mathbb{R}^{n+1}_+} t^\alpha |\nabla(\psi_R u_j)|^2dydt.$$
Similar to the argument in 1.1), we have
\begin{align}\label{mu_infty}
\varlimsup_{j\rightarrow\infty}\big(\int_{\mathbb{R}^{n+1}_+} t^\beta|\psi_R u_j|^{p^*}dydt\big)^\frac{2}{p^*}\leq S_{n+1, \alpha,\beta}^{-1}\varlimsup_{j\rightarrow\infty}\int_{\mathbb{R}^{n+1}_+} t^\alpha \psi_R^2|\nabla u_j|^2dydt.
\end{align}
On the other hand,
$$\int_{\mathbb{R}^{n+1}_+\backslash B_{R+1}(0)}t^\alpha |\nabla u_j|^2dydt\leq \int_{\mathbb{R}^{n+1}_+}t^\alpha \psi_R^2|\nabla u_j|^2dydt\leq \int_{\mathbb{R}^{n+1}_+\backslash B_R(0)}t^\alpha |\nabla u_j|^2dydt,$$
$$\int_{\mathbb{R}^{n+1}_+\backslash B_{R+1}(0)}t^\beta|u_j|^{p^*}dydt\leq\int_{\mathbb{R}^{n+1}_+}t^\beta|\psi_R u_j|^{p^*}dydt\leq \int_{\mathbb{R}^{n+1}_+\backslash B_{R}(0)}t^\beta|u_j|^{p^*}dydt.$$
Then by the definition of $\mu_\infty,\ \nu_\infty$ and \eqref{mu_infty}, it holds
$$\mu_\infty\ge S_{n+1,\alpha,\beta}\nu_\infty^\frac2{p^*}.$$

1.3). Further, we assume $\|\mu\|=S_{n+1, \alpha,\beta}\|\nu\|^\frac{2}{p^*}$. By \eqref{compa} and H\"older inequality, for any $\varphi\in C^\infty_0(\overline{ \mathbb{R}^{n+1}_+})$,
$$\int_{\mathbb{R}^{n+1}_+ }|\varphi|^{p^*}d\nu\le  S_{n+1, \alpha,\beta}^{-\frac{p^*}{2}}\big(\int_{\mathbb{R}^{n+1}_+ }|\varphi|^2 d\mu\big)^\frac{p^*}{2}\leq S_{n+1, \alpha,\beta}^{-\frac{p^*}{2}}\|\mu\|^{\frac{p^*-2}{2}}\int_{\mathbb{R}^{n+1}_+ }|\varphi|^{p^*} d\mu,$$
then we have $$\nu=S_{n+1, \alpha,\beta}^{-\frac{p^*}{2}}\|\mu\|^{\frac{p^*-2}{2}}\mu.$$
This means
\begin{align*}
	\big(\int_{\mathbb{R}^{n+1}_+ }|\varphi|^{p^*}d\nu\big)^{\frac{1}{p^*}} \leq S_{n+1, \alpha,\beta}^{-\frac{1}{2}}\big(\int_{\mathbb{R}^{n+1}_+}\varphi^2 d\mu\big)^\frac{1}{2}
	=\|\nu\|^{-\frac{p^*-2}{2p^*}}\big(\int_{\mathbb{R}^{n+1}_+}\varphi^2 d\nu\big)^\frac{1}{2}.
\end{align*}
Then for any open set $\Omega$ in $\mathbb{R}^{n+1}$, $$\nu(\Omega\cap\overline{\mathbb{R}^{n+1}_+})^{\frac{1}{p^*}}\leq\nu(\overline{\mathbb{R}^{n+1}_+})^{-\frac{p^*-2}{2p^*}}\nu(\Omega\cap\overline{\mathbb{R}^{n+1}_+})^\frac{1}{2}.$$
 Since $\alpha<\beta+2,$ we have that $p^*>2$. If $\nu(\Omega\cap\overline{\mathbb{R}^{n+1}_+})>0$, we have $\nu(\overline{\mathbb{R}^{n+1}_+})\leq \nu(\Omega\cap\overline{\mathbb{R}^{n+1}_+})$, 
 which implies that $\nu$ is centered at a single point, so is $\mu$.

2).  We discuss the general case. Write $v_j=u_j-u$. Since $v_j\rightharpoonup 0$ weakly in ${\cal D}^{1,2}_\alpha(\mathbb{R}^{n+1}_+)$, we have for any $h\in C^\infty_0(\overline{ \mathbb{R}^{n+1}_+})$,
\begin{eqnarray*}
&&	\int_{\mathbb{R}^{n+1}_+}t^\alpha|\nabla u_j|^2 hdydt\\
	&=&	\int_{\mathbb{R}^{n+1}_+}t^\alpha|\nabla v_j|^2 hdydt+2\int_{\mathbb{R}^{n+1}_+}t^\alpha\nabla v_j\nabla u hdydt+\int_{\mathbb{R}^{n+1}_+}t^\alpha|\nabla u|^2 hdydt\\
	&\rightarrow&\int_{\mathbb{R}^{n+1}_+}hd\mu+\int_{\mathbb{R}^{n+1}_+}t^\alpha|\nabla u|^2 hdydt.
\end{eqnarray*}
Then we obtain that
$$t^\alpha|\nabla u_j|^2dydt\rightharpoonup \mu+|\nabla u|^2dydt\ \text{ weakly in } \mathcal{M}(\mathbb{R}^{n+1}_+).$$
According to Brezis-Lieb Lemma, we have for every nonnegative $h\in C^\infty_0(\overline{\mathbb{R}^{n+1}_+})$,
$$\int_{\mathbb{R}^{n+1}_+}t^\beta |u|^{p^*}hdydt=\lim_{j\rightarrow\infty}\big(\int_{\mathbb{R}^{n+1}_+}t^\beta |u_j|^{p^*}hdydt-\int_{\mathbb{R}^{n+1}_+}t^\beta |v_j|^{p^*}hdydt\big).$$
Hence we obtain that
\[
t^\beta|u_j|^{p^*}dydt\rightharpoonup \nu+t^\beta|u|^{p^*}dydt\ \text{ weakly in }\mathcal{M}(\mathbb{R}^{n+1}_+).
\]
Part $(i)$ follows from the corresponding inequality for $\{v_j\}$.

Since
\begin{eqnarray*}
	&&\varlimsup_{j\rightarrow\infty}\int_{\mathbb{R}^{n+1}_+\backslash B_R(0)}t^\alpha|\nabla v_j|^2dydt\\
	&=&\varlimsup_{j\rightarrow\infty}\int_{\mathbb{R}^{n+1}_+\backslash B_R(0)}t^\alpha|\nabla u_j|^2dydt-\int_{\mathbb{R}^{n+1}_+\backslash B_R(0)}t^\alpha|\nabla u|^2dydt,
	\end{eqnarray*}
we obtain that
$$\lim_{R\rightarrow\infty}\varlimsup_{j\rightarrow\infty}\int_{\mathbb{R}^{n+1}_+\backslash B_R(0)}t^\alpha|\nabla v_j|^2dydt=\lim_{R\rightarrow\infty}\varlimsup_{j\rightarrow\infty}\int_{\mathbb{R}^{n+1}_+\backslash B_R(0)}t^\alpha|\nabla u_j|^2dydt=\mu_\infty.$$
By  Brezis-Lieb Lemma, we have
$$\lim_{j\rightarrow\infty}\big(\int_{\mathbb{R}^{n+1}_+\backslash B_R(0)}t^\beta |u_j|^{p^*}dydt-\int_{\mathbb{R}^{n+1}_+\backslash B_R(0)}t^\beta |v_j|^{p^*}dydt\big)=\int_{\mathbb{R}^{n+1}_+\backslash B_R(0)}t^\beta |u|^{p^*}dydt,$$
which implies
$$\lim_{R\rightarrow\infty}\varlimsup_{j\rightarrow\infty}\int_{\mathbb{R}^{n+1}_+\backslash B_R(0)}t^\beta |v_j|^{p^*}dydt=\nu_\infty.$$
Part $(ii)$ follows from the corresponding inequality for $\{v_j\}$.

\medskip

Next, we prove $(iii) $ and $(iv)$. For every $R>1$, we have
\begin{eqnarray*}
	& &\varlimsup_{j\rightarrow\infty}\int_{\mathbb{R}^{n+1}_+}t^\alpha|\nabla u_j|^2dydt\\
	&=&\varlimsup_{j\rightarrow\infty}\int_{\mathbb{R}^{n+1}_+}t^\alpha\psi_R|\nabla u_j|^2dydt+\varlimsup_{j\rightarrow\infty}\int_{\mathbb{R}^{n+1}_+}t^\alpha(1-\psi_R)|\nabla u_j|^2dydt\\
	&=&\varlimsup_{j\rightarrow\infty}\int_{\mathbb{R}^{n+1}_+}t^\alpha\psi_R|\nabla u_j|^2dydt+\int_{\mathbb{R}^{n+1}_+}(1-\psi_R)d\mu+\int_{\mathbb{R}^{n+1}_+}t^\alpha(1-\psi_R)|\nabla u|^2dydt.
\end{eqnarray*}
When $R\rightarrow\infty$, we get,  by Lebesgue dominated convergence theorem, that
$$\varlimsup_{j\rightarrow\infty}\int_{\mathbb{R}^{n+1}_+}t^\alpha|\nabla u_j|^2dydt=\mu_\infty+\|\mu\|+\int_{\mathbb{R}^{n+1}_+}t^\alpha|\nabla u|^2dydt.$$
Similarly, we can get
$$\varlimsup_{j\rightarrow\infty}\int_{\mathbb{R}^{n+1}_+}t^\beta| u_j|^{p^*}dydt=\nu_\infty+\|\nu\|+\int_{\mathbb{R}^{n+1}_+}t^\beta| u|^{p^*}dydt.$$
Lemma \ref{CCP-1}  is proved.
\end{proof}

\medskip

\textbf{Proof of Proposition \ref{Min}.}  Let $\{u_j\}\subset{\cal  D}^{1,2}_\alpha(\mathbb{R}^{n+1}_+)$ be a nonnegative minimizing sequence of functions for $S_{n+1, \alpha,\beta}$ with  $\int_{\mathbb{R}^{n+1}_+} t^\beta|u_j|^{p^*}dydt =1$.
Define $$Q_j(\lambda)=\sup_{z\in\mathbb{R}^{n}}\int_{B_\lambda^+((z,0))}t^\beta |u_j|^{p^*}dydt.$$
Since for every $j$,
$$\lim_{\lambda\rightarrow 0^+}Q_j(\lambda)=0,\quad\lim_{\lambda\rightarrow \infty}Q_j(\lambda)=1,$$
there exists $\lambda_j>0$ such that $Q_j(\lambda_j)=\frac{1}{2}$. Moreover, there exists $z_j\in\mathbb{R}^{n}$ such that
$$\int_{B^+_{\lambda_j}((z_j,0))}t^\beta |u_j|^{p^*}dydt=Q_j(\lambda_j)= \frac{1}{2},$$
since $$ \lim_{|z|\rightarrow\infty}\int_{B^+_{\lambda_j}((z,0))}t^\beta|u_j|^{p^*}dydt=0.$$
Let $w_j=u^{\lambda_j,z_j}_j$, then $w_j$ satisfies $\int_{\mathbb{R}^{n+1}_+} t^\beta|w_j|^{p^*}dydt =1$, $\lim_{j\to\infty}\|w_j\|^2_{{\cal D}^{1,2}_\alpha(\mathbb{R}^{n+1}_+)}=S_{n+1,\alpha,\beta}$, and
\begin{equation}\label{Ex-m-1}
\frac12=\int_{B^+_{1}(0)}t^\beta|w_j|^{p^*}dydt=\sup_{z\in\mathbb{R}^{n}}\int_{B^+_{1}((z,0))}t^\beta|w_j|^{p^*}dydt.
\end{equation}
Since ${\cal D}^{1,2}_{\alpha}(\mathbb{R}^{n+1}_+)\subset{\cal D}^{1,2}_{\alpha, loc}(\mathbb{R}^{n+1}_+)$, by compact embedding lemma (Lemma \ref{CompE}),
we have, after  passing to a subsequence, that
\begin{eqnarray*}
	& &w_j\rightharpoonup w\quad\text{in}~{\cal D}^{1,2}_{\alpha}(\mathbb{R}^{n+1}_+),\\
	& &w_j\rightarrow w\quad\text{in}~ L^{2}_{\alpha, loc}(\overline{\mathbb{R}^{n+1}_+} ),\\
	& &w_j\rightarrow w\quad \text{a.e. in } \overline{ \mathbb{R}^{n+1}_+},\\
	& &\nu_j=t^\beta|w_j-w|^{p^*}dydt\rightharpoonup \nu \quad\text{weakly in }~\mathcal{M}(\mathbb{R}^{n+1}_+ ),\\
	& &\mu_j=t^{\alpha}|\nabla (w_j-w)|^2 dy dt\rightharpoonup \mu\quad\text{weakly in}~ \mathcal{M}(\mathbb{R}^{n+1}_+ ).
\end{eqnarray*}
 From Lemma \ref{CCP-1}, we have
\begin{align}\label{4.3-1}
1=\lim_{j\rightarrow\infty}\int_{\mathbb{R}^{n+1}_+ }t^\beta|w_j|^{p^*}dydt
=\int_{\mathbb{R}^{n+1}_+ }t^\beta|w|^{p^*}dydt+ \|\nu\|+\nu_\infty,
\end{align}
\begin{align}\label{4.3-2}
S_{n+1,\alpha,\beta}=\lim_{j\rightarrow\infty}\int_{\mathbb{R}^{n+1}_+ }t^\alpha|\nabla w_j|^2dydt
=\int_{\mathbb{R}^{n+1}_+}t^{\alpha}|\nabla w|^2dydt+\|\mu\|+\mu_\infty,
\end{align}
\begin{align}\label{4.3-3}
\|\mu\|\geq S_{n+1,\alpha,\beta}\|\nu\|^{\frac{2}{p^*}},\
\mu_\infty\ge S_{n+1,\alpha,\beta}\nu_\infty^\frac2{p^*},
\end{align}
where
\begin{eqnarray*}
	\mu_\infty&=&\lim_{R\to\infty}\varlimsup_{m\to\infty}\int_{\mathbb{R}^{n+1}_+\backslash B_R(0)}t^{\alpha}|\nabla w_j|^2 dy dt,\\
	\nu_\infty&=&\lim_{R\to\infty}\varlimsup_{m\to\infty}\int_{\mathbb{R}^{n+1}_+\backslash B_R(0)}t^\beta |w_j|^{p^*}dydt.
\end{eqnarray*}
Combining \eqref{GGN-2}, \eqref{4.3-1}, \eqref{4.3-2}, \eqref{4.3-3} and $p^*>2$, we have
\begin{eqnarray*}
	S_{n+1,\alpha,\beta}
	&\ge&S_{n+1,\alpha,\beta}\big[\big(\int_{\mathbb{R}^{n+1}_+ }t^\beta|w|^{p^*}dydt\big)^\frac2{p^*}+\|\nu\|^\frac2{p^*}+\nu_\infty^\frac2{p^*}\big]\nonumber\\
	&\ge&S_{n+1,\alpha,\beta}\big(\int_{\mathbb{R}^{n+1}_+ }t^\beta|w|^{p^*}dydt+\|\nu\|+\nu_\infty\big)^\frac2{p^*}\nonumber\\
	&=&S_{n+1,\alpha,\beta},
\end{eqnarray*}
which implies
\begin{eqnarray}\label{Ex-m-2}
\big(\|w\|^{p^*}_{L_\beta^{p^*}(\mathbb{R}^{n+1}_+ )}\big)^\frac2{p^*}+\|\nu\|^\frac2{p^*}+\nu_\infty^\frac2{p^*}=\big(\|w\|^{p^*}_{L_\beta^{p^*}(\mathbb{R}^{n+1}_+ )}+\|\nu\|+\nu_\infty\big)^{\frac{2}{p^*}}=1.
\end{eqnarray}
 Since $2/p^*<1$, above equality indicates that only one term is equal to $1$ and the others must be $0$.
 By \eqref{Ex-m-1}, $\nu_\infty\leq\frac12$, then $\nu_\infty=0$.
 If $\|\nu\|=1$, then $w=0$ and $\|\mu\|=S_{n+1, \alpha,\beta}\|\nu\|^{\frac{2}{p^*}}$. By the last statement in Lemma \ref{CCP-1}, we have that $\mu$ and $\nu$ are concentrated on a single point $x^*$. We claim $x^*=(z^*,0)\in\partial\mathbb{R}^{n+1}_+$, then by \eqref{Ex-m-1},
 $$\frac12\geq \int_{B^+_{1}(x^*)}t^\beta|w_j|^{p^*}dydt\rightarrow\|\nu\|=1,$$
 contradiction. It follows that $
 \nu=0, \|w\|_{L^{p^*}_\beta(\mathbb{R}^{n+1}_+)}=1$, which implies that $w$ is the extremal function for $S_{n+1,\alpha,\beta}$.

 To prove the claim, we argue by contradiction. Assume  $x^*=(z^*,t^*)$ for some $t^*>0$. For $n>1$, since  $\beta<\frac{n+1}{n-1}\alpha$, we know $p^*<2^*:=\frac{2(n+1)}{n-1}$. For every $0<\varepsilon<t^*$, we have $$\lim_{j\rightarrow\infty}\int_{B_\varepsilon(x^*)}t^\beta|u_j|^{p^*}dydt=1.$$ But for $0<\varepsilon<\frac{t^*}{2}$, by H\"{o}lder inequality,
 \begin{eqnarray*}
  \big( \int_{B_\varepsilon(x^*)}t^\beta|w_j|^{p^*}dydt\big)^{\frac{2}{p^*}}
    \leq C\big(\int_{B_\varepsilon(x^*)}|w_j|^{2^*}dydt\big)^{\frac{2}{2^*}}\varepsilon^{\frac{2(2^*-p^*)}{2^*p^*}(n+1)},
 \end{eqnarray*}
where \begin{align*}
\big(\int_{B_\varepsilon(x^*)}|w_j|^{2^*}dydt\big)^{\frac{2}{2^*}}&\leq C\int_{B_\varepsilon(x^*)}(|\nabla w_j|^2+|w_j|^2)dydt\\
&\leq C\int_{B_\varepsilon(x^*)}t^\alpha(|\nabla w_j|^2+|w_j|^2)dydt\leq C.
\end{align*}
 Therefore,  for $\varepsilon$ small enough, $\lim_{j\to \infty}\big( \int_{B_\varepsilon(x^*)}t^\beta|w_j|^{p^*}dydt\big)^{\frac{2}{p^*}}<1$, contradiction. For $n=1$, we replace $2^*$ by a power $q>p^*$ in the above calculation. Similarly, we can get the same contradiction.

 Finally, replacing $w$ by $|w|$, we get the nonnegative extremal function for  $S_{n+1,\alpha,\beta}$.
 \hfill$\Box$

\begin{remark}\label{special beta}
	 Assume $n\geq 2$. For $\alpha=\beta=0$, by the classical Sobolev inequality, $S_{n+1,0,0}$ can be attained by  functions taking the form of $\big(\frac{\lambda}{\lambda^2+|y-y^o|^2+t^2}\big)^{\frac{n-1}{2}}$, but for general $\beta=\frac{n+1}{n-1}\alpha$, the minimizer may not exist. For example, we know from Lemma \ref{appro}: for $\alpha\ge1$, $C^\infty_0(\mathbb{R}^{n+1}_+)$ is dense in ${\cal D}^{1,2}_\alpha(\mathbb{R}^{n+1}_+)$.  Let $f(y,t)=t^\frac{\alpha}{2}u(y,t)$, then for $\beta=\frac{n+1}{n-1}\alpha$, \eqref{sharp-C-1} is equivalent to
	\begin{equation}\label{replacement}
	S=\inf_{f\in C^\infty_0(\mathbb{R}^{n+1}_+)\backslash\{0\}}\frac{\int_{\mathbb{R}^{n+1}_+}(|\nabla f|^2 -\frac{\alpha(2-\alpha)}{4}\frac{f^2}{t^2})dydt}{(\int_{\mathbb{R}^{n+1}_+} |f|^{\frac{2(n+1)}{n-1}}dydt)^{\frac{n-1}{n+1}}}.
	\end{equation}
	Observe: for $\alpha=2$,  $S$ is exact the best Sobolev constant in $\mathbb{R}^{n+1}$, thus cannot be attained by any functions in the closure of $C_0^{\infty}(\mathbb{R}^{n+1}_+)$ under norm $\|f\|=(\int_{\mathbb{R}^{n+1}_+}|\nabla f|^2dydt)^{\frac12}$. This indicates that there is no extremal function for $S_{1,2,\frac{2(n+1)}{n-1}}$. Besides, for $\alpha\geq 2$ and $\beta=\frac{n+1}{n-1}\alpha$, the nonexistence of extremal function for $S_{n+1,\alpha,\beta}$ can be obtained by Mancini, Sandeep\cite[Theorem 1.5]{Mancini}. See more details in Remark \ref{no-solu}.
\end{remark}

\section{Regularity of extremal functions}

Throughout this section,  we always assume
\begin{align}\label{eq:constr-a-b}
 \alpha>0,\quad \beta>-1,\quad \alpha+\beta\geq 0,\quad \frac{n-1}{n+1}\beta \leq \alpha< \beta+2.
\end{align}
Since $p^*=\frac{2(n+\beta+1)}{n+\alpha-1}$, the above condition  indicates  $2<p^*\leq 2^*$ for $n\geq 2$ and $p^*>2$ for $n=1$.

In  this section,  we shall prove Theorem \ref{thm-reg} using the Moser iteration technique: under condition \eqref{eq:constr-a-b} on $\alpha$ and $\beta$,  the weak positive solutions to  \eqref{genequ-1} are H\"older continuous up to the boundary.
Comparing with the classical Moser iteration technique (see \cite{GT1983,Brezis1990} for example), we use weighted Sobolev inequality \eqref{GGN-2} instead of classical Sobolev inequality.
\begin{proposition}\label{thm:regularity}
	Suppose $\alpha, \beta$ satisfy condition \eqref{eq:constr-a-b} and $0\leq u\in \mathcal{D}^{1,2}_{\alpha}(\RpN)$ is a weak solution to equation \eqref{genequ-1}. Then, for any $1\leq q<\infty$, we have
	\[u^q\in \mathcal{D}^{1,2}_{\alpha,loc}(\RpN)
	\quad\text{and}\quad u^q\in L^2_{\beta,loc}(\overline{\mathbb{R}^{n+1}_+}).\]
\end{proposition}

\begin{proof} We shall prove this by iteration. Suppose $\eta\in C_0^\infty(\overline{\mathbb{R}^{n+1}_+})$, $\theta>0$ and $K\geq 0$. Denote
	\[\phi=\eta^2 u\cdot\min\{u^{2\theta}, M^2\},\]
	where $\theta$ is to be chosen. By Lemma \ref{density}, it is easy to check that $\phi \in \mathcal{D}^{1,2}_{\alpha}(\RpN)$. Testing \eqref{weak-1} with $\phi$, we have
	\begin{align}\label{eq:uphi-1}
	\int_{\RpN} t^\alpha \nabla u \cdot \nabla \phi dydt=\int_{\RpN} t^\beta u^{p^*-1}\phi dydt.
	\end{align}
	Since	\begin{align*}
	\nabla u \cdot \nabla \phi 
	\geq&\frac{1}{2}\eta^2|\nabla u|^2\min\{u^{2\theta}, M^2\}-2|\nabla \eta|^2|u|^2\min \{u^{2\theta},M^2\}\\
	&+2\theta^{-1}\eta^2  u^2 |\nabla (u^{\theta})|^2 \chi_{\{u^\theta\leq M\}},
	\end{align*}
back to \eqref{eq:uphi-1}, we have
	\begin{align*}
	&\frac12\int_{\mathbb{R}^{n+1}_+}t^\alpha \eta^2|\nabla u|^2\min\{u^{2\theta},M^2\}dydt+2\theta^{-1}\int_{\{u^\theta\leq M\}}t^\alpha \eta^2 u^2 |\nabla u^{\theta}|^2  dydt\\
	\leq& 2\int_{\mathbb{R}^{n+1}_+} t^\alpha |\nabla \eta|^2u^2\min \{u^{2\theta},M^2\}dydt+\int_{\mathbb{R}^{n+1}_+} t^\beta \eta^2 u^{p^*}\min\{u^{2\theta},M^2\}dydt.
	\end{align*}
	Denote $w=u\cdot\min\{u^\theta,M\}$. The above inequality implies
	\begin{equation}\label{test-w}
\int_{\mathbb{R}^{n+1}_+}t^\alpha |\nabla (\eta w)|^2dydt
\leq C(1+\theta)\big(\int_{\mathbb{R}^{n+1}_+} t^\alpha |\nabla \eta|^2 w^2 dydt+\int_{\mathbb{R}^{n+1}_+}t^\beta \eta^2 u^{p^*-2} w^2dydt\big).
\end{equation}
For $L>0$, divide $\mathbb{R}^{n+1}_+$ into $\{u\leq L\}$ and $\{u\geq L\}$, then the second term of RHS in \eqref{test-w} satisfies
\begin{equation*}
\begin{split}
&\int_{\mathbb{R}^{n+1}_+}t^\beta \eta^2 u^{p^*-2}w^2dydt\\
\leq&L^{p^*-2}\int_{\mathbb{R}^{n+1}_+} t^\beta \eta^2w^2dydt
+C\big(\int_{\{u\geq L\} \cap \text{supp} \eta}t^\beta u^{p^*}dydt\big)^{1-\frac{2}{p^*}}\big(\int_{\mathbb{R}^{n+1}_+}t^\beta (\eta w)^{p^*}dydt\big)^{\frac{2}{p^*}}\\
\leq&L^{p^*-2}\int_{\mathbb{R}^{n+1}_+} t^\beta \eta^2 w^2 dydt
+S^{-1}_{n+1,\alpha,\beta}\big(\int_{\{u\geq L\} \cap \text{supp} \eta}t^\beta u^{p^*}dydt\big)^{1-\frac{2}{p^*}}	\int_{\mathbb{R}^{n+1}_+}t^\alpha |\nabla (\eta w)|^2dydt.
\end{split}
\end{equation*}
Back to \eqref{test-w}, since $p^*>2$ and $\lim_{L\to \infty}\int_{\{u\geq L\} \cap \text{supp}\eta}t^\beta u^{p^*}dydt= 0 $, for $L$ large enough, we have
\begin{eqnarray}\label{4.5-1}
\int_{\mathbb{R}^{n+1}_+}t^\alpha |\nabla (\eta w)|^2dydt\leq C(\theta, L)\Big(\int_{\mathbb{R}^{n+1}_+} t^\alpha w^2|\nabla \eta|^2dydt+\int_{\mathbb{R}^{n+1}_+} t^\beta w^2\eta^2dydt\Big).
\end{eqnarray}

For $\alpha\geq\beta$, by \eqref{4.5-1}, we have
\begin{align*}
\int_{\mathbb{R}^{n+1}_+}t^\alpha |\nabla (\eta w)|^2dydt\leq C\int_{\text{supp}\eta} t^\beta w^2(|\nabla \eta|^2+\eta^2)dydt,
\end{align*}
where $C$ is independent of $M$. Let $M\rightarrow +\infty$, then for any $\eta\in C_0^\infty(\mathbb{R}^{n+1})$,
\begin{align}\label{loc-integral}
\int_{\mathbb{R}^{n+1}_+}t^\alpha |\nabla (\eta u^{\theta+1})|^2dydt\leq C(\eta)\int_{\text{supp}\eta} t^\beta u^{2(\theta+1)}dydt.
\end{align}
For $u\in{\cal D}^{1,2}_{\alpha}(\mathbb{R}^{n+1}_+)$, we have $u^{\frac{p^*}{2}}\in L^{2}_{\beta}(\overline{\mathbb{R}^{n+1}_+})$. Take $\theta_1=\frac{p^*}{2}-1$, then according to \eqref{loc-integral} and $\int_{K}t^\alpha u^{p^*}dydt\leq C(K)\int_{K}t^\beta u^{p^*}dydt$ for any compact set $K \subset\overline{\mathbb{R}^{n+1}_+}$, we have $u^{\frac{p^*}{2}}\in{\cal D}^{1,2}_{\alpha,loc}(\mathbb{R}^{n+1}_+)$, which implies that $u^{(\frac{p^*}{2})^2}\in L^2_{\beta,loc}(\overline{\mathbb{R}^{n+1}_+})$. Taking $\theta_{i+1}+1=\frac{p^*}{2}(\theta_i+1)$, $i\geq 1$, we can get for $q=\big(\frac{p^*}{2}\big)^{i}, \, i\geq 1$, $u^q\in L^2_{\beta,loc}(\overline{\mathbb{R}^{n+1}_+})$ and $u^q\in {\cal D}^{1,2}_{\alpha,loc}(\mathbb{R}^{n+1}_+)$. And for general $q\geq 1$, we can get the conclusion by  an interpolation inequality.

For $0< \alpha<\beta$, by \eqref{4.5-1}, we have
\begin{align}\label{loc-integral-2}
\int_{\mathbb{R}^{n+1}_+}t^\alpha |\nabla (\eta u^{\theta+1})|^2dydt\leq C(\eta)\int_{\text{supp}\eta} t^\alpha u^{2(\theta+1)}dydt,
\end{align}
Similarly, we can get that for $q=\big(\frac{p_\alpha}{2}\big)^{i}, \, i\geq 1$, $u^q\in L^2_{\alpha,loc}(\overline{\mathbb{R}^{n+1}_+})$ and $u^q\in {\cal D}^{1,2}_{\alpha,loc}(\mathbb{R}^{n+1}_+)$, where $p_\alpha=\frac{2(n+\alpha+1)}{n+\alpha-1}$. At the same time, $u^q\in L^2_{\beta,loc}(\overline{\mathbb{R}^{n+1}_+})$. For general $q\geq 1$, we can get the conclusion by an interpolation inequality.
\end{proof}

\begin{proposition}\label{local-infty}
	Suppose $0\leq u\in \mathcal{D}^{1,2}_{\alpha}(\RpN)$ is a weak solution to equation \eqref{genequ-1} and $\alpha, \beta$ satisfy condition \eqref{eq:constr-a-b}, then $u\in L^\infty_{loc}(\overline{\mathbb{R}^{n+1}_+})$.
\end{proposition}
\begin{proof}
	Since there is no singularity in the interior, we only consider the $L_{loc}^{\infty}$ property near the boundary. Denote $\phi=\eta^2 u^{2\theta+1}$ for some $\theta>0$ and $\phi\in C^\infty_0(\overline{\mathbb{R}^{n+1}_+})$ with $supp \eta\subset \overline{B_2^+}=\overline{B_2^+(0)}$. By Proposition \ref{thm:regularity} and Lemma \ref{density}, it holds $\phi\in \cal{D}_\alpha^{1,2}(\mathbb{R}^{n+1}_+)$. Testing \eqref{weak-1} by $\phi$, we have
	\begin{align*}
	&2\int_{B_2^+}t^\alpha \eta  u^{2\theta+1} \nabla \eta\cdot \nabla  udydt+(2\theta+1)\int_{B_2^+} t^\alpha \eta^2  u^{2\theta}|\nabla  u|^2dydt=\int_{B_2^+}t^\beta \eta^2 u^{p^*+2\theta}dydt.
	\end{align*}
	It follows that
	\begin{align*}
	\int_{B_2^+}t^\alpha \eta^2  u^{2\theta}|\nabla u|^2dydt\leq C(\theta) \int_{B_2^+} \big(t^\alpha |\nabla \eta|^2  u^{2\theta+2} +t^\beta \eta^2u^{p^*+2\theta}\big)dydt
	\end{align*}
	for some constant $C(\theta)=\frac2{2\theta+\frac12}\le4$.
	Write $w= u^{\theta+1}$, then we have
	\begin{align}\label{eq:moser-bt-i2}
	\int_{B_2^+}t^\alpha |\nabla (\eta w)|^2dydt \leq C(\theta)(\theta+1)^2\big(\int_{B_2^+}t^\alpha |\nabla\eta|^2w^2dydt+\int_{B_2^+}t^\beta \eta^2u^{p^*-2}w^{2}dydt\big).
	\end{align}
	Using H\"{o}lder's inequality, we obtain that the second term of RHS in \eqref{eq:moser-bt-i2} satisfies
	\begin{align}\label{alpha>beta-1}
	\int_{B_2^+}t^\beta  \eta^2 u^{p^*-2} w^{2}dydt\leq\big(\int_{B_2^+}t^\beta u^{(p^*-2)\frac{q}{q-1}}dydt\big)^{1-\frac{1}{q}}\big(\int_{B_2^+}t^\beta(\eta w)^{2q} dydt\big)^{\frac{1}{q}}
	\end{align}
	for some $q$ fixed such that $p^*>2q>2$. By Theorem \ref{thm:regularity}, we know
	\begin{align*}
	V:=\big(\int_{B_2^+}t^\beta u^{(p^*-2)\frac{q}{q-1}}dydt\big)^{1-\frac{1}{q}}<\infty.
	\end{align*}
Also,  by Young's inequality, we have
\begin{eqnarray}\label{alpha>beta-2}
	\big(\int_{B_2^+}t^\beta(\eta w)^{2q}dydt\big)^{\frac{1}{2q}}\leq \delta \big(\int_{B_2^+}t^\beta(\eta w)^{p^*}dydt\big)^{\frac{1}{p^*}}+\delta^{-\sigma}\big( \int_{B_2^+} t^\beta \eta^2 w^2dydt\big)^{\frac12},~~~~~
\end{eqnarray}
where $\sigma=\frac{p^*(q-1)}{p^*-2q}$. Putting these back to \eqref{eq:moser-bt-i2}, one gets
	\begin{align*}
	\int_{B_2^+}t^\alpha |\nabla (\eta w)|^2dydt
	\leq& C(1+\theta)^2\big[\int_{B_2^+}t^\alpha |\nabla\eta|^2w^2dydt+V \delta^2 \big(\int_{B_2^+}t^\beta(\eta w)^{p^*}dydt\big)^{\frac{2}{p^*}}\notag\big.\\
	&\big.+V\delta^{-2\sigma} \int_{B_2^+} t^\beta \eta^2 w^2dydt\big] .
	\end{align*}
	
	For $\alpha\geq\beta$, using inequality \eqref{GGN-2} and choosing $\delta>0$ small enough, we have
	\begin{align}\label{eq:m-fineq}
	&\big(\int_{B_2^+}t^\beta(\eta w)^{p^*}dydt\big)^{\frac{2}{p^*}}\notag\\
\leq&  C(1+\theta)^2\int_{B_2^+}t^\alpha |\nabla\eta|^2w^2dydt+C(1+\theta)^{2\sigma+2} V^{\sigma+1}\int_{B_2^+} t^\beta \eta^2 w^2dydt\notag\\
\leq &  C(1+\theta)^2\int_{B_2^+}t^\beta |\nabla\eta|^2w^2dydt+
	C(1+\theta)^{2\sigma+2} V^{\sigma+1}\int_{B_2^+} t^\beta \eta^2 w^2dydt	.
	\end{align}
	For $r\leq 2$ and $p>2$, define
	\[\Phi(p,r)=\big(\int_{B_r^+}t^\beta u^pdydt\big)^{\frac{1}{p}}.\]
	Set $\gamma=2(1+\theta)$, and choose $\eta=1$ in $B_{r_2}^+$ and $\eta=0$ in $B_2^+\backslash B_{r_1}^+$, where $0<r_2<r_1\le2$. Then \eqref{eq:m-fineq} shows that for any $\gamma>2$
	\[\Phi\big(\frac{p^*}{2}\gamma,r_2\big)\leq \big[\frac{C(\gamma \sqrt{V})^{2(\sigma+1)}}{(r_1-r_2)^2}\big]^{\frac{1}{\gamma}}\Phi(\gamma,r_1).\]
	By iterating the above inequality: set $r_m=1+2^{-m}$, $\gamma_0=p>2$ and $\gamma_m=\gamma_{m-1}\frac{p^*}{2}$, $m=1,2,\cdots$, one gets
	\[\Phi(\gamma_m,r_m)\leq (C\cdot \sqrt{V}p)^{\frac{2(1+\sigma)}{p}\sum (p^*/2)^{-k}}(2(\frac{p^*}{2})^{\sigma+1})^{\frac{2}{p}\sum k(p^*/2)^{-k}}\Phi(p,2).\]
	Since $p^*>2$, we have$\sum (p^*/2)^{-k}<\infty$ and $\sum k(p^*/2)^{-k}<\infty$. Letting $m\to \infty$, by $\beta>-1$, we have $\sup_{B_1^+} u<\infty$.
	
 For $0<\alpha<\beta$, changing the weight from $\beta$
		to $\alpha$ in  \eqref{alpha>beta-1} and \eqref{alpha>beta-2}, one has
	\begin{align*}
	\int_{B_2^+}t^\beta  \eta^2 u^{p^*-2} w^{2}dydt&\leq\big(\int_{B_2^+}t^{\tilde{\gamma}} u^{(p^*-2)\frac{q}{q-1}}dydt\big)^{1-\frac{1}{q}}\big(\int_{B_2^+}t^\alpha(\eta w)^{2q} dydt\big)^{\frac{1}{q}}\\
	&\leq C\big(\int_{B_2^+}t^\beta u^{(p^*-2)\frac{q}{q-1}}dydt\big)^{1-\frac{1}{q}}\big(\int_{B_2^+}t^\alpha(\eta w)^{2q} dydt\big)^{\frac{1}{q}}\\
	&\leq V\big(\int_{B_2^+}t^\alpha(\eta w)^{2q} dydt\big)^{\frac{1}{q}},
	\end{align*} and
	\begin{align*}
	\big(\int_{B_2^+}t^\alpha(\eta w)^{2q}dydt\big)^{\frac{1}{2q}}\leq \delta \big(\int_{B_2^+}t^\alpha(\eta w)^{p_\alpha}dydt\big)^{\frac{1}{p_\alpha}}+\delta^{-\tilde{\sigma}}\big( \int_{B_2^+} t^\alpha \eta^2 w^2dydt\big)^{\frac12},
	\end{align*}
	where $p_\alpha=\frac{2(n+\alpha+1)}{n+\alpha-1}>2q>2$, $\tilde{\gamma}=\frac{\beta q-\alpha}{q-1}>\beta$ and $\tilde{\sigma}=\frac{p_\alpha(q-1)}{p_\alpha-2q}$. Similarly, we can get
	\begin{align*}
	&\big(\int_{B_2^+}t^\alpha(\eta w)^{p_\alpha}dydt\big)^{\frac{2}{p_\alpha}}\\
	\leq & C(1+\theta)^2\int_{B_2^+}t^\alpha |\nabla\eta|^2w^2dydt+C(1+\theta)^{2\tilde{\sigma}+2} V^{\tilde{\sigma}+1}\int_{B_2^+} t^\alpha \eta^2 w^2dydt.
	\end{align*}
	Let $$\tilde{\Phi}(p,r)=\big(\int_{B_r^+}t^\alpha u^pdydt\big)^{\frac{1}{p}}.$$
	Iterating as before, we can get $\sup_{B_1^+} u<+\infty$.
\end{proof}

The $L^\infty_{loc}$ bound yields that  $u$ is actually smooth in $\mathbb{R}^{n+1}_+$ by the standard elliptic estimates.  Next, we shall show that $u$ is H\"{o}lder continuous up to the boundary. To that end, we firstly need to establish some lemmas.

We need the following weak  weighted Poincar\'e inequalities.

\begin{lemma}\label{lem:poincare-equal}
Let $\alpha>0$. There exists $C$ depending only on $n$ and $\alpha$ such that
	\begin{align}\label{eq:poin-eq}
	\int_{B^+_r(X)}t^\alpha |u-u_{B^+_r(X),\alpha}|^2dydt\leq Cr^2\int_{B^+_{3r}(X)}t^\alpha |\nabla u|^2dydt,
	\end{align}
	holds for any $r>0$, $X\in\overline{\mathbb{R}^{n+1}_+}$ and $u\in \mathcal{D}^{1,2}_{\alpha,loc}(\mathbb{R}^{n+1}_+)$. Here we write \[u_{B^+_r(X),\alpha}=\frac{\int_{B^+_r(X)}t^\alpha u(y,t)dydt}{\int_{B^+_r(X)}t^\alpha dydt}.\]
\end{lemma}

\begin{remark}\label{even-1}
For $u\in \mathcal{D}^{1,2}_{\alpha,loc}(\mathbb{R}^{n+1}_+)$, similarly to the proof of  Lemma \ref{density}, we can show that: for any $\varepsilon>0$ and compact subset $K\subset\mathbb{R}^{n+1}$, there is $v\in C^\infty_0(\mathbb{R}^{n+1})$ which is even with respect to $t$, such that
	$$\int_{K\cap\mathbb{R}^{n+1}_+ }t^\alpha |\nabla v-\nabla u|^2dydt<\varepsilon.$$
\end{remark}
\begin{proof} According to Remark \ref{even-1}, we only need to
prove the  inequality \eqref{eq:poin-eq}  for $r=1$, and $u\in C^\infty_0(\mathbb{R}^{n+1})$ which is even with respect to $t$. The general case follows by scaling and approximation. Now we only need to show that there exists a constant $c$ such that
\begin{align}\label{pe-1}
\int_{B^+_1(X)}t^\alpha |u-c|^2dydt\leq C\int_{B^+_3(X)}t^\alpha |\nabla u|^2dydt.
\end{align}
In fact, we have
\begin{align*}&\int_{B^+_1(X)}t^\alpha|u
-u_{B^+_1(X),\alpha}|^2dydt\\
\leq &2\int_{B^+_1(X)}t^\alpha|u-c|^2dydt+2\int_{B^+_1(X)}t^\alpha |c-u_{B^+_1(X),\alpha}|^2dydt\\
\leq& 4\int_{B^+_1(X)}t^\alpha |u-c|^2dydt,
\end{align*}
Combining the above and \eqref{pe-1}, we can obtain \eqref{eq:poin-eq}.

Next, we prove \eqref{pe-1}. We consider $t_c$ in two cases for $X=(y_c,t_c)\in\overline{\mathbb{R}^{n+1}_+}$.
	
If $t_c\geq\frac{5}{4}$,  then $B^+_1(X)=B_1(X)$ and $dist(B_1(X), \partial\mathbb{R}^{n+1}_+)\geq\frac14$.
Inequality \eqref{eq:poin-eq} just becomes the classical Poincar\'{e} inequality on $\mathbb{R}^{n+1}$.  We can choose $c=\frac{1}{w_{n+1}}\int_{B_1(X)}u(y,t)dydt$, then
\begin{align*}
\int_{B_1(X)}t^\alpha |u-c|^2dydt\leq& (t_c+1)^\alpha \int_{B_1(X)}|u-c|^2dydt\\
\leq&C(t_c+1)^\alpha \int_{B_1(X)}|\nabla u|^2dydt\\
\leq&C\big(\frac{t_c+1}{t_c-1}\big)^\alpha\int_{B_1(X)}t^\alpha|\nabla u|^2dydt\\
\leq&9^\alpha C\int_{B_1(X)}t^\alpha|\nabla u|^2dydt.
\end{align*}

If $t_c\in [0,\frac{5}{4})$, without loss of generality, we assume that $y_c=0$, then
\begin{align}\label{pe-2}
B^+_{1}(X)\subset B^+_{t_c+1}(0)\subset B^+_{\sqrt{t_c^2+(t_c+1)^2}}(X)\subset B^+_3(X).
\end{align}

\noindent $1)$ Suppose that $\alpha=m$ is a positive integer. Let $\tilde{u}(y,z)=u(y,t)$ with $y\in\mathbb{R}^n,\ z\in \mathbb{R}^{m+1}$ and $|z|=t\geq 0$. We consider $u(y,t)$ on
$B^+_{t_c+1}(0)$, that is,  we  consider $\tilde{u}(y,z)$ on $B^{n+m+1}_{t_c+1}(0)\subset\mathbb{R}^{n+m+1}$.
Under polar coordinates, $dz=t^mdtdS_m$, then it yields
$$\int_{B^{n+m+1}_{t_c+1}(0)}(|\nabla \tilde{u}|^2+|\tilde{u}|^2) dydz
=(m+1)\omega_{m+1}\int_{B^+_{t_c+1}(0)}t^{m}(|\nabla u|^2+|u|^2) dydt<\infty.$$
By classical Poincar\'e  inequality on $\mathbb{R}^{n+m+1}$,
\begin{align}\label{pe-3}
\int_{B^{n+m+1}_{t_c+1}(0)}|\tilde u-c|^2dydz\leq C\int_{B^{n+m+1}_{t_c+1}(0)}|\nabla \tilde u|^2dydz,
\end{align}
where $c=\frac{1}{w_{n+m+1}(t_c+1)^{n+m+1}}\int_{B^{n+m+1}_{t_c+1}(0)}\tilde udydz$. \eqref{pe-3} is equivalent to
\begin{align}\label{pe-4}
\int_{B^+_{t_c+1}(0)}t^m|u-c|^2dydt\leq C\int_{B^+_{t_c+1}(0)}t^{m}|\nabla u|^2dydt.
\end{align}
Combining \eqref{pe-4} and \eqref{pe-2},  we arrive at \eqref{pe-1}.

\noindent $2)$ Suppose that $\alpha\in (m-1,m)$ for some positive integer $m$.  We choose cut-off function $\eta\in C^\infty_0(B_\frac{5}{4}(X))$ with $\eta|_{B_1(X)}=1,\ 0\leq \eta\leq 1,\ |\nabla \eta|\leq C$. Using the fact $\frac{2(n+\alpha+1)}{n+m-1}>2$,
 H\"older inequality and weighted Sobolev inequality \eqref{GGN-2}, we have
\begin{align}\label{pe-5}
\int_{B^+_1(X)}t^\alpha|u-c|^2dydt\leq& C\big(\int_{B^+_\frac{5}{4}(X)}t^{\alpha}|\eta(u-c)|^{\frac{2(n+\alpha+1)}{n+m-1}}dydt\big)^\frac{n+m-1}{n+\alpha+1}\nonumber\\
\leq & C \int_{B^+_\frac{5}{4}(X)}t^{m}|\nabla (\eta(u-c))|^2dydt\nonumber\\
\leq& C\int_{B^+_\frac{5}{4}(X)}t^m |\nabla u|^2dydt+C\int_{B^+_\frac{5}{4}(X)}t^m |u-c|^2dydt.
\end{align}
Noting that
	\begin{align*}
		B^+_{\frac{5}{4}}(X)\subset B^+_{t_c+\frac{5}{4}}(0)\subset B^+_{\sqrt{t_c^2+(t_c+\frac{5}{4})^2}}(X)\subset B^+_3(X),
	\end{align*}
similar to the case $1)$, there exists a constant $C$ such that
	$$\int_{B^+_\frac{5}{4}(X)}t^m |u-c|^2dydt\leq C\int_{B^+_3(X)}t^m |\nabla u|^2dydt.$$
Back to \eqref{pe-5} and using $\alpha<m$, we have
\begin{eqnarray*}
\int_{B^+_1(X)}t^\alpha|u-c|^2dydt\leq C\int_{B^+_3(X)}t^m |\nabla u|^2dydt\leq C\int_{B^+_3(X)}t^\alpha |\nabla u|^2dydt.
\end{eqnarray*}
\end{proof}

Now let us deal with general $\alpha$ and $\beta$ satisfying \eqref{eq:constr-a-b}.
\begin{lemma}\label{thm:poincare-weight}
Suppose $\alpha$ and $\beta$ satisfy \eqref{eq:constr-a-b}, then there exists $C$ depends on $n,\alpha,\beta$ such that
\begin{align}\label{eq:poin-weight}
\frac{\int_{B^+_r(X)}t^\beta |u-u_{B^+_r(X),\beta}|dydt}{\int_{B^+_r(X)}t^\beta dydt}\leq Cr\big(\frac{\int_{B^+_{4r}(X)}t^\alpha |\nabla u|^2dydt}{\int_{B^+_{4 r}(X)}t^\alpha dydt}\big)^{\frac12}
\end{align}
holds for any $r>0$, $X\in\overline{\mathbb{R}^{n+1}_+}$ and $u\in \mathcal{D}^{1,2}_{\alpha,loc}(\mathbb{R}^{n+1}_+)$, where \[u_{B^+_r(X),\beta}=\frac{\int_{B^+_r(X)}t^\beta u(y,t)dydt}{\int_{B^+_r(X)}t^\beta dydt}.\]
\end{lemma}
\begin{proof}
It also suffices to prove the above inequality for $r=1$.
For the same reason in Lemma \ref{lem:poincare-equal}, if we can find a constant $c$ such that
\begin{align*}
\frac{\int_{B^+_1(X)}t^\beta |u-c|dydt}{\int_{B^+_1(X)}t^\beta dydt}\leq C\big(\frac{\int_{B^+_{4}(X)}t^\alpha |\nabla u|^2dydt}{\int_{B^+_{4}(X)}t^\alpha dydt}\big)^{\frac12},
\end{align*}
then \eqref{eq:poin-weight} is verified.

Let $X=(y_c,t_c)\in\overline{\mathbb{R}^{n+1}_+}$.
	
$1)$ For $t_c>5$, $B^+_4(X)=B_4(X)$ and $dist(B_4(X),\partial \mathbb{R}^{n+1}_+)>1$, then
the inequality can be reduced to the Poincar\'e inequality in Euclidean space without weight.
	
$2)$ For $0\leq t_c\leq 5$, there is a constant $C$ such that
$$(\int_{B^+_{4 }(X)}t^\alpha dydt)^{\frac12}\leq C\int_{B^+_1(X)}t^\beta dydt,$$	
so it suffices to prove that
\begin{align}\label{eq:poincare-sim-2}
\int_{B^+_1(X)}t^\beta |u-c|dydt\leq C\big(\int_{B^+_4(X)}t^\alpha |\nabla u|^2dydt\big)^{\frac12}
\end{align} holds for some  constant $c$.

Suppose $\eta\in C_0^\infty(B_\frac{5}{4}(X))$ is a cut-off function with $\eta|_{B_1(X)}=1$, $0\leq\eta\leq1$ and $|\nabla \eta|\leq C$. By H\"older's inequality and \eqref{GGN-2}, we have
\begin{align}\label{5/4}
\int_{B^+_1(X)}t^\beta|u-c|dydt
\leq&C \big(\int_{B^+_{\frac{5}{4}}(X)}t^\beta|\eta(u-c)|^{p^*}dydt\big)^{\frac{1}{p^*}}\nonumber\\
\leq &C \big(\int_{B^+_{\frac{5}{4}}(X)}t^\alpha|\nabla u|^2dydt\big)^{\frac12}+C\big(\int_{B^+_{\frac{5}{4}}(X)}t^\alpha|u-c|^2dydt\big)^{\frac12}.
\end{align}
Taking $c=u_{B^+_{\frac{5}{4}}(X),\alpha}$ and using Lemma \ref{lem:poincare-equal}, we have
\begin{align*}
\int_{B^+_{\frac{5}{4}}(X)}t^\alpha |u-c|^2dydt\leq C \int_{B^+_{\frac{15}{4}}(X)}t^\alpha|\nabla u|^2dydt.
\end{align*}
Back to \eqref{5/4}, we obtain \eqref{eq:poincare-sim-2}. The lemma is proved.
\end{proof}

\medskip
Suppose $d\mu$ is a doubling measure on some domain $\Omega\subset \mathbb{R}^{n+1}$, that is $\mu(2B)\leq C(\mu)\mu(B)$ for any $2B\subset \Omega$. A function $w\in L_{loc}^1(\Omega,d\mu)$ is said to be in $BMO(\Omega,\mu)$ if there is a constant $C>0$ such that for every ball $B$ satisfying $2B\subset\Omega $, it holds that
\[\frac{1}{\mu(B)}\int_{B}|w-w_B|d\mu\leq C.\]
Here $w_B=\frac{1}{\mu(B)}\int_{B}wd\mu$ is the average of $w$ on $B$.
The least $C$ such that the above inequality holds is called the $BMO(\Omega,\mu)-$norm of $w$. Similar to the classical result of BMO space on Euclidean space, we have the following result from  Corollary 19.10 in \cite{JTO1993}.

\begin{lemma}[John-Nirenberg lemma for doubling measures] \label{lem:JN}Suppose $\mu$ is a doubling measure. A function $w$ is in $BMO(\Omega,d\mu)$ if and only if there exist constant $c$ and $C$ such that
	\[\frac{1}{\mu({B})}\int_Be^{c|w-w_B|}d\mu\leq C\]
	for every ball $B$ such that $2B\subset \Omega$.
\end{lemma}

One consequence of this lemma is that
\begin{align}\label{JN-1}
\int_B e^{cw}d\mu\int_{{B}}e^{-cw}d\mu\leq C[\mu(B)]^2.
\end{align}
\begin{proposition}\label{thm:inf-moser}
	Suppose that  \eqref{eq:constr-a-b} holds and $0\leq u\in\mathcal{D}_{\alpha,loc}^{1,2}(\mathbb{R}^{n+1}_+)$ satisfies
	\begin{align}\label{g}
	\int_{\RpN} t^\alpha \nabla u \cdot \nabla \phi dydt\geq \int_{\RpN}t^\beta g\phi dydt
	\end{align}
	for some $g\in L^\infty_{loc}(\overline{\mathbb{R}^{n+1}_+})$ and any $0\leq \phi\in {\cal D}^{1,2}_\alpha(\mathbb{R}^{n+1}_+)$ with compact support. Then there exist $C>0$ depending on $n,\alpha,\beta$ such that for any $r>0$ and $y\in\mathbb{R}^n$,
	\begin{align}\label{harnack}
	C\big(\inf_{B_r^+(y,0)}u+r^{{\beta+2-\alpha}}|g|_{L^\infty(B^+_r(y,0))}\big)\geq \frac{1}{r^{n+1+\beta}}\int_{B_{2r}^+(y,0)}t^\beta u dydt.
	\end{align}
\end{proposition}

\begin{proof} We just prove the result for $r=1$ and $y=0$, the general case follows by rescaling and translation. Let $k=|g|_{L^\infty(B_r)}+\varepsilon$ for some $\varepsilon>0$. Define $\bar u=u+k$.
Plugging $\phi=\eta^2\bar u^{2\theta+1}$ in \eqref{g}  for some $\theta<-\frac12$ and cut-off function $\eta$ with $supp\, \eta \subset \overline{B_3^+}$, leads to
\begin{align*}
2\int_{B_3^+}t^\alpha \eta \bar u^{2\theta+1} \nabla \eta\cdot \nabla \bar udydt+(2\theta+1)\int_{B_3^+} t^\alpha \eta^2 \bar u^{2\theta}|\nabla \bar u|^2dydt\geq \int_{\RpN}t^\beta \eta^2\bar u^{2\theta+1}gdydt.
\end{align*}
Choosing a fixed $\theta_0<-\frac{1}{2}$ such that $\theta\leq\theta_0$, we have
\begin{eqnarray}\label{inf-1}
& &\int_{B_3^+}t^\alpha \eta^2 \bar u^{2\theta}|\nabla \bar u|^2dydt\nonumber\\
&\leq &(\frac{2}{2\theta+1})^2\int_{B_3^+} t^\alpha |\nabla \eta|^2 \bar u^{2\theta+2}dydt +(-\frac{2}{2\theta+1})\int_{B_3^+}t^\beta\eta^2 \frac{|g|}{k}\bar u^{2\theta+2}dydt\nonumber\\
&\leq& C(\theta_0)\int_{B_3^+}(t^\alpha |\nabla \eta|^2 \bar u^{2\theta+2}+t^\beta\eta^2 \bar{u}^{2\theta+2})dydt.
\end{eqnarray}
Define
\begin{equation*}
w=
\begin{cases}
\bar u^{\theta+1}, & \theta\neq -1 ,\\
\log \bar u, & \theta=-1 .
\end{cases}
\end{equation*}
For $\theta\leq\theta_0$, it follows from \eqref{inf-1} that
\begin{align}\label{inf-2}
\int_{B_3^+}t^\alpha \eta^2 |\nabla w|^2dydt\leq
\begin{cases}C(1+\theta)^2\int_{B_3^+}(t^\alpha |\nabla \eta|^2 w^2 +t^\beta\eta^2 w^2)dydt,&\theta\neq  -1,\\
C\int_{B_3^+}(t^\alpha |\nabla \eta|^2 +t^\beta \eta^2)dydt,&\theta=-1.
\end{cases}
\end{align}
	
For $\alpha\geq\beta$, if $\theta\neq 1$ and $\theta\leq\theta_0$, by inequality \eqref{GGN-2} and \eqref{inf-2}, we have
\begin{align}\label{inf-3}
\big(\int_{B_3^+}t^\beta (\eta w)^{p^*}dydt\big)^{\frac{2}{p^*}}\leq (C(1+\theta)^2+2)\int_{B_3^+}(t^\beta |\nabla \eta|^2 w^2 +t^\beta\eta^2 w^2)dydt.
\end{align}
For $p\neq 0$, define
	\[ \Phi(p,r)=\big(\int_{B_r^+}t^\beta \bar u^pdydt\big)^{\frac{1}{p}}.\]
Set $0<r_2<r_1\leq 3$ and $\gamma=2(\theta+1)$ for some $\theta\leq\theta_0<-\frac12$ and $\theta\neq -1$. Choose a cut-off function $\eta=1$ in $B_{r_2}^+$ and $\eta=0$ in $B_3^+\backslash B_{r_1}^+$, $|\nabla \eta|\leq \frac{C}{r_1-r_2}$. Then \eqref{inf-3} implies
\begin{align}
\Phi(\gamma,r_1)\leq \big[\frac{C\gamma^2}{(r_1-r_2)^2}\big]^{\frac{1}{|\gamma|}}\Phi(\gamma\frac{p^*}{2},r_2),\quad \gamma<0,\label{inf-4}\\
\Phi(\gamma\frac{p^*}{2},r_2)\leq \big[\frac{C}{(r_1-r_2)^2}\big]^{\frac{1}{\gamma}}\Phi(\gamma,r_1),\quad\gamma>0.
\label{inf-5}
\end{align}
Iterate inequality \eqref{inf-4}. That is, for some $p_0\in(0,1)$, set $r_m=2+2^{-m}$, $\gamma_0=-p_0$  and $\gamma_m=\gamma_{m-1}\frac{p^*}{2}$, $m=1,2,\cdots$. Sending $m\to\infty$, we have for $\beta>-1$,
\begin{equation}\label{mj-1}C\inf_{B_1^+}\bar u\geq  \Phi(-p_0,3),
\end{equation}
where $C>0$ depends on $n,\alpha, p_0$.
	
Iterate  inequality \eqref{inf-5}. That is, set $r_m=+2^{-m}$, $\gamma_0=p_0\in (0,1)$ and $\gamma_m=\gamma_{m-1}\frac{p^*}{2}$, $m=1,2,\cdots$, where we choose $p_0$ to guarantee  $\gamma_m\neq1$ for any $m$. After some finite steps, one gets
\[\Phi(\gamma_{m_0},r_{m_0})\leq C\Phi(\gamma_{m_0-1},r_{m_0-1})\leq C\Phi(p_0,3),\]
where $\gamma_{m_0-1}<1<\gamma_{m_0}$.	 By H\"{o}lder's inequality, we have
\begin{equation}\label{mj-a}
\Phi(1,2)\leq C\Phi(p_0,3).
\end{equation}
	
Next, we want to show for some $p_0$ small enough that
\begin{equation}\label{mj-2}
\Phi(p_0,3)\leq C\Phi(-p_0,3).
\end{equation}
Indeed, for $\theta=-1$, by \eqref{inf-2} we have
\begin{align}\label{mj-3}
\int_{B_{12}^{+}}t^\alpha \eta^2 |\nabla  w|^2dydt\leq C\int_{B_{12}^{+}}(t^\alpha |\nabla \eta|^2 +t^\beta \eta^2)dydt,
\end{align}	
where $\eta$ is a cut-off function with $supp\,\eta\subset B_{12}$.
Let $B^+_{2R}(X)\subset B^+_{12}$ for $X\in\overline{\mathbb{R}^{n+1}_+},\ R>0$,
and one can choose a cut-off function $\eta$ such that $\eta=1$ on $B_R(X)$, $supp \eta \subset B_{2R}(X)$ and $|\nabla\eta|\leq 2/R$. Then inequality \eqref{mj-3} implies
\begin{align}\label{2R}
\int_{B^+_{R}(X)}t^\alpha |\nabla w|^2dydt\leq CR^{-2} \int_{B^+_{2R}{(X)}}t^\alpha  dydt+C\int_{B^+_{2R}{(X)}}t^\beta  dydt.
\end{align}
For any $X=(y_c,t_c)\in\overline{\mathbb{R}^{n+1}_+}$ and $ r>0$ satisfying $B_{2r}(X)\subset B_3$, we have that $B_{8r}(X)\subset B_{12}$. By \eqref{2R}, Lemma \ref{thm:poincare-weight} and $\beta+2-\alpha>0$, we get
\begin{align*}
\frac{\int_{B^+_{r}(X)}t^\beta |w-w_{B^+_{r}(X),\beta}|dydt}{\int_{B^+_{r}(X)}t^\beta dydt}\leq& Cr\Big(\frac{\int_{B^+_{4r}(X)}t^\alpha |\nabla w|^2dydt}{\int_{B^+_{4r}(X)}t^\alpha dydt}\Big)^{\frac 12}\\
\leq&Cr\Big(\frac{r^{-2}\int_{B^+_{8r}{(X)}}t^\alpha  dydt+C\int_{B^+_{8r}{(X)}}t^\beta  dydt}{\int_{B^+_{4r}(X)}t^\alpha dydt}\Big)^{\frac 12}\\
\leq& C+C(t_c+r)^{\frac{\beta+2-\alpha}{2}}\leq C.
\end{align*}

Now, we extend $w$ evenly to the whole  space $\mathbb{R}^{n+1}$. For any $X\in\overline{\mathbb{R}^{n+1}_+}$ and $ r>0$ satisfying  $B_{2r}(X)\subset B_3$, we have
\begin{align*}
\frac{\int_{B_{r}(X)}|t|^\beta |w-w_{B_{r}(X),\beta}|dydt}{\int_{B_{r}(X)}|t|^\beta dydt}
\leq& \frac{2\int_{B_{r}(X)}|t|^\beta |w-w_{B^+_{r}(X),\beta}|dydt}{\int_{B_{r}(X)}|t|^\beta dydt}\\
\leq&\frac{4\int_{B^+_{r}(X)}t^\beta |w-w_{B^+_{r}(X),\beta}|dydt}{\int_{B^+_{r}(X)}t^\beta dydt}\leq C,
\end{align*}
which shows $w\in BMO(B_3, |t|^\beta dydt)$. Since $|t|^\beta$ is a weight with doubling property, that is for $B_{2r}(X)\subset B_3$,
\[\int_{B_{2r}(X)}|t|^{\beta}dydt\leq C(\beta)\int_{B_r(X)}|t|^{\beta}dydt.\]
Using  \eqref{JN-1}, there exists some $p_0>0$ small such that
\[\int_{B_3}e^{p_0w}|t|^\beta dydt\int_{B_3}e^{-p_0 w}|t|^\beta dydt\leq C.\]
Notice that $w=\log \bar u$, and $\bar{u}$ is even with respect to $t$, then from the above inequality, we obtain  \eqref{mj-2}.
	Combining \eqref{mj-1}, \eqref{mj-a} with \eqref{mj-2} and  letting $\varepsilon\to 0$,  we get our conclusion.
	
	For $0<\alpha<\beta$, similar to the proof in Proposition \ref{local-infty}, we replace $\beta,\, p^*$ by $\alpha,\, p_\alpha$ and get the conclusion.
\end{proof}

\begin{corollary}
	Suppose $0\leq u\in \mathcal{D}^{1,2}_{\alpha}(\RpN)$ is a weak solution to equation \eqref{genequ-1} and $\alpha, \beta$ satisfy condition \eqref{eq:constr-a-b}, then $u$ is H\"{o}lder continuous up to  the boundary $\partial \mathbb{R}^{n+1}_+$.
\end{corollary}
\begin{proof}
	For $r<\frac12$, define $M(r)=\sup_{B_r^+}u$, $m(r)=\inf_{B_r^+}u$, and $\omega(r)=M(r)-m(r)$, then for any nonnegative $\phi\in\mathcal{D}^{1,2}_{\alpha}(\RpN)$ with  $supp\phi\subset\overline{B_{2r}^+}$, it holds
	\begin{align*}
	\int_{B_{2r}^+ }t^{\alpha} \nabla [M(4r)-u] \cdot \nabla \phi dy dt &=-\int_{B_{2r}^+ }t^\beta u^{p^*-1}\phi dydt\geq -M(2)^{p^*-1}\int_{B_{2r}^+ }t^\beta \phi dydt,\\
	\int_{B_{2r}^+ }t^{\alpha} \nabla [u-m(4r)] \cdot \nabla \phi dy dt &=\int_{B_{2r}^+ }t^\beta u^{p^*-1}\phi dydt\geq 0.
	\end{align*}
	 By Proposition \ref{thm:inf-moser}, we get
	\begin{align*}
	\frac{1}{r^{n+1+\beta}}\int_{B_{2r}^+}t^\beta [M(4r)-u]dydt&\leq C\inf_{B_{r}^+}[M(4r)-u]+Cr^{{\beta+2-\alpha}}M(2)^{p^*-1}\\
	&=C[M(4r)-M(r)]+Cr^{{\beta+2-\alpha}}M(2)^{p^*-1},\\
	\frac{1}{r^{n+1+\beta}}\int_{B_{2r}^+}t^\beta[u-m(4r)]dydt&\leq C\inf_{B_{r}^+}[u-m(4r)]=C[m(r)-m(4r)].
	\end{align*}
	Summing the above two inequalities leads to
	\[\omega(r)\leq \frac{C-1}{C}\omega(4r)+r^{{\beta+2-\alpha}}M(2)^{p^*-1}.\]
	Then we conclude (see, for example \cite[Lemma 8.23]{GT1983}), that $u$ is H\"{o}lder continuous up to  the boundary.
\end{proof}

\begin{remark}\label{harnack-1}
	By \eqref{harnack} and the maximum principle, for nonnegative weak solution of equation \eqref{genequ-1}, if there is a point $(y_0,t_0)\in \overline{\mathbb{R}^{n+1}_+}$ such that $u(y_0,t_0)=0$, then $u\equiv 0$.
\end{remark}

\section{Classification results}
Though, in certain  cases  (see, for example, Obata \cite{Ob}, Escolar \cite{Es}, Beckner \cite{Bec2001}, Jerison and Lee \cite{JL1988}) one can use conformal invariant property to obtain the best constant for the sharp Sobolev type inequalities,  the more powerful way is to classify all positive solutions to the Euler-Lagrange equations satisfied by the extremal functions.  In this section, we shall prove Theorem \ref{L-1} through the proofs of a sequence of  propositions.

By Theorem \ref{thm-reg} we know that if $u$ is a nonnegative weak solution to equation \eqref{genequ-1}, then $u\in C^2(\mathbb{R}^{n+1}_+) \cap C^\gamma_{loc}(\overline{\mathbb{R}^{n+1}_+})$ for some $\gamma \in (0, 1)$, and $u$ satisfies
	\begin{align}\label{u-equ}
		-div(t^\alpha\nabla u)=t^\beta u^{p^*-1},\quad\text{in }\mathbb{R}^{n+1}_+.
	\end{align}
 Besides, by Remark \ref{harnack-1}, we only need to consider positive weak solution.

First, we use the method of moving spheres to determine the boundary value $u(y, 0)$.
\begin{proposition}\label{Prop-ms-1}Assume that  $\alpha>0, \ \beta>-1,
\   \frac{n-1}{n+1}\beta\leq \alpha<\beta+2$, and  $u \in   {\cal D}_{\alpha}^{1,2}(\mathbb{R}^{n+1}_+)$ is a positive  weak solution to \eqref{genequ-1}. Then $u$  takes the form of
	$$u(y,0)=k\Big(\frac{A}{A^2+|y-y^o|^2}\Big)^{\frac{n+\alpha-1}{2}}$$
on $\partial \mathbb{R}^{n+1}_+$, where  $k, A>0$ are some constants and $y^o\in\mathbb{R}^{n}$.
\end{proposition}

For any fixed $b\in\partial\mathbb{R}^{n+1}_+$ and $\lambda>0$, set
$$u_b(y,t)=u((y,t)+b),\quad (y,t)\in\overline{\mathbb{R}^{n+1}_+},$$
$$u_{\lambda,b}(y,t)=\frac{\lambda^{n+\alpha-1}}{|(y,t)|^{n+\alpha-1}}u_b(\frac{\lambda^2 (y,t)}{|(y,t)|^2}),\quad (y,t)\in\overline{\mathbb{R}^{n+1}_+}\backslash\{0\},$$
$$w_{\lambda,b}(y,t)=u_b(y,t) -u_{\lambda,b}(y,t),\quad (y,t)\in\overline{\mathbb{R}^{n+1}_+}\backslash\{0\}.$$
Since $u \in   {\cal D}_{\alpha}^{1,2}(\mathbb{R}^{n+1}_+)$ is a positive  weak solution to \eqref{genequ-1}, by the proof of Lemma \ref{weak solution} in the Appendix, we know that $u_{\lambda,b} \in   {\cal D}_{\alpha}^{1,2}(\mathbb{R}^{n+1}_+)$ satisfies
\begin{equation*}
\int_{\mathbb{R}^{n+1}_+ }t^{\alpha} \nabla u_{\lambda,b} \cdot \nabla \phi dy dt
=\int_{\mathbb{R}^{n+1}_+ }t^{\beta} u_{\lambda,b}^{p^*-1}\phi dydt,
\end{equation*}
for any $\phi\in   {\cal D}_{\alpha}^{1,2}(\mathbb{R}^{n+1}_+)$ with $supp \phi\subset\overline{\mathbb{R}^{n+1}_+}\backslash\{0\}$.
Therefore, for such a test function $\phi$, $w_{\lambda,b}$ satisfies
\begin{equation}\label{weak-1-2}
\int_{\mathbb{R}^{n+1}_+ }t^{\alpha} \nabla w_{\lambda,b} \cdot \nabla \phi dy dt
=(p^*-1)\int_{\mathbb{R}^{n+1}_+ }t^{\beta} \varphi^{p^*-2}w_{\lambda,b}\phi dydt,
\end{equation}
where $\varphi(y,t)=s(y,t)v_b(y,t)+(1-s(y,t))v_{\lambda,b}(y,t)$ for some $s(y,t)\in [0,1]$.

Let $$\Sigma_{\lambda,b}=\{(y,t)\in\overline{\mathbb{R}^{n+1}_+}\backslash B_\lambda(0): w_{\lambda,b}(
y,t)>0\}.$$
Define
$w_{\lambda,b}^+=\max\{w_{\lambda,b},0\}$ in $\overline{\mathbb{R}^{n+1}_+}\backslash B_\lambda(0)$, and extend it to  the rest of $\overline{\mathbb{R}^{n+1}_+}$ with value zero. For simplicity, we still denote $w_{\lambda,b}^+$ as the new function after extension. It is easy to see that $w_{\lambda,b}^+ \in   {\cal D}_{\alpha}^{1,2}(\mathbb{R}^{n+1}_+)$.

\medskip

\textbf{Claim 1.} For $\lambda$ large enough, $w_{\lambda,b}\leq 0$ in $\overline{\mathbb{R}^{n+1}_+}\backslash B_\lambda(0)$.

\begin{proof}	
	Taking $w_{\lambda,b}^+$ as the  test function in \eqref{weak-1-2},
 we have
	$$	\int_{\Sigma_{\lambda,b}}t^\alpha|\nabla w_{\lambda,b}^+|^2	dydt =(p^*-1)\int_{\Sigma_{\lambda,b}}t^{\beta}\varphi^{p^*-2}|w_{\lambda,b}^+|^2dydt.$$
Since $0<u_{\lambda,b}\leq \varphi\leq u_b$ in $\Sigma_{\lambda,b}$, we have
\begin{align}\label{claim-1}
\int_{\Sigma_{\lambda,b}}t^\alpha|\nabla w_{\lambda,b}^+|^2dydt	\leq&(p^*-1)\int_{\Sigma_{\lambda,b}}t^{\beta}u_b^{p^*-2}|w_{\lambda,b}^+|^2dydt\nonumber\\			\leq&(p^*-1)\big(\int_{\Sigma_{\lambda,b}}t^{\beta}u_b^{p^*}dydt\big)^\frac{p^*-2}{p^*}
\big(\int_{\Sigma_{\lambda,b}}t^{\beta}|w_{\lambda,b}^+|^{p^*}dydt\big)^{\frac{2}{p^*}}\nonumber\\	\leq&(p^*-1)S_{n+1, \alpha,\beta}^{-1}\big(\int_{\Sigma_{\lambda,b}}t^{\beta}u_b^{p^*}dydt\big)^\frac{p^*-2}{p^*}
\int_{\Sigma_{\lambda,b}}t^\alpha|\nabla w_{\lambda,b}^+|^2dydt.
\end{align}
By Lemma \ref{density}, we know that $u_b\in L^{p^*}_\beta(\mathbb{R}^{n+1}_+)$. Since $ p^*>2$ and $\Sigma_{\lambda,b}\subset\overline{\mathbb{R}^{n+1}_+}\backslash B_{\lambda}(0)$, we have
\begin{eqnarray*}
(p^*-1)S_{n+1, \alpha,\beta}^{-1}\big(\int_{\Sigma_{\lambda,b}}t^{\beta}u_b^{p^*}dydt\big)^\frac{p^*-2}{p^*}
<\frac{1}{2}
\end{eqnarray*}for sufficiently large $\lambda$.
Bringing this to \eqref{claim-1}, we obtain
$$\int_{\Sigma_{\lambda,b}}t^\alpha|\nabla w_{\lambda,b}^+|^2=0.$$
This implies $w_{\lambda,b}\leq 0$ in $\overline{\mathbb{R}^{n+1}_+}\backslash B_\lambda(0)$.
\end{proof}
\bigskip

Now, we define $\lambda_b=\inf\{\lambda>0: \forall \mu>\lambda,w_{\mu,b}\leq 0 \text{ in } \overline{\mathbb{R}^{n+1}_+}\backslash B_{\mu}(0)\}$.

\medskip

\textbf{Claim 2.} There exists $b\in\partial \mathbb{R}^{n+1}_+$, such that $\lambda_b>0$.

\begin{proof}	If for all $b\in \partial\mathbb{R}^{n+1}_+$, $\lambda_{b}=0$,  we have for all $b\in\partial\mathbb{R}^{n+1}_+$ and $\lambda>0$,
	$$u_b(y,t)\leq\frac{\lambda^{n+\alpha-1}}{|(y,t)|^{n+\alpha-1}}u_b(\frac{\lambda^2(y,t)}{|(y,t)|^2}),\quad (y,t)\in\overline{\mathbb{R}^{n+1}_+}\backslash B_\lambda(0).$$
It follows from the first Li-Zhu lemma (see, for example, Dou and Zhu \cite [Lemma 3.7] {DZ15} ) that $u$ only depends on $t$. Due to $u>0$, we have $\int_{\mathbb{R}^{n+1}_+}t^\beta u^{p^*}dydt=\infty$, which implies that $u\notin{\cal D}^{1,2}_{\alpha}(\mathbb{R}^{n+1}_+)$ by Lemma \ref{density} in the Appendix, contradiction.
\end{proof}

\bigskip

\textbf{Claim 3.} Suppose $\lambda_b>0$ for some $b\in\partial\mathbb{R}^{n+1}_+$,  then we have $w_{\lambda_b,b}\equiv 0$ in $\overline{\mathbb{R}^{n+1}_+}\backslash\{0\}$.

\begin{proof}
	First, by the continuity of $w_{\lambda_b,b}$, we know $w_{\lambda_b,b}\leq0$ in $\overline{\mathbb{R}^{n+1}_+}\backslash B_{\lambda_b}(0)$.  It then follows from \eqref{weak-1-2} that
	\begin{align*}
		-div(t^\alpha \nabla w_{\lambda_b,b})=(p^*-1)t^\beta \varphi^{p^*-2}w_{\lambda_b,b}\leq 0,\quad \text{in } \mathbb{R}^{n+1}_+\backslash\overline{ B_{\lambda_b}(0)}.
\end{align*}
We prove Claim 3 by contradiction. Assume $w_{\lambda_b,b}\not\equiv0$ in $\overline{\mathbb{R}^{n+1}_+}\backslash B_{\lambda_b}(0)$. For any open subset $U\subset\subset \mathbb{R}^{n+1}_+\backslash\overline{ B_{\lambda_b}(0)}$, since the divergent operator $div(t^\alpha\nabla)$ is uniformly elliptic in $U$, it holds  $w_{\lambda_b,b}<0$ in $U$ via the maximum principle, which implies that $w_{\lambda_b,b}<0$ in $\mathbb{R}^{n+1}_+\backslash\overline{ B_{\lambda_b}(0)}$.

Since $u_b\in L^{p^*}_{\beta}(\overline{\mathbb{R}^{n+1}_+})$, there is $R>0$ large enough such that
	$$(p^*-1) S^{-1}_{n+1, \alpha,\beta}\big(\int_{\mathbb{R}^{n+1}_+\backslash B_{R}(0)}t^{\beta}u_b^{p^*}\big)^\frac{p^*-2}{p^*}<\frac{1}{4}.$$
Take $\delta_1>0$ small enough, such that
$$ (p^*-1)S^{-1}_{n+1, \alpha,\beta}\big(\int_{\Omega_{\delta_1}}t^{\beta}u_b^{p^*}\big)^\frac{p^*-2}{p^*}
<\frac{1}{4},$$
where $\Omega_{\delta_1}=\big((B_R^+(0)\backslash B_{\lambda_b+\delta_1}^+(0))\cap\{(y,t): 0<t<\delta_1\}\big)\cup\big(B_{\lambda_b+\delta_1}^+(0)\backslash B_{\lambda_b-\delta_1}^+(0)\big)$ (See Figure 1).

\begin{figure}[!htbp]
\center
\includegraphics[height=2.0in,width=3.7in]{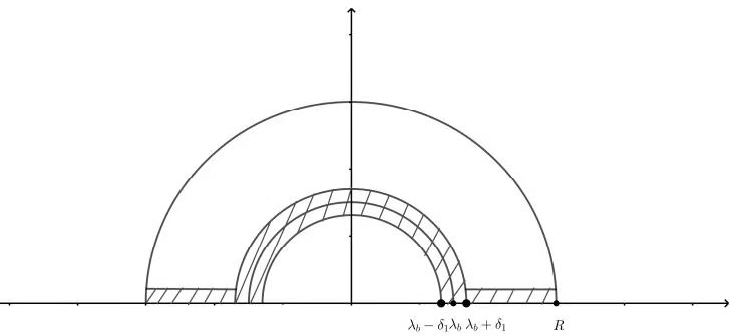}
\vspace{1em}
\caption{ Domain of $\Omega_{\delta_1}$}

\end{figure}

Since $w_{\lambda_b,b}<0$ in compact set $\overline{\big(B_R^+(0)\backslash B_{\lambda_b+\delta_1}\big)\cap\{(y,t): t>\delta_1\}}$, we have $$w_{\lambda_b,b}<-K<0 \text{ in } \overline{\big(B_R^+(0)\backslash B_{\lambda_b+\delta_1}\big)\cap\{(y,t): t>\delta_1\}}.$$ Due to the continuity, there exists $\delta_2$ small enough, such that $0<\delta_2<\delta_1$ and for any $\lambda\in[\lambda_b-\delta_2,\lambda_b]$,
$$w_{\lambda,b}<-\frac{K}{2}<0 \text{ in } \overline{\big(B_R^+(0)\backslash B_{\lambda_b+\delta_1}\big)\cap\{(y,t): t>\delta_1\}},$$
 then $\Sigma_{\lambda,b}\subset (\Sigma_{\lambda,b}\cap B_R^c(0))\cup \overline{\Omega_{\delta_1}}$.
Therefore, for $\lambda\in [\lambda_b-\delta_2,\lambda_b],$ we have
\begin{eqnarray*}
& &(p^*-1)S_{n+1, \alpha,\beta}^{-1}\big(\int_{\Sigma_{\lambda,b}}t^{\beta}u_b^{p^*}\big)^\frac{p^*-2}{p^*}\\
&\leq&(p^*-1)S_{n+1, \alpha,\beta}^{-1}\big(\int_{\Sigma_{\lambda,b}\cap B_R^c(0)}t^{\beta}u_b^{p^*}\big)^\frac{p^*-2}{p^*}
+(p^*-1)S_{n+1, \alpha,\beta}^{-1}\big(\int_{\Omega_{\delta_1}}t^{\beta}u_b^{p^*}\big)^\frac{p^*-2}{p^*}\\
&<&\frac{1}{2}.
\end{eqnarray*}
	Similar to the proof of  Claim 1, we have for $\lambda\in[\lambda_b-\delta_2,\lambda_b]$, $w_{\lambda,b}\leq 0$ in $\overline{\mathbb{R}^{n+1}_+}\backslash B_{\lambda}(0)$,  which is contradictory to the definition of $\lambda_b$. Hence, $w_{\lambda_b,b}\equiv 0$ in $\overline{\mathbb{R}^{n+1}_+}\backslash\{0\}$.
\end{proof}
\bigskip

\textbf{Claim 4.} $\lambda_{b}>0$ for all $b\in \partial \mathbb{R}^{n+1}_+$.

\begin{proof}
	It follows from Claim 2 and Claim 3 that there exists some $\bar{b}\in\partial\mathbb{R}^{n+1}_+$ such that $\lambda_{\bar{b}}>0$ and $w_{\lambda_{\bar{b}},\bar{b}}\equiv 0$ in $\overline{\mathbb{R}^{n+1}_+}\backslash
	\{0\}$. That is
	$$u_{\bar{b}}(y,t)=\frac{\lambda_{\bar{b}}^{n+\alpha-1}}{|(y,t)|^{n+\alpha-1}}u_{\bar{b}}(\frac{\lambda_{\bar{b}}^2(y,t)}{|(y,t)|^2}), \quad\forall (y,t)\in \overline{\mathbb{R}^{n+1}_+}\backslash\{0\}.$$
	Clearly, $$\lim_{|(y,t)|\rightarrow\infty}|(y,t)|^{n+\alpha-1}u_{\bar{b}}(y,t)=\lambda_{\bar{b}}^{n+\alpha-1}u_{\bar{b}}(0,0),$$
	that is
	\begin{equation}\label{decay}
	\lim_{|(y,t)|\rightarrow\infty}|(y,t)|^{n+\alpha-1}u(y,t)=\lambda_{\bar{b}}^{n+\alpha-1}u(\bar{b}).
	\end{equation}
	Suppose the contrary to Claim 4 for some $b\in\partial\mathbb{R}^{n+1}_+$, then
	$$ | (y,t) |^{n+\alpha-1}u_{ b } ( y ,t) \leq  \lambda^{n+\alpha-1} u_{ b } \left( \frac{ \lambda^{ 2 } (y,t) }{ | (y,t) |^{ 2 } } \right) , \quad \forall \lambda > 0 , (y,t)\in \overline{\mathbb{R}^{n+1}_+}\backslash B_{\lambda}(0).$$
	Fixing $\lambda>0$ in the above and sending $|(y,t)|$ to $\infty$, by \eqref{decay}, we have
	$$u_{\bar{b}}(0)\leq\frac{\lambda^{n+\alpha-1}}{\lambda_{\bar{b}}^{n+\alpha-1}}u(b).$$
	Sending $\lambda$ to $0$, we have $u(\bar{b})\leq 0,$
	 which yields a contradiction.
\end{proof}

\bigskip
The second Li-Zhu Lemma  in \cite[Lemma 2.5]{LZ1995} and its generalization for continuous functions due to Li and Nirenberg \cite[lemma 5.8]{Li2004} are stated as follows.

\begin{lemma}\label{LZ-2} For a given parameter $\mu\in \mathbb{R}$,
 if $f\in C(\mathbb{R}^{n})$ $(n\geq 1)$ satisfies: for all $b\in\mathbb{R}^{n}$, there exists $\lambda_b>0$ such that
	\begin{equation}\label{kel}
	f ( x + b ) = \big(\frac{ \lambda _ { b } }{ | x |}\big)^{\mu} f ( \frac{ \lambda_{ b }^{ 2 } x}{ | x |^{ 2 } } + b ) , \quad \forall x\in \mathbb { R }^{ n  } \backslash \{ 0 \}.
	\end{equation}
	Then for some $a\geq 0, d>0, x_0\in \mathbb{R}^{n},$
	$$f(x)=(\frac{a}{|x-x_0|^2+d})^{\frac{\mu}{2}},\quad \forall x\in\mathbb{R}^{n},
	$$ or
	$$f(x)=-(\frac{a}{|x-x_0|^2+d})^{\frac{\mu}{2}},\quad \forall x\in\mathbb{R}^{n}.$$
\end{lemma}

\textbf{Proof of Proposition \ref{Prop-ms-1}.}  From  {Claim 4}, we know $\lambda_{b}>0$ for all $b\in \partial \mathbb{R}^{n+1}_+$. Then it follows from  {Claim 3 } that
\begin{equation}\label{id}
u_b(y,t)=\frac{\lambda_b^{n+\alpha-1}}{|(y,t)|^{n+\alpha-1}}u_b(\frac{\lambda_b^2 (y,t)}{|(y,t)|^2}),\quad\forall (y,t)\in \overline{\mathbb{R}^{n+1}_+}\backslash\{0\}.
\end{equation}
 By Lemma \ref{LZ-2} and $u>0$, we have
$$u(y,0)=k\Big(\frac{A}{A^2+|y-y^{o}|^2}\Big)^{\frac{n+\alpha-1}{2}},\quad  y\in \mathbb{R}^{n},$$
 for some $k, A>0$ and $ y^{o}\in \mathbb{R}^{n}.$
 \hfill$\Box$

\bigskip

Without loss of generality, we assume that $k=1, A=1$ and $y^o=0$. By \eqref{decay}, we have
$$1=\lim_{|(y,t)|\rightarrow \infty}|(y,t)|^{n+\alpha-1}u(y,t)=\lambda_b^{n+\alpha-1}u(b)=\lambda_b^{n+\alpha-1}\Big(\frac{1}{1+|b|^2}\Big)^{\frac{n+\alpha-1}{2}},$$
which implies $$\lambda_b=\sqrt{1+|b|^2}.$$
For any $(y,t)\in\overline{\mathbb{R}^{n+1}_+}\backslash\{b\}$,  by \eqref{id} we have
\begin{equation}\label{sym}
u(y,t)=\frac{(1+|b|^2)^{\frac{
			n+\alpha-1}{2}}}{|(y,t)-b|^{n+\alpha-1}}u(b+\frac{(1+|b|^2)((y,t)-b)}{|(y,t)-b|^2}).
\end{equation}

Set $e_{n+1}=(0,1)$, and define the Kelvin transformation with respect to $ \partial B_{1}( -{e_{n+1}})$ as
\begin{equation}\label{transform-1}
x:=(x', x_{n+1})=-e_{n+1}+\frac{(y,t)+e_{n+1}}{|(y,t)+e_{n+1}|^2} .
\end{equation} This projects $\mathbb{R}^{n+1}_+$ to $B:= B_{\frac12}(-\frac{e_{n+1}}{2}),$
and $\partial \mathbb{R}^{n+1}_+$ to $\partial B$. Set
\begin{equation}\label{transform-2}
\psi(x)=\frac{1}{|x+e_{n+1}|^{n+\alpha-1}}u(-e_{n+1}+\frac{x+e_{n+1}}{|x+e_{n+1}|^2}),\quad x\in B.
\end{equation}
From the boundary value and  $C^2$ regularity  of $u$, and  equation \eqref{u-equ}
 we know that  $\psi$ is a positive, interior $C^2$ smooth function, and satisfies
\begin{equation}\label{ball-1}
\begin{cases}
\Delta \psi-\frac{2\alpha\nabla \psi\cdot(x +\frac{e_{n+1}}{2})}{\frac{1}{4}-|x+\frac{e_{n+1}}{2}|^2}-\frac{\alpha(n+\alpha-1)}{\frac{1}{4}-|x+\frac{e_{n+1}}{2}|^2}\psi
=-C\big({\frac{1}{4}-|x+\frac{e_{n+1}}{2}|^2}\big)^{\beta-\alpha}\psi^{\frac{n+2\beta-\alpha+3}{n+\alpha-1}},\;&\text{in } B,\\
\psi=1,&\text{on } \partial B
\end{cases}
\end{equation}
for some unknown $C>0$.

Further, the equality \eqref{sym} implies the following symmetric result.

\begin{lemma}\label{rigidity}
 $\psi$ is radially symmetric about the center $-\frac{e_{n+1}}{2}$.
 \end{lemma}

\begin{proof}
	Combining \eqref{sym}, \eqref{transform-1} with \eqref{transform-2}, we have
\begin{equation*}
\frac{1}{|(y,t)+e_{n+1}|^{n+\alpha-1}}\psi(x)=\frac{(1+|b|^2)^{\frac{n+\alpha-1}{2}}}{|(y,t)-b|^{n+\alpha-1}}
\frac{1}{|b+e_{n+1}+\frac{(1+|b|^2)((y,t)-b)}{|(y,t)-b|^2}|^{n+\alpha-1}}\psi(x^b),
\end{equation*}
where $$x^b=-e_{n+1}+\frac{\frac{(1+|b|^2)((y,t)-b)}{|(y,t)-b|^2}+b+e_{n+1}}{|\frac{(1+|b|^2)((y,t)-b)}{|(y,t)-b|^2}+b+e_{n+1}|^2}$$
 is the symmetry point of $b+\frac{(1+|b|^2)((y,t)-b)}{|(y,t)-b|^2}$ with respect to the sphere $\partial B_1(-e_{n+1})$.
A direct calculation yields
$$|x+\frac{e_{n+1}}{2}|^2=|x^b+\frac{e_{n+1}}{2}|^2=\frac{1}{4}-\frac{t}{|(y,t)+e_{n+1}|^2},$$ and
$$\frac{\psi(x)}{\psi(x^b)}=\big(\frac{(1+b^2)|(y,b)+e_{n+1}|^2}{|(y,t)-b|^2\cdot|b+e_{n+1}+\frac{(1+|b|^2)((y,t)-b)}{|(y,t)-b|^2}|^2}\big)^{\frac{n+\alpha-1}{2}}=1.$$
Since $b\mapsto x^b$ is a one-one and onto map from $\partial\mathbb{R}^{n+1}_+$  to $\partial B_{\sqrt{\frac{1}{4}-\frac{t}{|(y,t)+e_{n+1}|^2}}}(-\frac{e_{n+1}}{2})$, we have that  $\psi $ is radially symmetric about the center $-\frac{e_{n+1}}{2}$.
\end{proof}

\begin{remark}
	For any $b\in\partial\mathbb{R}^{n+1}_+$, the transformation \eqref{transform-1} maps the sphere $\partial B_{\sqrt{1+|b|^2}}(b)$ onto some hyperplane passing through $-\frac{e_{n+1}}{2}$ (see also Li and Zhang \cite{LZ2003}).
\end{remark}
\medskip

Write $r=|x+\frac{e_{n+1}}{2}|$ and
$
\psi(r)=\psi(x),\ x\in B.
$  Up to a positive constant multiplier, $\psi$ satisfies the following ODE
\begin{equation}\label{ode-2}
\begin{cases}
\psi''+(\frac{n}{r}-\frac{2\alpha r}{\frac{1}{4}-r^2})\psi'-\frac{\alpha(n-1+\alpha)}{\frac{1}{4}-r^2}\psi
=-\big({\frac{1}{4}-r^2}\big)^{\beta-\alpha}\psi^{p^*-1},\; r\in (0,\frac{1}{2}),\\
\psi(\frac{1}{2})=K
\end{cases}
\end{equation}
for some unknown constant $K>0$. And we also have the  boundary condition
 \begin{align}\label{boundary-alpha}
		\lim_{r\rightarrow(\frac12)^-}(\frac14-r^2)^\alpha\psi'(r)=0
		\end{align}
  for $\alpha>0$  (see Lemma \ref{n-bd} in the Appendix).
Summarizing the above analysis, we shall consider $0<\psi\in C^2[0,\frac12)\cap C^0[0,\frac12]$ satisfying the following ODE
\begin{equation}\label{eq:uni-main-eq}
\begin{cases}
\psi''(r)+(\frac{n}{r}-\frac{2\alpha r}{\frac{1}{4}-r^2})\psi'(r)-\frac{\alpha(n+\alpha-1)}{\frac{1}{4}-r^2}\psi
=- ({\frac{1}{4}-r^2})^{\beta-\alpha} \psi^{p^*-1},  \ r\in (0,\frac12),\\
\psi(\frac{1}{2})=K, \ \ \ \psi'(0)=0,\ \ \ \lim_{r\rightarrow(\frac12)^-}(\frac14-r^2)^\alpha\psi'(r)=0,\\
\end{cases}
\end{equation}
for some $K>0$.

The proof of the uniqueness result will follow from next  two propositions.

\begin{proposition}\label{uniq}
	Suppose $0< \alpha<\beta+2$ and $\beta>-1$. Then for any $K>0$, there is at most one solution $\psi\in C^2[0,\frac{1}{2})\cap C^0[0,\frac{1}{2}]$ to equation \eqref{eq:uni-main-eq}.
\end{proposition}

\begin{proof}
Let $$c(r)=r^n(\frac{1}{4}-r^2)^\alpha,$$
and $\psi$ be any solution of \eqref{eq:uni-main-eq}, then we have
	\begin{eqnarray*}
		(c(r)\psi'(r))'&=&c(r)\big(\psi''(r)+(\frac{n}{r}-\frac{2\alpha r}{\frac{1}{4}-r^2})\psi'(r)\big)\\
		&=&c(r)\big\{\frac{\alpha(n+\alpha-1)}{\frac{1}{4}-r^2}\psi(r)
		-({\frac{1}{4}-r^2})^{\beta-\alpha}\psi(r)^{p^*-1}\big\}\\
		&=&r^n\big\{\alpha(n+\alpha-1)(\frac{1}{4}-r^2)^{\alpha-1}\psi(r)
		-(\frac{1}{4}-r^2)^{\beta}\psi(r)^{p^*-1}\big\}.
	\end{eqnarray*}
	For any fixed $K>0$, assume that $\psi_1,\psi_2\in C^2[0,\frac{1}{2})\cap C^0[0,\frac{1}{2}]$ are two solutions to \eqref{eq:uni-main-eq}.
Write $$\omega=\psi_1-\psi_2.$$  It is easy to verify
	\begin{align}\label{w-equ}
		(c(r)\omega'(r))'=r^n\big\{\alpha(n+\alpha-1)(\frac{1}{4}-r^2)^{\alpha-1}\omega(r)
		-\big({\frac{1}{4}-r^2}\big)^{\beta}\omega(r)h(r)\big\}
	\end{align}
for $0<r<\frac12$, where $h(r)=(p^*-1)\big(\theta(r)\psi_1(r)+(1-\theta(r))\psi_2(r)\big)^{\frac{2(\beta+2-\alpha)}{n+\alpha-1}}$ is bounded on $[0,\frac{1}{2}]$. Since $$\lim_{r\rightarrow (\frac{1}{2})^-}(\frac{1}{4}-r^2)^\alpha\frac{\partial \omega}{\partial r}=0$$
for $\alpha>0$,
integrating \eqref{w-equ}, we have
$$c(r)\omega'(r)=-\int_{r}^{\frac{1}{2}}s^n\big\{\alpha(n+\alpha-1)(\frac{1}{4}-s^2)^{\alpha-1}
-(\frac{1}{4}-s^2)^{\beta}h(s)\big\}\omega(s)ds$$
 for $0<r<\frac12$. That is,
$$\omega'(r)=-\frac{1}{r^n(\frac{1}{4}-r^2)^\alpha}\int_{r}^{\frac{1}{2}}s^n\big\{\alpha(n+\alpha-1)(\frac{1}{4}-s^2)^{\alpha-1}
	-(\frac{1}{4}-s^2)^{\beta}h(s)\big\}\omega(s)ds$$
for $0<r<\frac12$. Integrating the above on $(r,\frac12)$  and noting $\omega(\frac{1}{2})=0$, we have
	\begin{eqnarray*}
		\omega(r)=\int_r^\frac{1}{2}\frac{1}{\tau^n(\frac{1}{4}-\tau^2)^\alpha}\int_{\tau}^{\frac{1}{2}}s^n\big\{\alpha(n+\alpha-1)(\frac{1}{4}-s^2)^{\alpha-1}
		-(\frac{1}{4}-s^2)^{\beta}h(s)\big\}\omega(s)dsd\tau.
	\end{eqnarray*}
Let $\varepsilon\in(0,\frac14)$. For $0< \alpha<\beta+2$ and $\beta>-1$,  we have
	\begin{align*}
	&	\sup_{r\in [\frac{1}{2}-\varepsilon,\frac{1}{2}]}|\omega(r)|\\
	\leq&C\sup_{r\in [\frac{1}{2}-\varepsilon,\frac{1}{2}]}|\omega(r)|\sup_{r\in [\frac{1}{2}-\varepsilon,\frac{1}{2}]}\int_r^\frac{1}{2}\frac{1}{(\frac{1}{2}-\tau)^\alpha}\int_{\tau}^{\frac{1}{2}}\big((\frac{1}{2}-s)^{\alpha-1}+(\frac{1}{2}-s)^\beta\big)dsd\tau\\
	\leq&C\sup_{r\in [\frac{1}{2}-\varepsilon,\frac{1}{2}]}|\omega(r)|\sup_{r\in [\frac{1}{2}-\varepsilon,\frac{1}{2}]}\big((\frac{1}{2}-r)+(\frac{1}{2}-r)^{\beta+2-\alpha}\big)\\
	\leq&C\big(\varepsilon+\varepsilon^{\beta+2-\alpha}\big)\sup_{r\in [\frac{1}{2}-\varepsilon,\frac{1}{2}]}|\omega(r)|,
	\end{align*}
	where $C$ is a constant independent of $\varepsilon$. Choose $\varepsilon>0$ small enough, such that $C\big(\varepsilon+\varepsilon^{\beta+2-\alpha}\big)<1$, then $$\omega(r)=0,\quad \forall r\in [\frac{1}{2}-\varepsilon,\frac{1}{2}].$$
We claim that $\omega\equiv 0$ on $[0,\frac12]$. Otherwise, take
	$$r_0=\inf\{r\in[0,\frac{1}{2}]:\, \omega(s)=0,\;\forall s\in[r,\frac{1}{2}]\}.$$
Obviously, $0<r_0<\frac12-\varepsilon$ and $\omega(r_0)=0$. Analyzing as before, we can show that there is a small $\varepsilon_0>0$, such that $\omega(r)=0$ on $[r_0-\varepsilon_0,r_0]$, which contradicts  the definition of $r_0$. Hence, we obtain the uniqueness.
	
\end{proof}

\begin{proposition}\label{thm:uniqueness-main} Suppose $\alpha>0,\ \beta>-1,\  \frac{n-1}{n+1}\beta\le \alpha<\beta+2$. If $n=1$, we also assume
\[\frac{1-(1-\alpha)^2}{4}\leq \frac{\alpha(2+\beta)}{(\alpha+\beta+2)^2}.\]
Then there exists at most one $K$ such that \eqref{eq:uni-main-eq} has a solution $0<\psi\in C^2[0,\frac12)\cap C^0[0,\frac12]$.
\end{proposition}

To prove Proposition \ref{thm:uniqueness-main}, we shall use some known results as follows.

Let  $w(r)=(\frac{1-r^2}4)^{\frac{n+\alpha-1}{2}}\psi(\frac r 2)$. For $r \in [0, 1)$,  $w(r)$ satisfies

\begin{align}\label{5.12}
    \left[\frac{1-r^2}{2}\right]^2&(w''+\frac{n}{r}w')+(n-1)\frac{1-r^2}{2}rw'+\frac{n^2-(1-\alpha)^2}{4}w=-w^{p^*-1}.
\end{align}
We also have $w'(0)=0$.

Now we view $w$ as a positive radial function defined  on the unit disc $B^{n+1}$. The above equation actually can be interpreted in hyperbolic space. That is:  if $B^{n+1}=\mathbb{H}^{n+1}$ is equipped with the hyperbolic  metric $4/(1-|x|^2)|dx|^2$,  the above equation is equivalent to  (for example, see \cite[pg. 666]{Mancini}):
\begin{align}\label{hy-1}
    \Delta_{\mathbb{H}} w+\frac{n^2-(1-\alpha)^2}{4}w=-w^{p^*-1}.
\end{align}
Such an equation was already studied in Mancini and Sandeep\cite{Mancini},  in which  Kwong and Li's method in  \cite{Kwong92} was used. We will borrow some of their arguments to establish our uniqueness result.

Define
\begin{align}\label{v-psi}
v(s):=w(\tanh \frac s2)=\frac{e^\frac{n+\alpha-1}{2}s}{(e^s+1)^{n+\alpha-1}}\psi(\frac{\tanh\frac{s}{2}}{2}),\ s\geq 0,
\end{align}  and
$$q(s)= (\sinh s)^n,\ s\geq 0.$$
Equation \eqref{5.12} can be written as
\begin{equation}\label{hy-2}
v''+\frac{n}{\tanh s} v' +\frac{n^2-(1-\alpha)^2}4 v+v^{p^*-1}=0,\ \ \ \ v'(0)=0,
\end{equation}
or
\begin{equation}\label{q}
	(qv')'+\frac{n^2-(1-\alpha)^2}4 qv+qv^{p^*-1}=0,\ \ \ \ v'(0)=0.
\end{equation}
Noting that $\psi(r)$ is bounded, the asymptotic behavior of $v(s)$ at infinity  is
\begin{equation}\label{hy-3}
\lim_{s \to \infty}  {v(s)}\cdot {e^{\frac{(n+\alpha-1)s}2}} =\psi(\frac{1}{2})=K>0.
\end{equation}

The uniqueness result in Mancini and Sandeep \cite{Mancini} is as follows:
\begin{proposition}\label{mancini-thm}
	Assume $n\geq 2,\ 2<p\leq \frac{2(n+1)}{n-1},\ \lambda\leq \frac{n^2}{4}$ or $n= 1,\ p>2,\ \lambda\leq \frac{2p}{(p+2)^2}$, then the equation
	\begin{align}\label{mancini-equ}
	v''+\frac{n}{\tanh s}v'+\lambda v+v^{p-1}=0,\quad s>0, \quad v'(0)=0
	\end{align}
	has at most one positive solution satisfying
	\begin{equation}\label{bound}
	\begin{cases}
	\int_{0}^{\infty}q[|v'|^2+|v|^2]ds<\infty, &\ \text{for }\lambda< \frac{n^2}{4},\\
	\int_{0}^{\infty}q[(v'+\frac{n}{2}(\tanh\frac{s}{2})v)^2+\frac{nv^2}{(2\cosh\frac{s}{2})^2}]ds<\infty,&\ \text{for }\lambda= \frac{n^2}{4}.
	\end{cases}
	\end{equation}
\end{proposition}

Take $\lambda=\frac{n^2-(\alpha-1)^2}{4}$ and $p=p^*$ in Proposition \ref{mancini-thm}, then the case $\alpha\geq 1$ in Proposition \ref{thm:uniqueness-main} can be obtained by the following Lemma.

\begin{lemma}\label{alpha>1}
	Assume $n\geq 1,\ \alpha\geq 1,\ \beta>-1$,  $\psi\in C^2[0,\frac12)\cap C^0[0,\frac12]$ is a positive solution to  \eqref{eq:uni-main-eq}, and  $v$ satisfies \eqref{v-psi}. Then $v$ is a positive solution to \eqref{hy-2} satisfying \eqref{bound}  with $\lambda=\frac{n^2-(\alpha-1)^2}{4}$.
\end{lemma}

\begin{proof}  According to the definition of $v$ in \eqref{v-psi}, it is easy to check that $v$ satisfies \eqref{hy-2}.
	By \eqref{hy-3}, we know that $q|v|^2=O(e^{(1-\alpha)s})$ at infinity, i.e.,
	\begin{align*}
	\lim_{s\to \infty}e^{-(1-\alpha)s}q|v|^2=\frac{K^2}{2^n}>0,
	\end{align*}
and then
	\begin{align}\label{bd-1}
	\int_0^\infty q|v|^2ds<\infty \ \text{ for } \alpha>1, \quad \int_0^\infty q|v|^2ds=\infty\ \text{ for } 0<\alpha\leq 1.
	\end{align}
By \eqref{v-psi},
we have
\begin{equation}\label{v-der}
	v'(s)=-\frac{(n+\alpha-1)e^{\frac{n+\alpha-1}{2}s}(e^s-1)}{2(e^s+1)^{n+\alpha}}\psi(\frac{\tanh\frac{s}{2}}{2})+\frac{e^{\frac{n+\alpha+1}{2}s}}{(e^s+1)^{n+\alpha+1}}\psi'(\frac{\tanh\frac{s}{2}}{2}).
\end{equation}
Write
\begin{eqnarray*}
\int_{0}^\infty q|v'|^2ds\leq 2(\uppercase\expandafter{\romannumeral1}+\uppercase\expandafter{\romannumeral2}),
\end{eqnarray*}
where
\begin{eqnarray*}
	\uppercase\expandafter{\romannumeral1}=\int_0^\infty \frac{(n+\alpha-1)^2}{4}\sinh^n s\frac{e^{(n+\alpha-1)s}(e^s-1)^2}{(e^s+1)^{2(n+\alpha)}}|\psi(\frac{\tanh\frac{s}{2}}{2})|^2 ds,
	\end{eqnarray*}
and
\begin{eqnarray*}
	\uppercase\expandafter{\romannumeral2}&=&\int_0^\infty \sinh^n s\big|\frac{e^{\frac{n+\alpha+1}{2}s}}{(e^s+1)^{n+\alpha+1}}\psi'(\frac{\tanh\frac{s}{2}}{2})\big|^2 ds\\
	&=&\int_0^{\frac{1}{2}}r^n(\frac14-r^2)^{\alpha}|\psi'(r)|^2dr.
\end{eqnarray*}
Similar to \eqref{bd-1}, we  can check that $\uppercase\expandafter{\romannumeral1}<\infty$ for $\alpha>1$ and $\uppercase\expandafter{\romannumeral1}=\infty$ for $0<\alpha\leq 1$.
To estimate $\uppercase\expandafter{\romannumeral2}$, multiplying $r^n(\frac14-r^2)^\alpha$ in the two sides of \eqref{eq:uni-main-eq}, we obtain
\begin{equation}\label{psi}
[r^n(\frac14-r^2)^\alpha\psi']'-\alpha(n+\alpha-1)r^n(\frac14-r^2)^{\alpha-1}\psi
=-r^n(\frac14-r^2)^\beta\psi^{p^*-1}
\end{equation}for $0<r<\frac12$.
Since $\psi'(0)=0$ and $\lim_{r\rightarrow(\frac12)^-}(\frac14-r^2)^\alpha\psi'(r)=0$, multiplying $\psi$ and integrating from $0$ to $\frac12$ on \eqref{psi}, we have
\begin{equation*}
  \uppercase\expandafter{\romannumeral2}=\int_0^{\frac12} r^n(\frac14-r^2)^\beta\psi^{p^*}dr-\alpha(n+\alpha-1)\int_0^{\frac12}r^n(\frac14-r^2)^{\alpha-1}\psi^2dr<\infty
\end{equation*}since $\beta>-1$, $\alpha>0$  and $\psi$ is bounded.
Hence, we obtian \eqref{bound} for $\lambda<\frac{n^2}{4}$. 

 For $\alpha=1$,  that is, $\lambda=\frac{n^2}{4}$,  we have
$$\int_{0}^{\infty}q(v'+\frac{n}{2}(\tanh\frac{s}{2})v)^2ds=\uppercase\expandafter{\romannumeral2}<\infty,$$
and $\frac{nqv^2}{(2\cosh\frac s2)^2}=O(e^{-s}),\ s\gg 1$. Hence, \eqref{bound} holds.
\end{proof}

\medskip

From the proof of Lemma \ref{alpha>1}, we observe that for $0<\alpha<1$, $v$ does not satisfy \eqref{bound}. On the other hand, by utilizing \eqref{boundary-alpha} and  \eqref{hy-3}, we can still analyze the energy function introduced in \cite{Mancini} and obtain the uniqueness result. In fact, by using \eqref{boundary-alpha} and \eqref{hy-3}, we can also give a shorter proof for the case $\alpha\geq 1$.

Now we recall the energy function $\mathcal{E}_{\hat{v}}(s)$, which was introduced in \cite{Mancini}: set
$$\hat{v}(s)=(\sinh^a s)v(s),\quad s\geq 0,$$ and
\begin{eqnarray*}\label{E}
	\mathcal{E}_{\hat{v}}(s)&=&\frac12(\sinh^b s)(\hat{v}')^2+\frac{|\hat{v}|^{p^*}}{p^*}+\frac12 G\hat{v}^2\\
	&=&\frac{\sinh^{ap^*}s}{2}v^2\big[(\frac{a}{\tanh s}+\frac{v'}{v})^2+\frac{2v^{p^*-2}}{p^*}+A+\frac{B}{\sinh^2 s}\big], \quad s>0,
\end{eqnarray*}
where $$G(s)=A\sinh^b s+B\sinh^{b-2} s,\ \ s> 0,$$
$$a=\frac{2n}{p^*+2}=\frac{n(n+\alpha-1)}{2n+\alpha+\beta},\ \ b=a(p^*-2)=\frac{2n(\beta+2-\alpha)}{2n+\alpha+\beta},\ \ B=\frac{a}{2}(2-ap^*),$$
\begin{align*}
A=&\frac{n^2-(1-\alpha)^2}{4}-\frac{a^2p^*}{2}=\frac{n^2-(1-\alpha)^2}{4}-\frac{n^2(n+\alpha-1))(n+\beta+1)}{(2n+\alpha+\beta)^2}\\
=&
\frac{(\alpha+\beta)(n+\alpha-1)}{4(2n+\alpha+\beta)^2}[4n(1-\alpha)+(\alpha+\beta)(n+1-\alpha)].
\end{align*}
It is easy to verify
\begin{align*}
(\sinh^b s)\hat{v}''+\frac12 [\sinh^b s]'\hat{v}'+G(s)\hat{v}+\hat{v}^{p^*-1}=0,\ \ s>0,
\end{align*}
and
\begin{equation}\label{diff-E}\frac{d}{ds}\mathcal{E}_{\hat{v}}(s)=\frac12 G'\hat{v}^2=\frac{1}{2}[Ab\sinh^2 s+B(b-2)]\hat{v}^2\sinh^{b-3}s \cosh s,\quad s>0.
\end{equation}
For $\alpha>0,\ \beta>-1,\ \frac{n-1}{n+1}\beta\leq \alpha<\beta+2$, we have

\noindent (A1) If $n=1$,  then $ 1<ap^*<2,\  0<b<2,\ B>0.$ It holds
\begin{equation*}
\begin{cases}
&G'(s)<0,\ \ \forall s>0,\quad\text{for } A\leq 0, \\
& G'(s)(s-c)>0,\ \ \exists c>0,\ \forall s>0,\ s\neq c,\quad\text{for } A>0.
\end{cases}
\end{equation*}
(A2) If $n\geq 2$ and $p^*<2^*$,  then $n<ap^*< n+1,\, 0<b< 2,\ B<0.$   It holds
\begin{equation*}
\begin{cases}
&G'(s)>0,\ \ \forall s>0,\quad\text{for }A\geq 0,  \\
&G'(s)(s-c)<0,\ \exists\  c>0,\ \forall s>0,\ s\neq c,\quad\text{for }A<0.
\end{cases}
\end{equation*}
(A3) If  $n\geq 2$ and $p^*=2^*$(i.e. $\beta=\frac{n+1}{n-1}\alpha$),  then $ap^*=n+1,\ b=2,\ \ B<0,\ A=\frac{\alpha(2-\alpha)}{4}$. It holds
\begin{equation*}
\begin{cases}
& G'(s)>0,\quad~\text{for }A> 0,\\
& G'(s)\equiv 0,\quad \text{for } A=0,\\
& G'(s)<0,\quad \forall s>0,\quad \text{for }A< 0.
\end{cases}
\end{equation*}

\medskip
For any $s,\ \sigma,\ S$ satisfying $0\leq s\leq \sigma\leq S$, integrating $\frac{d}{d\tau}\mathcal{E}_{\hat{v}_1}(\tau)-\gamma^2(\sigma)\frac{d}{d\tau}\mathcal{E}_{\hat{v}_2}(\tau)$ on $(s,S)$, and using \eqref{diff-E},  we have
\begin{align}\label{E-difference}
&\big(\mathcal{E}_{\hat{v}_1}(S)-\gamma^2(\sigma)\mathcal{E}_{\hat{v}_2}(S)\big)-\big(\mathcal{E}_{\hat{v}_1}(s)-\gamma^2(\sigma)\mathcal{E}_{\hat{v}_2}(s)\big)\nonumber\\
=&\frac12\int^S_{s}G'(\tau)[\hat{v}^2_1(\tau)-\gamma^2(\sigma)\hat{v}^2_2(\tau)]d\tau\nonumber\\
=&\frac12\int^S_{s}G'(\tau)[\gamma^2(\tau)-\gamma^2(\sigma)]\hat{v}^2_2(\tau)d\tau.
\end{align}
We have the following estimates
\begin{lemma}\label{E-var-S}
	(1) $\lim_{\varepsilon\to 0}\big(\mathcal{E}_{\hat{v}_1}(\varepsilon)-\gamma^2(\varepsilon)\mathcal{E}_{\hat{v}_2}(\varepsilon)\big)=0,$
	and for $i=1, 2$,
	\begin{equation}\label{E-var}
	\lim_{\varepsilon\to 0}\mathcal{E}_{\hat{v}_i}(\varepsilon)=
	\begin{cases}
	\lim_{\varepsilon\to 0}\frac{(a^2+B)v_i^2(\varepsilon)}{2}\sinh^{ap^*-2}\varepsilon=+\infty,&\   n=1,\\
	0 &\  n\geq 2.
	\end{cases}
	\end{equation}
	(2)
	\begin{equation*}
	\lim_{S\to\infty}\big(\mathcal{E}_{\hat{v}_1}(S)-\gamma^2(S)\mathcal{E}_{\hat{v}_2}(S)\big)
=
\begin{cases} \frac{K_1^2(K_1^{p^*-2}-K_2^{p^*-2})}{p^*2^{ap^*}}>0,& \alpha+\beta=0,\\
	+ \infty, & \alpha+\beta<0,\\
	0,&\alpha+\beta>0.
	\end{cases}
	\end{equation*} In particular, for $\alpha> 1$, and  $i=1, 2$,
	\begin{equation}\label{E-S}
	\lim_{S\to\infty}\mathcal{E}_{\hat{v}_i}(S)=\begin{cases}
	0,\ & A<0,\\
	\frac{K_i^2(\alpha-1)^2}{2^{ap^*+1}},\ & A=0,\\
	+\infty ,\ & A>0.
	\end{cases}
	\end{equation}
\end{lemma}
We relegate the proof of this lemma to the end of this section.


\begin{remark}\label{no-solu} For $n\geq 2$ and $p^*=2^*$,
integrating \eqref{diff-E} on $(0,\infty)$,  and using (A3) and Lemma \ref{E-var-S},  we have
\begin{align*}
	0>\frac{1}{2}\int_{0}^\infty G'(s)\hat{v}^2(s) ds=\lim_{S\to\infty}\mathcal{E}_{\hat{v}}(S)-\lim_{\varepsilon\to 0}\mathcal{E}_{\hat{v}}(\varepsilon)=0,&  \quad A<0 \, (\text{i.e. } \alpha>2) \\
\end{align*} and
\begin{align*}
0=\frac{1}{2}\int_{0}^\infty G'(s)\hat{v}^2(s) ds=\lim_{S\to\infty}\mathcal{E}_{\hat{v}}(S)-\lim_{\varepsilon\to 0}\mathcal{E}_{\hat{v}}(\varepsilon)=\frac{K^2}{2^{ap^*+1}},& \quad A=0\, (\text{i.e. } \alpha=2),
	\end{align*}
contradiction. Therefore, for $n\geq 2,\ \alpha\geq 2$ and $\beta=\frac{n+1}{n-1}\alpha$, there is no positive solution in $ C^2[0,\frac12)\cap C^0[0,\frac12]$ to equation \eqref{eq:uni-main-eq}. The nonexistence result can also be obtained by \cite[Theorem 1.5]{Mancini}. This means that equation \eqref{genequ-1} has no positive weak solution for $n\geq 2,\ \alpha\geq 2$ and $\beta=\frac{n+1}{n-1}\alpha$.
\end{remark}

We also need the following two lemmas proved in \cite{Mancini}.
\begin{lemma}\label{v'}(\cite[Lemma 3.4 and Lemma 3.5]{Mancini})
	Assume $n\geq 1, \ p>2,\ \lambda\leq \frac{n^2}{4}$, and $v$ is a positive solution to \eqref{mancini-equ}, satisying \eqref{bound}. Then
	\begin{align}\label{v'-1}
	\lim_{s\to \infty}\frac{v'(s)}{v(s)}=-\frac{n-\sqrt{n^2-4\lambda}}{2}.
	\end{align}	
\end{lemma}

\begin{lemma}\label{mancini}(\cite[Lemma 4.1]{Mancini})
	Assume $\lambda\leq \frac{n^2}{4}$, $2<p\leq \frac{2(n+1)}{n-1}$ if  $n\geq 2$, and  $p>2$ if $n=1$. Let $v$ and $\bar{v}$ be two distinct positive solutions to \eqref{mancini-equ}. Then for any given $R,\ M>0$, there is $\delta=\delta(v,R)$, such that if $v(0),\ \bar{v}(0)\leq M$, then
	\begin{align*}
	v(s_i)=\bar{v}(s_i),\ i=1,2,\quad 0<s_1<s_2\leq R \quad \Rightarrow\quad s_2-s_1\geq \delta.
	\end{align*}
\end{lemma}

\noindent\textbf{Proof of  Proposition  \ref{thm:uniqueness-main}.}
We prove it by contradiction. Assume that $\psi_1,\ \psi_2$ are two positive solution to \eqref{eq:uni-main-eq}  satisfying $$\psi_1(\frac12)=K_1>\psi_2(\frac12)=K_2>0.$$
	
\noindent{\bf Claim 1. }
$\psi_1$ and $\psi_2$ must intersect with each other.

Assume on the contrary, $\psi_1(r)>\psi_2(r)$ on $[0,\frac12]$. Write $W=\psi_1\psi_2'-\psi_2\psi_1'$. From \eqref{eq:uni-main-eq}, we have
\begin{equation}\label{W}
[r^n(\frac14-r^2)^\alpha W]'=r^n(\frac14-r^2)^\beta\psi_1\psi_2(\psi_1^{p^*-2}-\psi_2^{p^*-2}),\ 0<r<\frac 12.
\end{equation}
Since $W(0)=0$, $\lim_{r\rightarrow(\frac12)^-}(\frac14-r^2)^\alpha W(r)=0$ and $p^*>2$,  we have
$$0=\int_0^\frac12 r^n(\frac14-r^2)^\beta\psi_1\psi_2(\psi_1^{p^*-2}-\psi_2^{p^*-2})dr>0,$$
contradiction.

\medskip

Define
\begin{align*}
\theta(r)=\frac{\psi_1(r)}{\psi_2(r)},\ \ \text{for} \ \ r\in[0,\frac{1}{2}].
\end{align*}
Claim 1 implies: there is $r_0\in(0,\frac12)$, such that $\theta(r_0)=1$ and $\theta(r)>1$ in $(r_0,\frac12]$. We also have for $r\in[r_0,\frac12)$,
$$-r^n(\frac14-r^2)^\alpha W(r)=\int_r^\frac12 \tau^n(\frac14-\tau^2)^\beta\psi_1\psi_2(\psi_1^{p^*-2}-\psi_2^{p^*-2})d\tau>0,$$
which implies $W(r)<0,\ \forall r \in[r_0,\frac12)$. Furthermore,
\begin{align}\label{der-r_0}
\theta'(r)=-\frac{W(r)}{\psi_2^2(r)}>0 ,\quad \forall r\in[r_0,\frac12).
\end{align}
Set
$$\tilde{r}_0=\inf\{r>0: \theta'(\tau)>0 ,\forall \tau\in(r,\frac12)\}.$$
Obviously, $0\leq\tilde{r}_0<r_0$.

\medskip

\noindent{\bf Claim 2.} $\tilde{r_0}=0$.

Suppose on the contrary that $\tilde{r}_0>0$, then we have $\theta'(\tilde{r}_0)=0$ and $\theta'(r)>0$ in $(\tilde{r}_0,\frac12)$. This implies $W(\tilde{r}_0)=0$ and $\theta(\tilde{r}_0)<\theta(r_0)=1.$  Note that $v_i$ satisfies \eqref{v-psi} for $i=1,2.$ Set 
$$\gamma(s)=\frac{v_1(s)}{v_2(s)}=\frac{\psi_1(r)}{\psi_2(r)}=\theta(r),\ \ r=\frac{\tanh\frac{s}{2}}{2},\ s\geq 0,$$
$$ s_0=\ln\frac{1+2r_0}{1-2r_0},\,\tilde{s}_0=\ln\frac{1+2\tilde{r}_0}{1-2\tilde{r}_0}.$$
It is easy to see
\begin{align}\label{gamma-pro}
\gamma(\tilde{s}_0)<1,\ \ \gamma'(\tilde{s}_0)=0,\ \ \gamma'(s)>0,\ \ \forall s\in(\tilde{s}_0, \infty).
\end{align}
It  follows from Remark \ref{no-solu} that we do not need to consider the case of $n\geq 2,\ p^*=2^*,\ $ and $A\leq 0 $. We prove Claim 2 in the following three cases.

\textbf{Case $1$.}  $n=1$ and $A\leq 0$, that is,  $\alpha,\ \beta$ satisfying \eqref{constraint}.  By (A1), it holds $G'(s)<0,\ s>0$.  Choose $s,\ \sigma,\ S$ satisfying $s=\sigma=\tilde{s}_0<S$ in \eqref{E-difference}.  From $\gamma'(s)>0,\ s>\tilde{s}_0$,  one has
\begin{align}\label{E-limit-1}
&\big(\mathcal{E}_{\hat{v}_1}(S)-\gamma^2(\tilde{s}_0)\mathcal{E}_{\hat{v}_2}(S)\big)-\big(\mathcal{E}_{\hat{v}_1}(\tilde{s}_0)-\gamma^2(\tilde{s}_0)\mathcal{E}_{\hat{v}_2}(\tilde{s}_0)\big)\nonumber\\
=&\frac12\int^S_{\tilde{s}_0}G'(\tau)[\gamma^2(\tau)-\gamma^2(\tilde{s}_0)]\hat{v}^2_2(\tau)d\tau<0.
\end{align}
Since  $\gamma'(\tilde{s}_0)=0$ and $\gamma(\tilde{s}_0)<1$, we have  $\frac{v_1'(\tilde{s}_0)}{v_1(\tilde{s}_0)}-\frac{v_2'(\tilde{s}_0)}{v_2(\tilde{s}_0)}=0$, $v_{1}(\tilde{s}_0)<v_2(\tilde{s}_0)$. It follows
\begin{equation}\label{s_0}
\begin{split}
\mathcal{E}_{\hat{v}_1}(\tilde{s}_0)-\gamma^2(\tilde{s}_0)\mathcal{E}_{\hat{v}_2}(\tilde{s}_0)
=\frac{\sinh^{ap^*}\tilde{s}_0}{2}v_1(\tilde{s}_0)^2\big[
0+\frac{2(v_1^{p^*-2}(\tilde{s}_0)-v_2^{p^*-2}(\tilde{s}_0))}{p^*}\big]<0.
\end{split}
\end{equation}
Noting that $0<\gamma(\tilde{s}_0)<\gamma(S)$, together with Lemma \ref{E-var-S},  we have
\begin{align*}
\big(\mathcal{E}_{\hat{v}_1}(S)-\gamma^2(\tilde{s}_0)\mathcal{E}_{\hat{v}_2}(S)\big)
-\big(\mathcal{E}_{\hat{v}_1}(\tilde{s}_0)-\gamma^2(\tilde{s}_0)\mathcal{E}_{\hat{v}_2}(\tilde{s}_0)\big)>0
\end{align*}for $S$ large enough. This contradicts \eqref{E-limit-1}.


\textbf{Case $2$.} $n\geq 2$ and $A\geq 0$. According to (A2) and (A3), we know  $G'(s)\geq 0,\ s>0.$ If $\tilde{r}_0>0$, we first show that there is a second intersection point for $\psi_1$ and $\psi_2$. If not, it holds $\psi_1(r)<\psi_2(r)$ in $(0,\tilde{r}_0)$.  Then by \eqref{W} and $W(\tilde{r}_0)=0$, we have
$$0=\int_0^{\tilde{r}_0}r^n(\frac14-r^2)^\beta\psi_1\psi_2(\psi_1^{p^*-2}-\psi_2^{p^*-2})dr<0,$$
contradiction. Hence, there is $r_1\in (0,\tilde{r}_0)$, such that $\theta(r_1)=1$ and $\theta(r)<1$ in $(r_1, r_0)$. Similarly, we have $\theta'(r)<0$ in $[r_1, \tilde{r}_0)$.

Set
$$\tilde{r}_1=\inf\{r>0: \theta'(\tau)<0 , \forall \tau\in(r,\tilde{r}_0)\},$$
$$s_1=\ln\frac{1+2r_1}{1-2r_1},\,\tilde{s}_1=\ln\frac{1+2\tilde{r}_1}{1-2\tilde{r}_1}.$$
Obviously, $0\leq \tilde{s}_1<s_1$. Next, we claim that $\tilde{s}_1>0.$

If $\tilde{s}_1=0$, then $\gamma'(s)<0,\  s\in(0,\tilde{s}_0)$. Choose $s,\ \sigma,\ S$ satisfying $s=\varepsilon<\sigma=S=\tilde{s}_0$ in \eqref{E-difference}. It yields
\begin{equation*}
\begin{split}
&(\mathcal{E}_{\hat{v}_1}(\tilde{s}_0)-\gamma^2(\tilde{s}_0)\mathcal{E}_{\hat{v}_2}(\tilde{s}_0))-(\mathcal{E}_{\hat{v}_1}(\varepsilon)-\gamma^2(\tilde{s}_0)\mathcal{E}_{\hat{v}_2}(\varepsilon))\\
=&\frac12\int_\varepsilon^{\tilde{s}_0}G'(\tau)[\gamma^2(\tau)-\gamma^2(\tilde{s}_0)]\hat{v}^2_2(\tau)d\tau\geq 0.
\end{split}
\end{equation*}
For $\varepsilon>0$ small enough, by \eqref{s_0} and Lemma \ref{E-var-S}, we have $$(\mathcal{E}_{\hat{v}_1}(\tilde{s}_0)-\gamma^2(\tilde{s}_0)\mathcal{E}_{\hat{v}_2}(\tilde{s}_0))-(\mathcal{E}_{\hat{v}_1}(\varepsilon)-\gamma^2(\tilde{s}_0)\mathcal{E}_{\hat{v}_2}(\varepsilon))<0 ,$$ contradiction. We obtain $\tilde{s}_1>0$, thus $\tilde{r}_1>0$.

Arguing as before, there exist two sequences of numbers $\{s_j\}_{j=1}^\infty$ and $ \{\tilde{s}_j\}_{j=1}^\infty $ with
$$0<\cdots<\tilde{s}_j<s_j<\tilde{s}_{j-1}<s_{j-1}<\cdots<\tilde{s}_0<s_0<+\infty,$$
 and $$\gamma(s_j)=1,\,\gamma'(\tilde{s}_j)=0.$$
It is easy to see  $s_j-s_{j+1}\rightarrow0$ as $j\rightarrow\infty$.
However, Lemma \ref{mancini} implies that there is $\delta>0$, depending on $v_1(0), v_2(0)$ and  $s_0$, such that $s_j-s_{j+1}>\delta$,  which yields a contradiction. Thus, we  obtain Claim 2 for the case  $n\geq 2$ and $A\geq 0$.

\textbf{Case $3$.}  $n\geq 2,\ A<0$ and $ p^*<2^*$. By (A2) and (A3), there is $c>0$ such that $G'(s)(s-c)<0,\ s>0,\ s\neq c$.

If $c\geq \tilde{s_0}$,  we take $s,\ \sigma,\ S$ satisfying $s=\tilde{s}_0,\ \sigma=c<S $ in \eqref{E-difference}. Since $\gamma'(s)>0,\ s>\tilde{s}_0$,  one has
\begin{eqnarray}\label{s-3-1}
& &\big(\mathcal{E}_{\hat{v}_1}(S)-\gamma^2(c)\mathcal{E}_{\hat{v}_2}(S)\big)
-\big(\mathcal{E}_{\hat{v}_1}(\tilde{s}_0)-\gamma^2(c)\mathcal{E}_{\hat{v}_2}(\tilde{s}_0)\big)\nonumber\\
&=&\frac12\int_{\tilde{s}_0}^S G'(\tau)[\gamma^2(\tau)-\gamma^2(c)]\hat{v}^2_2(\tau)d\tau< 0.
\end{eqnarray}
Using \eqref{s_0}, Lemma \ref{E-var-S} and $0<\gamma(\tilde{s_0})<\gamma(c)<\gamma(S)$, we have
\begin{align*}
\big(\mathcal{E}_{\hat{v}_1}(S)-\gamma^2(c)\mathcal{E}_{\hat{v}_2}(S)\big)
-\big(\mathcal{E}_{\hat{v}_1}(\tilde{s}_0)-\gamma^2(c)\mathcal{E}_{\hat{v}_2}(\tilde{s}_0)\big)>0
\end{align*}for $S$ large enough.
 It contradicts \eqref{s-3-1}.

If $c<\tilde{s_0}$, similar to the proof of Case  $2$, we define the corresponding $s_1$ and $\tilde{s}_1$. Claim $\tilde{s}_1=0$. If $\tilde{s}_1>0$, then $\gamma'(s)<0$ in $ (0,\tilde{s}_0)$. Choosing $s=\varepsilon<\sigma=c,\  S=\tilde{s}_0$ in \eqref{E-difference}, we have 
\begin{equation*}
\begin{split}
&\big(\mathcal{E}_{\hat{v}_1}(\tilde{s}_0)-\gamma^2(c)\mathcal{E}_{\hat{v}_2}(\tilde{s}_0)\big)-\big(\mathcal{E}_{\hat{v}_1}(\varepsilon)-\gamma^2(c)\mathcal{E}_{\hat{v}_2}(\varepsilon)\big)\\
=&\frac12\int_{\varepsilon}^{\tilde{s}_0} G'(\tau)[\gamma^2(\tau)-\gamma^2(c)]\hat{v}^2_2(\tau)d\tau>0.
\end{split}
\end{equation*}
By \eqref{s_0}, Lemma \ref{E-var-S} and $\gamma(\varepsilon)>\gamma(c)>\gamma(\tilde{s_0})>0$, we have
\begin{align*}
\big(\mathcal{E}_{\hat{v}_1}(\tilde{s}_0)-\gamma^2(c)\mathcal{E}_{\hat{v}_2}(\tilde{s}_0)\big)
-\big(\mathcal{E}_{\hat{v}_1}(\varepsilon)-\gamma^2(c)\mathcal{E}_{\hat{v}_2}(\varepsilon)\big)<0
\end{align*} for $\varepsilon>0$ small enough,
contradiction. Hence, $\tilde{s}_1>0$. Similar to  case $2$, we can get two sequences $\{s_j\}$ and $\{\tilde{s}_j\}$, and get the contradiction by Lemma \ref{mancini}. We hereby  complete the proof of Claim 2.

\medskip

By Claim 2,  we have  $\gamma'(s)>0,\ s>0$. Taking $s=\varepsilon\ll 1,\ S\gg 1$ and
\begin{equation}\label{sigma}
\sigma=
\begin{cases}
\varepsilon,\ & \ n=1,\ A\leq 0,\\
S,\ & \ n\geq2: \  p^*<2^*,\ A\geq 0\text{ or } p^*=2^*,\ A>0,\\
c,\ & \ n\geq 2,\ p^*<2^*,\ A<0
\end{cases}
\end{equation}	 in \eqref{E-difference},
Using (A1)-(A3),  one has
\begin{eqnarray}\label{M-0}
& &\big(\mathcal{E}_{\hat{v}_1}(S)-\gamma^2(\sigma)\mathcal{E}_{\hat{v}_2}(S)\big)
-\big(\mathcal{E}_{\hat{v}_1}(\varepsilon)-\gamma^2(\sigma)\mathcal{E}_{\hat{v}_2}(\varepsilon)\big)
\nonumber\\
&=&\frac12\int^S_{\varepsilon}G'(\tau)[\gamma^2(\tau)-\gamma^2(\sigma)]\hat{v}^2_2(\tau)d\tau<0.
\end{eqnarray}
For $\sigma$ satisfying \eqref{sigma},  using Lemma \ref{E-var-S} and $0<\gamma(\varepsilon) \leq\gamma(\sigma)\leq\gamma(S)
$, we have
\begin{equation*}
\liminf_{\varepsilon\rightarrow0^+, S\rightarrow\infty}[\big(\mathcal{E}_{\hat{v}_1}(S)-\gamma^2(\sigma)\mathcal{E}_{\hat{v}_2}(S)\big)-\big(\mathcal{E}_{\hat{v}_1}(\varepsilon)-\gamma^2(\sigma)\mathcal{E}_{\hat{v}_2}(\varepsilon)\big)]\geq 0,
\end{equation*}
 which contradicts  \eqref{M-0}. Therefore, there cannot be two distinct positive solutions $\psi_1$ and $\psi_2$. We thus complete the proof of Proposition \ref{thm:uniqueness-main}.
\qed

\medskip

 Finally, we present the proof for Lemma \ref{E-var-S}.

\noindent\textbf{Proof of Lemma \ref{E-var-S}.}
	(I) By $\psi_i(0)>0,\ \psi'_i(0)=0$ and \eqref{v-der}, we have $v_i(0)>0,\ v'_i(0)=0$. For $0<\varepsilon\ll 1$, one has
\begin{align*}
\mathcal{E}_{\hat{v}_i}(\varepsilon)
=\frac{(a^2+B)v_i^2(\varepsilon)}{2}\sinh^{ap^*-2}\varepsilon+O(\sinh^{ap^*-1}\varepsilon).
\end{align*}
Since $\hat{v}_1(\varepsilon)=\gamma(\varepsilon)\hat{v}_2(\varepsilon)$, it yields
	\begin{align*}
	\mathcal{E}_{\hat{v}_1}(\varepsilon)-\gamma^2(\varepsilon)	\mathcal{E}_{\hat{v}_2}(\varepsilon)=O(\sinh^{ap^*-1}\varepsilon).
	\end{align*}
	According to (A1)-(A3), we get result in part (1).
	
	(II) For $\alpha>1$, by  Lemma \ref{v'} and Lemma \ref{alpha>1}, it holds
	\begin{align}\label{v'/v}
	\frac{v_i'(S)}{v_i(S)}=-\frac{n+\alpha-1}{2}+o(1), \ S\gg 1.
	\end{align}
Invoking \eqref{hy-3} and $S\gg1$, we have
	\begin{align*}
	\mathcal{E}_{\hat{v}_i}(S)=&\frac{K_i^2}{2^{ap^*+1}}e^{(ap^*-(n+\alpha-1))S}[(a-\frac{n+\alpha-1}{2})^2+A+O(1)]\\
	=&\frac{K_i^2}{2^{ap^*+1}}e^{(ap^*-(n+\alpha-1))S}[\frac{n(\alpha+\beta)(n+\alpha-1)(\beta+2-\alpha)}{2(2n+\alpha+\beta)^2}+o(1)].
	\end{align*}
Noting that $2a=2n-ap^*$,  we have
	\begin{align}\label{A-value}
	A=\frac{n^2-(1-\alpha)^2}{4}-\frac{a^2p^*}{2}=\frac{n^2-(\alpha-1)^2-ap^*(2n-ap^*)}{4}=\frac{(ap^*-n)^2-(\alpha-1)^2}{4}.
	\end{align}
For $ap^*>n$ and $\alpha>1$, it holds that $A<0$ if and only if $n<ap^*<n+\alpha-1$, and $A>0$ if and only if $ap^*>n+\alpha-1$. Therefore, $\lim_{S\to\infty}\mathcal{E}_{\hat{v}_i}(S)=0,\ \ i=1,2$ for $A<0$. Noting that $\alpha>1$ and $\beta>-1$, we have $\alpha+\beta>0$, then $\lim_{S\to\infty}\mathcal{E}_{\hat{v}_i}(S)=+\infty,\ \ i=1,2$ for $A>0$. For $A=0$, it is easy to check that $\lim_{S\to\infty}\mathcal{E}_{\hat{v}_i}(S)=\frac{K_i^2(\alpha-1)^2}{2^{ap^*+1}}>0,\ \ i=1,2.$ We get \eqref{E-S}.
	
	\medskip
	
	Divide $\mathcal{E}_{\hat{v}_1}(S)-\gamma^2(S)\mathcal{E}_{\hat{v}_2}(S)$ into two parts:
	\begin{eqnarray*}
		\uppercase\expandafter{\romannumeral3}&=&\frac{\sinh^{ap^*}S}{2}v_1^2\big[(\frac{a}{\tanh S}+\frac{v_1'}{v_1})^2-(\frac{a}{\tanh S}+\frac{v_2'}{v_2})^2\big]\\
		&=&\frac{\sinh^{ap^*}S}{2}v_1^2(S)\big(\frac{v_1'(S)}{v_1(S)}-\frac{v_2'(S)}{v_2(S)}\big)\big(\frac{2a}{\tanh S}+\frac{v_1'(S)}{v_1(S)}+\frac{v_2'(S)}{v_2(S)}\big),
	\end{eqnarray*}
and
	\begin{eqnarray*}
		\uppercase\expandafter{\romannumeral4}=\frac{\sinh^{ap^*}S}{p^*}v_1^2(S)(v_1^{p^*-2}(S)-v_2^{p^*-2}(S)).
	\end{eqnarray*}
	By \eqref{hy-3}, it holds
	\begin{align}\label{roman-4}
	\uppercase\expandafter{\romannumeral4}=
	\big(\frac{K_1^2(K_1^{p^*-2}-K_2^{p^*-2})}{p^* 2^{ap^*}}+o(1)\big)e^{-\frac{(\alpha+\beta)(n+\beta+1)}{2n+\alpha+\beta}S},\quad S\gg 1.
	\end{align}
We consider the cases $0<\alpha< 1$ and $\alpha\geq 1$, separately.
	
\textbf{Case 1.} $0<\alpha<1$. From \eqref{boundary-alpha}, \eqref{hy-3} and \eqref{v-der}, we have
$$	\frac{v_i'(S)}{v_i(S)}=-\frac{n+\alpha-1}{2}+o(e^{(\alpha-1)S}), \ S\gg 1,$$
and then
\begin{align}\label{es-1}
\frac{v_1'(S)}{v_1(S)}-\frac{v_2'(S)}{v_2(S)}=o(e^{(\alpha-1)S}), \ S\gg 1,\end{align}
and
\begin{align}
\label{es-2}\frac{2a}{\tanh S}+\frac{v_1'(S)}{v_1(S)}+\frac{v_2'(S)}{v_2(S)}=-\frac{(\alpha+\beta)(n+\alpha-1)}{2n+\alpha+\beta}+o(e^{(\alpha-1)S}), \ S\gg 1.
\end{align}
Therefore, for $\alpha+\beta=0$, it holds
\begin{equation*}
	\lim_{S\rightarrow\infty}\uppercase\expandafter{\romannumeral3}=\lim_{S\rightarrow\infty}o(e^{(ap^*-n-(1-\alpha))S})=\lim_{S\rightarrow\infty}o(e^{\frac{(\alpha+\beta)(n+\alpha-1)}{2n+\alpha+\beta}S})=0.
	\end{equation*}
Combining the above with \eqref{roman-4}, we arrive at
	\begin{equation}\label{S-infty}
	\lim_{S\rightarrow\infty}\big(\mathcal{E}_{\hat{v}_1}(S)-\gamma^2(S)\mathcal{E}_{\hat{v}_2}(S)\big)=\frac{K_1^2(K_1^{p^*-2}-K_2^{p^*-2})}{p^* 2^{ap^*}}>0.
	\end{equation}
	
	For $\alpha+\beta<0$, it holds $\uppercase\expandafter{\romannumeral3}=o(e^{(ap^*-n)S}),\ S\gg1$.  By \eqref{gamma-pro}, we know that  $\gamma'(S)=\frac{v_1(S)}{v_2(S)}(\frac{v_1'(S)}{v_1(S)}-\frac{v_2'(S)}{v_2(S)})>0,\ S\gg 1$. Using $ap^*>n$ and \eqref{es-2}, it holds $$\lim_{S\rightarrow\infty}\uppercase\expandafter{\romannumeral3}\geq 0.$$ Together with \eqref{roman-4}, we have
	\begin{equation*}
	\lim_{S\rightarrow\infty}\big(\mathcal{E}_{\hat{v}_1}(S)-\gamma^2(S)\mathcal{E}_{\hat{v}_2}(S)\big)=+\infty.
	\end{equation*}
	
	For $\alpha+\beta>0$, we need more estimate of
	$\mathcal{E}_{\hat{v}_1}(S)-\gamma^2(S)\mathcal{E}_{\hat{v}_2}(S)$.
	We claim that
	\begin{equation}\label{3}
	\lim_{S\rightarrow\infty}[\sinh^{ap^*}S\big(v_1'(S)v_2(S)-v_1(S)v_2'(S)\big)]=0.
	\end{equation}
	In fact,for any  $0<\epsilon<\frac{(\alpha+\beta)(n+\beta+1)}{2n+\alpha+\beta}(<1+\beta)$, by \eqref{hy-3}, there exist $C_\epsilon>0$ and $S_0>0$, such that $$v_i(s)\leq C_\epsilon e^{(-\frac{n+\alpha-1}{2}+\frac{\epsilon}{p^*})s},\ \ \forall s\geq S_0,\ \  i=1,2.$$
It follows from \eqref{q} that
	\begin{equation}\label{difference}
	(qv_1')'v_2-(qv_2')'v_1+qv_1v_2(v_1^{p^*-2}-v_2^{p^*-2})=0.
	\end{equation}
	By \eqref{es-1},   we know $(qv_1'v_2-qv_1v_2')(S)=o(1),\ S\gg 1$. Integrating \eqref{difference}, for any $S\geq S_0$, we have
	\begin{eqnarray*}(qv_1'v_2-qv_1v_2')(S)&=&\int_{S}^{\infty}qv_1v_2(v_1^{p^*-2}-v_2^{p^*-2})ds\leq\int_{S}^{\infty}qv_1^{p^*-1}v_2ds\\
		&\leq&C(\epsilon)\int_S^\infty e^{(n-\frac{n+\alpha-1}{2}p^*+\epsilon)s}ds=O(e^{[\epsilon-(1+\beta)]S}).
	\end{eqnarray*}
Noting that by \eqref{gamma-pro}, we have $(qv_1'v_2-qv_1v_2')(S)>0,\ S\gg 1$, then
	\begin{equation*}
	\begin{split}
	&\lim_{S\rightarrow\infty}[\sinh^{ap^*}S\big(v_1'(S)v_2(S)-v_1(S)v_2'(S)\big)]\\
	=&\lim_{S\rightarrow\infty}O(e^{[ap^*-n+\epsilon-(1+\beta)]S})
	=\lim_{S\rightarrow\infty}O(e^{[\epsilon-\frac{(\alpha+\beta)(n+\beta+1)}{2n+\alpha+\beta}]S})=0.
	\end{split}
	\end{equation*}
We get the claim and arrive at
$$\lim_{S\rightarrow+\infty}\uppercase\expandafter{\romannumeral3}=0.$$
Combining  the above with \eqref{roman-4},  we have
	$$\lim_{S\rightarrow\infty}(\mathcal{E}_{\hat{v}_1}(S)-\gamma^2(S)\mathcal{E}_{\hat{v}_2}(S))=0.$$
	
\textbf{Case 2.} $\alpha\geq 1$. By  Lemma \ref{v'} and Lemma \ref{alpha>1}, \eqref{v'/v} holds, then we have
	$$\frac{2a}{\tanh S}+\frac{v_1'(S)}{v_1(S)}+\frac{v_2'(S)}{v_2(S)}=-\frac{(\alpha+\beta)(n+\alpha-1)}{2n+\alpha+\beta}+o(1), \ S\gg 1.$$
	Noting that $\alpha+\beta>0$ and using \eqref{3},
  we arrive at
  $$\lim_{S\rightarrow+\infty}\uppercase\expandafter{\romannumeral3}=0.$$
Combining this  with \eqref{roman-4}, we have
	$$\lim_{S\rightarrow\infty}(\mathcal{E}_{\hat{v}_1}(S)-\gamma^2(S)\mathcal{E}_{\hat{v}_2}(S))=0.$$
	We complete the proof of Lemma \ref{E-var-S}.
\hfill$\Box$

\bigskip

We now continue the proof of Theorem \ref{L-1}. First we observe:

$(i)$ If $\beta=\alpha-1$, \eqref{eq:uni-main-eq} has an obvious solution, $\psi=[\alpha(n+\alpha-1)]^{1/(p^*-2)}$. By Proposition \ref{uniq} and Proposition \ref{thm:uniqueness-main}, it is the unique solution provided $\alpha>0$ for $n\geq 2$ or $\alpha\in (0,\frac12]\cup[\frac14(1+\sqrt{17}),\infty)$ for $n=1$. Since $u$ is related to $\psi$ by \eqref{transform-2}, $u$ takes the form of \eqref{sol-1-0}. Combining Theorem \ref{sharp-C-2} with uniqueness result, we know that \eqref{sol-1-0} are the extremal functions of the optimal weighted Sobolev inequality \eqref{GGN-2}. A direct calculation yields the best constant
$$S_{n+1,\alpha,\alpha-1}=\alpha(n+\alpha-1)
\big[\pi^\frac n2\frac{\Gamma(\alpha)\Gamma({\frac{n}2+\alpha})}{\Gamma(n+2\alpha)}\big]^{\frac{1}{n+\alpha}}.$$

$(ii)$ If $\beta=\alpha$, it is easy to verify that \eqref{eq:uni-main-eq} has a solution $\psi(r)=C_{n,\alpha}(r^2+\frac14)^{-\frac{n+\alpha-1}{2}}$ for some suitable $C_{n,\alpha}$. By Proposition \ref{uniq} and Proposition \ref{thm:uniqueness-main}, it is the unique solution provided $\alpha>0$ for $n\geq 2$ or $\alpha\geq \sqrt{2}$ for $n=1$.  Using \eqref{transform-2}, we know that $u$ takes the form of \eqref{sol-2-0}. Combining Theorem \ref{sharp-C-2} with the uniqueness result, we know that \eqref{sol-2-0} are the extremal functions of the optimal weighted Sobolev inequality \eqref{GGN-2}. A direct calculation  yields the best constant
$$S_{n+1,\alpha,\alpha}=(n+\alpha-1)(n+\alpha+1)
\big[\frac{\pi^\frac n2}2\frac{\Gamma(\frac{\alpha+1}2)\Gamma(\frac{n+\alpha+1}2)}{\Gamma({n+\alpha+1})}
\big]^{\frac{2}{n+\alpha+1}}.$$
\hfill$\Box$

\section{Baouendi-Grushin operator and inequality}

As an application of the sharp Gargliardo-Nirenberg inequality, we shall derive the best constants for the sharp form of inequality \eqref{GGN-3}.   We first prove Proposition \ref{bec-1}.

\noindent\textbf{Proof of Propostion \ref{bec-1}.}
 Assume first $u\in  C_0^{\infty}(\mathbb{R}^{n+m})$ with  $u(x,z)=U(x,|z|),$ $ x\in\mathbb{R}^n,\ z\in\mathbb{R}^m$. 
 For $ \tau\ge0$, set
\begin{align}\label{change-1}
y=x, \ t=r^{\tau+1},\ r=|z|,\ \tilde{u}(y, t)= u(x, z).
\end{align} It is easy to verify
\begin{align*}dt=(\tau+1)r^{\tau}dr, \ dz=r^{m-1}drdS_{m-1},\  \partial_t\tilde{u}=\frac{\partial_r u}{(\tau+1)r^{\tau}},\ |\partial_r u|=|\nabla _z u|.\end{align*}
Then we have
\begin{align}\label{GB-3}
&\int_{\mathbb{R}_+^{n+1}} t^\alpha |\nabla \tilde{u}|^2 dy dt
=\int_{\mathbb{R}_+^{n+1}} t^\alpha \big[|\partial_t\tilde{u}|^2+|\nabla_y \tilde{u}|^2 \big]dy dt\nonumber \\
=&\int_{\mathbb{R}^{n}}\int_0^\infty r^{(\tau+1)\alpha} \big[ \frac1{(\tau+1)^2r^{2\tau}}|\partial_r u|^2+|\nabla_x u|^2\big](\tau+1)r^{\tau}dr dx \nonumber \\
=&\frac {1}{m\omega_m(\tau+1)}\int_{\mathbb{R}^{n+m}} |z|^{(\tau+1)\alpha -\tau-(m-1)}\big[|\nabla_z u|^2+(\tau+1)^2|z|^{2\tau}|\nabla_x u|^2 \big]dxdz,
\end{align}
 and
\begin{eqnarray}\label{GB-4}
\int_{\mathbb{R}_+^{n+1}} t^\beta | \tilde{u}|^{p^*} dy dt
&=&\int_{\mathbb{R}^{n}} \int_0^\infty r^{(\tau+1)\beta} | \tilde{u}|^{p^*}(\tau+1)r^{\tau}dr dy \nonumber \\
&=&\frac {\tau+1}{m\omega_m}\int_{\mathbb{R}^{n+m}}|z|^{(\tau+1)(\beta+1)-m}  | u|^{p^*}dxdz.
\end{eqnarray}
In \eqref{GB-3} and \eqref{GB-4}, let
$	(\tau+1)\alpha -\tau-(m-1)=0,\ \
(\tau+1)(\beta+1)-m=0,$
that is,
\begin{align}\label{abtau}
\alpha=\frac{m+\tau-1}{\tau+1},\quad \beta=\frac m{\tau+1}-1.
\end{align}
It is easy to check that  $p^*=\frac{2(n+\beta+1)}{n+ \alpha-1}=\frac{2(n(\tau+1)+m)}{n(\tau+1)+ m}=\frac{2Q}{Q-2},\ \alpha>0,\ \beta>-1$ and  $\frac {n-1}{n+1} \beta<\alpha<\beta+2.$ 
By Lemma \ref{density}, we have $\tilde{u}\in {\cal D}^{1,2}_{\alpha}(\mathbb{R}^{n+1}_+)$. Then by
weighted Sobolev inequality \eqref{GGN-2}, it holds
\begin{align}\label{GB-5}
( \int_{\mathbb{R}^{n+m} }  |u|^\frac{2Q}{Q-2}dx dz )^\frac{Q-2}{Q}\le C \int_{\mathbb{R}^{n+m} }(|\nabla_zu|^2+(\tau+1)^{2}|z|^{ 2\tau}|\nabla_x u|^2  )dx dz.
\end{align} 

For general $u\in{\cal D}^1_{\tau,z}(\mathbb{R}^{n+m})$, \eqref{GB-5} can be obtained by an approximation argument.
\hfill$\Box$

\medskip

In particular, for $m=1$ and $u\in C_0^\infty({\mathbb{R}^{n+1}})$, consider $u$ in $\mathbb{R}^{n+1}_+$ and $\mathbb{R}^{n+1}_-$ separately. Using the substitutions in \eqref{change-1} (with $-t$ instead of $t$ for $t<0$), similarly, we have
\begin{align*}
( \int_{\mathbb{R}^{n+1}_+ }  |u|^\frac{2Q}{Q-2}dx dz )^\frac{Q-2}{Q}\le C \int_{\mathbb{R}^{n+1}_+ }(|\nabla_zu|^2+(\tau+1)^{2}|z|^{ 2\tau}|\nabla_x u|^2  )dx dz,
\end{align*}
 and \begin{align*}
( \int_{\mathbb{R}^{n+1}_- }  |u|^\frac{2Q}{Q-2}dx dz )^\frac{Q-2}{Q}\le C \int_{\mathbb{R}^{n+1}_-}(|\nabla_zu|^2+(\tau+1)^{2}|z|^{ 2\tau}|\nabla_x u|^2  )dx dz.
\end{align*}
Combining the above two inequalities,  we obtain  the inequality in the whole space
\begin{align*}
( \int_{\mathbb{R}^{n+1} }  |u|^\frac{2Q}{Q-2}dx dz )^\frac{Q-2}{Q}\le C \int_{\mathbb{R}^{n+1}}(|\nabla_zu|^2+(\tau+1)^{2}|z|^{ 2\tau}|\nabla_x u|^2  )dx dz
\end{align*}
 for $u\in \cal{D}^1_\tau(\mathbb{R}^{n+1})$.

\medskip

Now we derive Theorem \ref{L-2} from Theorem \ref{L-1}.

\smallskip

\noindent\textbf{Proof of Theorem \ref{L-2}.}
Define \begin{align*}
F[v]=\frac{\int_{\mathbb{R}_+^{n+1}} t^\alpha |\nabla v|^2 dy dt}{\big(\int_{\mathbb{R}_+^{n+1}} t^\beta |v|^{p^*} dy dt\big)^{\frac{2}{p^*}}},\ v\in {\cal D}^{1,2}_{\alpha}(\mathbb{R}^{n+1}_+),
\end{align*} and
\begin{align}\label{G}
G[v]=\frac{\int_{\mathbb{R}^{n+m} }(|\nabla_zv|^2+(\tau+1)^{2}|z|^{ 2\tau}|\nabla_x v|^2  )dx dz}{\big(\int_{\mathbb{R}^{n+m} }|v|^{\frac{2Q}{Q-2}} dx dz\big)^{\frac{Q-2}{Q}}},\ v\in{\cal D}^1_{\tau}(\mathbb{R}^{n+m}).
\end{align}
For $\tau>0$, let $\alpha, \ \beta$ satisfy \eqref{abtau}, $u$, $\tilde{u}$ satisfy \eqref{change-1}, and $u\in C_0^\infty(\mathbb{R}^{n+m})$ with $u(x,z)=U(x,|z|)$. From the proof of Proposition \ref{bec-1}, we know $\tilde{u}\in {\cal D}^{1,2}_{\alpha}(\mathbb{R}^{n+1}_+)$, and
\begin{align}\label{functional}
G[u]=(m\omega_m)^{\frac{2}{Q}}(\tau+1)^{\frac{2Q-2}{Q}}F[\tilde{u}].
\end{align}
Taking the infimum of  $G[u]$, we ge
\begin{align*}
S_{\tau,z}(n,m)
\geq &(m\omega_m)^{\frac{2}{Q}}(\tau+1)^{\frac{2Q-2}{Q}}\inf_{v\in {\cal D}^{1,2}_\alpha(\mathbb{R}^{n+1}_+)}F[v]\\
=&(m\omega_m)^{\frac{2}{Q}}(\tau+1)^{\frac{2Q-2}{Q}}S_{n+1,\frac{m+\tau-1}{\tau+1},\frac{m}{\tau+1}-1}.
\end{align*}
On the other hand, noting that $\tau>0$ and $\partial_r u=(\tau+1)t^{\frac{\tau}{\tau+1}}\partial_t\tilde{u}$,  for  $\tilde{u}\in C_0^1(\overline{\mathbb{R}^{n+1}_+})$,  we have $u\in C_0^1(\mathbb{R}^{n+m})$ and $u$ is radially symmetric with respect to $z$. We obtain the reverse inequality. Thus
\begin{align*}
S_{\tau,z}(n,m)=
(m\omega_m)^{\frac{2}{Q}}(\tau+1)^{\frac{2Q-2}{Q}}S_{n+1,\frac{m+\tau-1}{\tau+1},\frac{m}{\tau+1}-1}.
\end{align*}
Besides, $u$ is the extremal function of $G$ in ${\cal D}^1_{\tau,z}(\mathbb{R}^{n+m})$ if and only if $\tilde{u}$ is the extremal function of $F$ in ${\cal D}^{1,2}_{\alpha}(\mathbb{R}^{n+1}_+)$.

For $n=1$, $\alpha$ and $\beta$ satisfy \eqref{constraint}, or $n\geq 2$, that is,  for $n=1,\ m\neq 2$, or $n\geq 2$, it follows from Theorem \ref{L-1}, that the extremal functions of  \eqref{GGN-3} in $\cal{D}^1_{\tau,z}(\mathbb{R}^{n+m})$ are in the following form
\begin{align}\label{sol-gru-1-1}
u(x, z)=k(\frac{1}{|x-x^o|^2+(|z|^{\tau+1}+A)^2})^{\frac{Q-2}{2(\tau+1)}} \psi(|\frac{(x-x^o, |z|^{\tau+1}+A)}{|x-x^o|^2+(|z|^{\tau+1}+A)^2}-(0,\frac{1}{2 A})|),
\end{align} where $k>0,\ A>0, \ x^o\in\mathbb{R}^{n}$, and $\psi$ is the unique positive solution to equation \eqref{ode-0}.

In particular, for $\tau=1$, that is,  for $\alpha=\frac{m}{2}$, $\beta=\frac{m}{2}-1$($n=1,\ m\neq 2$, or $n\geq 2$), it follows from Theorem \ref{L-1}, that the extremal functions of  \eqref{GGN-3} in $\cal{D}^1_{1,z}(\mathbb{R}^{n+m})$ are given by
\begin{align}\label{solu-tau=1}
u(x,z)=k\Big(\frac{A}{(A+|z|^2)^2+|x-x^o|^2}\Big)^{\frac{2n+m-2}{4}}, \quad x\in \mathbb{R}^{n},\ z\in\mathbb{R}^m,
\end{align}
 where $k>0, \ A>0, \ x^o\in\mathbb{R}^{n}$.
 Via a direct calculation,  the best constant is
\begin{align*}
S_{1,z}(n,m)=m(2n+m-2) \big[\pi^\frac {n+m}2\frac{
	\Gamma({\frac{n+m}2})}{\Gamma(n+m)}\big]^{\frac{2}{2n+m}}.
\end{align*}
We have hereby  completed the proof of  part $1)$ and part $2)$  in Theorem \ref{L-2}.
Next, we prove the second part.

Assume that $u$ is a positive weak solution to equation  \eqref{gru-1} in $\cal{D}^1_{\tau.z}(\mathbb{R}^{n+m})$, let $\tilde{u}$ and $u$ satisfy the transformation in \eqref{change-1}. Similar to the proof of Proposition \ref{bec-1}, $(\tau+1)^{-\frac{Q-2}2}\tilde{u}$ is a positive weak solution to equation \eqref{genequ-1} for $\alpha=\frac{m+\tau-1}{\tau+1},\ \beta=\frac m{\tau+1}-1$.  By Theorem \ref{L-1}, we know that $\tilde{u}$ is of the form \eqref{type-0}, then
$u$ is of the form \eqref{sol-gru-1-1}. In particular, for $\tau=1$, that is, for  $\alpha=\frac{m}{2}$, $\beta=\frac{m}{2}-1$($n=1,\ m\neq 2$, or $n\geq 2$), $u$ is of the form \eqref{solu-tau=1}.
\hfill$\Box$

\bigskip

Finally, we prove Proposition \ref{symmetrization-z}. Apparently, this will follow directly from the following rearrangement result (Proposition \ref{rearrangement-lemma}).

For a function $u(x,z)\in C_0^\infty(\mathbb{R}^{n+m})$, where $x\in \mathbb{R}^{n}$ and $z\in\mathbb{R}^m$, let $u_x^*$  denote the symmetric decreasing rearrangement of $|u|(\cdot,z)$ with respect to variable $x$, and $u_{z}^{*}$ denote the  symmetric decreasing rearrangement of $|u|(x,\cdot)$ with respect to variable $z$.


Let
$\mathcal{F} f$ be the Fourier transform of $f$, 
and $\mathcal{F}^{-1}f$ be the inverse Fourier transform of $f$. For $f\in L^1(\mathbb{R}^n)$,
\[
\mathcal{F}f(\xi)=\int_{\mathbb{R}^{n}}e^{-2\pi i\xi\cdot x}f(x) dx,
\] and for $f\in L^p(\mathbb{R}^{n})$, the Fourier transform can be defined via  approximation in Schwartz space ${\cal S}(\mathbb{R}^{n})(1<p\leq 2)$ or as a distribution($p>1$)( see e.g., \cite{SS2003}). For the function $u(x, z)$ over $\mathbb{R}^{n+m}$, we denote $\mathcal{F}_x u$ to be the Fourier transform of $u$ with respect to variable $x\in\mathbb{R}^n$, and $\mathcal{F}^{-1}_\xi u$ to be the inverse Fourier transform of $u$ with respect to variable $\xi\in\mathbb{R}^n$.


We shall establish the following  decreasing rearrangement property involving Fourier transform.

\begin{proposition}\label{rearrangement-lemma} Assume $\tau>0$, $\frac{2Q}{Q-2}=s$ is an integer, $u\in C_0^\infty(\mathbb{R}^{n+m})$ is  a real-valued function, and $w=\mathcal{F}^{-1}_\xi[(\mathcal{F}_x(u^*_x))^*_z]$. Then $w$ is real-valued, $w\in \cal{D}^1_{\tau,z}(\mathbb{R}^{n+m})$ and it holds
	\begin{eqnarray} \label{rearrangement-1}\small
	& &\frac{\int_{\mathbb{R}^{n+m}}(|\nabla_z u|^2+(\tau+1)^{2}|z|^{2\tau }|\nabla_x u|^2 )dxdz}{\big(\int_{\mathbb{R}^{n+m}}  |u|^{\frac{2Q}{Q-2}}dxdz\big)^\frac{Q-2}{Q}}\nonumber\\
	&\ge&\frac{\int_{\mathbb{R}^{n+m}}(|\nabla_z w|^2+(\tau+1)^{2}|z|^{2\tau }(|\nabla_x w|^2 )dx dz}{\big(\int_{\mathbb{R}^{n+m}}  |w|^{\frac{2Q}{Q-2}}dx dz\big)^\frac{Q-2}{Q}}.
	\end{eqnarray}
\end{proposition}

Obviously, the function $w$ obtained in Proposition \ref{rearrangement-lemma} is radially symmetric with respect to $x$ and $z$ respectively, since the Fourier transform of a radial function is radial.

Similar rearrangement results are obtained by Beckner \cite{Bec2001}, and recently by E. Lenzmann and J. Sok \cite{LS2019}.

 In order to prove Proposition \ref{rearrangement-lemma}, we need some technical lemmas.

\begin{lemma}\label{*1}(\cite[(4.9)]{LS2019})
	Assume $\tau>0$. For any measurable function $f:\ \mathbb{R}^{m}\to [0,\infty)$ that vanishes at infinity, it holds
	\begin{align*}
	\int_{\mathbb{R}^m}	|z|^{2\tau}f^*_z(z)dz\leq  \int_{\mathbb{R}^m}	|z|^{2\tau}f(z)dz.
	\end{align*}
\end{lemma}

\begin{lemma}\label{bll-ineq}(Brascamp-Lieb-Luttinger inequality, \cite{BLL1974})
	Suppose $f_j,\ 1\leq j\leq J$ are nonnegative measurable functions on $\mathbb{R}^m$ that vanish at infinity, and $a_{jl},\ 1\le j\le J,\ 1\le l\le L$ are real numbers. Then
	\begin{align*}
	\int_{\mathbb{R}^{mL}} \prod_{j=1}^J f_j(\sum_{l=1}^L a_{jl}z_l) dz_1\cdots dz_L \leq 	\int_{\mathbb{R}^{mL}} \prod_{j=1}^J (f_j)^*_z(\sum_{l=1}^L a_{jl}z_l) dz_1\cdots dz_L .
	\end{align*}
\end{lemma}

In particular, in Lemma \ref{bll-ineq},  taking $L=1$ and $a_{j1}=1$, we have
\begin{align}\label{*-ineq}
		\int_{\mathbb{R}^{m}} \prod_{j=1}^J f_j(z)dz\leq \int_{\mathbb{R}^{m}} \prod_{j=1}^J (f_j)^*_z(z)dz.
\end{align}
Since $f^*_z=|f|^*_z$, the above inequality also holds for $f$ changing sign or complex-valued.

\bigskip

For $f$ and $g$ in Schwartz space $\mathcal{S}(\mathbb{R}^n)$, we have the following basic facts (see 
\cite{SS2003}): 
\begin{equation}\label{F-de}
\mathcal{F}(\nabla f)(\xi)=(2\pi i\xi)\mathcal{F}f(\xi),\ \xi \in \mathbb{R}^n,
\end{equation}
\begin{align}\label{F-inverse}
\mathcal{F}^{-1} \mathcal{F} f =f,
\end{align}
and the Plancherel formula
\begin{align}\label{parseval}
\int_{\mathbb{R}^n}(\mathcal{F}f)\overline{(\mathcal{F}g)}d\xi=\int_{\mathbb{R}^n}f\bar{g}dx.
\end{align}
Combining \eqref{F-de} and \eqref{parseval}, we have
\begin{align}\label{F-inte}
	\int_{\mathbb{R}^{n}} |\nabla f|^2(x) dx=\int_{\mathbb{R}^{n}} |\mathcal{F}(\nabla f)|^2(\xi) d\xi=4\pi^2\int_{\mathbb{R}^n} |\xi|^2|\mathcal{F} f|^2(\xi)d\xi.
\end{align}

By the approximation, it is easy to see that \eqref{F-inverse} and \eqref{parseval} hold for $f,g\in L^2(\mathbb{R}^{n})$, while \eqref{F-de} and \eqref{F-inte} hold for $f\in H^{1,2}(\mathbb{R}^{n})$.

 \begin{lemma}\label{hy}(Hausdorff-Young inequality)
 	Suppose $f\in L^p(\mathbb{R}^n)$, $1\leq p\leq 2$, $\frac{1}{p}+\frac{1}{p'}=1$. Then $\mathcal{F}f\in L^{p'}(\mathbb{R}^n)$ and $\|\mathcal{F}f\|_{L^{p'}(\mathbb{R}^n)}\leq \|f\|_{L^p(\mathbb{R}^n)}$.
 \end{lemma}

\begin{lemma}\label{F-co}(\cite[Lemma A.4,]{LS2019})
  Let $s$ be an integer with $s\geq 2$, and $g\in\mathcal{F}(L^{\frac{s}{s-1}}(\mathbb{R}^n))$ is real-valued, then
  \begin{align*}
  	\mathcal{F}(g^s)(\xi)=(\mathcal{F}g)\ast\cdots\ast(\mathcal{F}g)(\xi),\ \forall \xi\in\mathbb{R}^n,
  \end{align*}
  with $s-1$ convolutions on the right side.
\end{lemma}

{\bf Proof of Proposition \ref{rearrangement-lemma}}.
Recall the definition
\begin{align*}
G[u]=\frac{\int_{\mathbb{R}^{n+m}}(|\nabla_z u|^2+(\tau+1)^{2}|z|^{2\tau }|\nabla_x u|^2 )dxdz}{\big(\int_{\mathbb{R}^{n+m}}  |u|^{\frac{2Q}{Q-2}}dxdz\big)^\frac{Q-2}{Q}}.
\end{align*}\
for $u\in C_0^{\infty}(\mathbb{R}^{n+m})$.  The proof is divided into  the following three steps.

\textbf{Step $1$.} We show $G[u^*_x]\leq G[u]$.
 For any fixed $z\in\mathbb{R}^{m}$, it holds
\begin{align}\label{equi-p}
	\int_{\mathbb{R}^n} |u^*_x(x,z)|^pdx=\int_{\mathbb{R}^n} |u(x,z)|^pdx,\ \forall p\geq 1,
\end{align}
and the P\'{o}lya-Szeg\"{o} inequality (sometimes called Dirichlet Principle)
\begin{align}\label{v-de-x}
	\int_{\mathbb{R}^n} |\nabla_x u^*_x(x,z)|^2dx\leq \int_{\mathbb{R}^n} |\nabla_x u(x,z)|^2dx.
\end{align}
 Let $e_j$ be the unit vector in $\mathbb{R}^m$ whose $j$-th component is $1$, $j=1,\cdots,m$.
By \eqref{equi-p} and \eqref{*-ineq}, it holds for any fixed $z\in\mathbb{R}^{m}$,
\begin{align*}
\int_{\mathbb{R}^{n}}|u^*_x(x,z)|^2dx=	\int_{\mathbb{R}^{n}}| u(x,z)|^2dx ,
\end{align*} and
\begin{align*}
\int_{\mathbb{R}^{n}}u^*_x(x,z+he_j)u^*_x(x,z)dx
\geq& \int_{\mathbb{R}^{n}}u(x,z+he_j)u(x,z)dx.
\end{align*}
As for the weak derivative of $u^*_x$ with respect to $z$, we apply the method in \cite{FD2002}. Letting $h\neq0$ , we have
\begin{align}\label{R-1}
&\int_{\mathbb{R}^{n+m}}\frac{|u^*_x(x,z+he_j)-u^*_x(x,z)|^2}{h^2}dxdz\nonumber\\
=&\int_{\mathbb{R}^{n+m}}\frac{|u^*_x(x,z+he_j)|^2-2u^*_x(x,z+he_j)u^*_x(x,z)+|u^*_x(x,z)|^2}{h^2}dxdz\nonumber\\
\leq& \int_{\mathbb{R}^{n+m}}\frac{|u(x,z+he_j)-u(x,z)|^2}{h^2}dxdz\nonumber\\
\leq &\int_{\mathbb{R}^{n+m}}|\int_0^1 ( \partial_{z_j} u)(x,z+\theta he_j)d\theta|^2dxdz\nonumber\\
\leq& \int_{\mathbb{R}^{n+m}}|\partial_{z_j} u(x,z)|^2dxdz.
\end{align}
 Define
	\[f_h(x,z)=\frac{u_{x}^{*}(x,z+he_j)-u_{x}^{*}(x,z)}{h}.
	\]
	It follows from \eqref{R-1} that $f_h$ is bounded in $L^2(\mathbb{R}^{n+m})$.   Using similar argument as that in \cite{FD2002}, we know: for a sequence $h_l\to0$, $f_{h_l}$  weakly converges to a certain $f_0$ in $L^2(\mathbb{R}^{n+m})$ such that $ \|f_0\|_{L^2(\mathbb{R}^{n+m})}\le\|\partial_{z_j}u\|_{L^2(\mathbb{R}^{n+m})}$. Furthermore, $f_0$
	is the weak derivative of $u_{x }^{*}$ with respect to $z_j$ and 
we have
\begin{align}\label{v-de-z}
\int_{\mathbb{R}^{n+m}} |\nabla_z u^*_x|^2dxdz\leq \int_{\mathbb{R}^{n+m}} |\nabla_z u|^2dxdz.
\end{align}
 By \eqref{equi-p}($p=\frac{2Q}{Q-2}$), \eqref{v-de-x} and \eqref{v-de-z}, we arrive at $G[u^*_x]\leq G[u]$.

\textbf{Step $2$.} we show $G[w]\leq G[u^*_x]$. Since
 $(\mathcal{F}_x (u^*_x))^*_z$ is nonnegative, radially symmetric with respect to $\xi$ and $ z$, respectively,
 by Plancherel formula, we have $(\mathcal{F}_x (u^*_x))^*_z\in L^2(\mathbb{R}^{n+m})$ , i.e.
 \begin{align*}
 	\int_{\mathbb{R}^{n+m}}|(\mathcal{F}_x (u^*_x))^*_z|^2 d\xi dz=\int_{\mathbb{R}^{n+m}}|\mathcal{F}_x (u^*_x)|^2 d\xi dz=\int_{\mathbb{R}^{n+m}}| u^*_x|^2 dx dz=\int_{\mathbb{R}^{n+m}}|u|^2 dxdz.
 \end{align*}

Noting that $(\mathcal{F}_x (u^*_x))^*_z$ is real-valued and radially symmetric with respect to $\xi\in\mathbb{R}^n$, we have that $w=\mathcal{F}^{-1}_\xi[(\mathcal{F}_x (u^*_x))^*_z]$ is real valued.

 Similar to the proof of existence of weak derivative of $u^*_x$ with respect to $z$ in Step $1$, there exists the weak derivative of $w$ with respect to $z$  and
 \begin{align*}
 \int_{\mathbb{R}^{n+m}}|\nabla_z w|^2dxdz\leq  \int_{\mathbb{R}^{n+m}}|\nabla_z (u^*_x)|^2dxdz.
 \end{align*}
 For any fixed $\xi\in\mathbb{R}^n$, choosing $f(z)=|\mathcal{F}_x (u^*_x)|^2(\xi,z)$ in Lemma \ref{*1}, and using
 $$(|\mathcal{F}_x (u^*_x)|^2)^*_z=|(\mathcal{F}_x (u^*_x))^*_z|^2=|\mathcal{F}_x w|^2,$$
 we have
 \begin{align*}
 	\int_{\mathbb{R}^{m}}|z|^{2\tau}|\mathcal{F}_x w|^2 dz\leq 	\int_{\mathbb{R}^{m}}|z|^{2\tau}|\mathcal{F}_x (u^*_x)|^2  dz.
 \end{align*}
Integrating it with $|\xi|^2d\xi$, we have
 \begin{align*}
\int_{\mathbb{R}^{n+m}}|z|^{2\tau}|\xi|^2|\mathcal{F}_x w|^2 d\xi dz\leq 	\int_{\mathbb{R}^{n+m}}|z|^{2\tau}|\xi|^2|\mathcal{F}_x (u^*_x)|^2  d\xi dz.
\end{align*}
Furthermore, using \eqref{F-inte},  we know that there exists the weak derivative of $w$ with respcet to $x$ and
\begin{align*}
\int_{\mathbb{R}^{n+m}}|z|^{2\tau}|\nabla_x w|^2 dxdz\leq \int_{\mathbb{R}^{n+m}}|z|^{2\tau}|\nabla_x (u^*_x) |^2 dxdz,
\end{align*}where $$\partial_k w=2\pi i \mathcal{F}^{-1}_\xi(\xi_k\mathcal{F}_x w),\ k=1\cdots,n.$$

To show that $G(w]\leq G[u^*_x]$, we only need to prove that
\begin{align*}
	\int_{\mathbb{R}^{n+m}} |u^*_x|^s dxdz\leq 	\int_{\mathbb{R}^{n+m}} |w|^s dxdz,
\end{align*} where $s=\frac{2Q}{Q-2}\in\mathbb{N}$. We employ Lemma \ref{F-co} to show it.

First, we check that for a.e. fixed $z\in\mathbb{R}^m, $ $u^*_x(\cdot,z),\ w(\cdot,z) \in\mathcal{F}(L^{\frac{s}{s-1}}(\mathbb{R}^{n}))$. Set $s'=\frac{s}{s-1}=\frac{2Q}{Q+2}\in (1,2)$. Since $\mathcal{F}_x (u^*_x)$ and $(\mathcal{F}_x (u^*_x))^*_z$ are radially symmetric with respect to $x$, we have
 \begin{align}\label{f-inverse}
	u^*_x=\mathcal{F}^{-1}_\xi\mathcal{F}_x(u^*_x)=\mathcal{F}_x\mathcal{F}_x(u^*_x),\
	w=\mathcal{F}^{-1}_\xi((\mathcal{F}_x(u^*_x))^*_z)=\mathcal{F}_x((\mathcal{F}_x(u^*_x))^*_z).
\end{align} It suffices to prove that $\mathcal{F}_x (u^*_x)(\cdot,z)\in L^{s'}(\mathbb{R}^n)$ and $(\mathcal{F}_x (u^*_x))^*_z(\cdot,z)\in L^{s'}(\mathbb{R}^n)$.
By H\"{o}lder inequality, $Q>n$, \eqref{parseval} and \eqref{F-inte}, one has
\begin{align*}
\int_{\mathbb{R}^{n}} |\mathcal{F}_x(u^*_x)|^{s'}(\xi,z)d\xi \leq& C\Big( \int_{\mathbb{R}^{n}} (1+|\xi|^2)|\mathcal{F}_x(u^*_x)|^{2}(\xi,z)d\xi\Big)^{\frac{s'}{2}}\nonumber\\
=& C\Big(\int_{\mathbb{R}^{n}} (|\nabla_x( u^*_x)(x,z)|^{2}+|u^*_x(x,z)|^2)dx\Big)^{\frac{s'}{2}}\nonumber\\
\leq &C\Big(\int_{\mathbb{R}^{n}} (|\nabla_x u(x,z)|^{2}+|u(x,z)|^2)dx\Big)^{\frac{s'}{2}}\\
\leq& C(u)<\infty
\end{align*}for any fixed $z\in\mathbb{R}^m$.
We also have
\begin{align}\label{s'-bd}
\int_{\mathbb{R}^{n+m}}|(\mathcal{F}_x (u^*_x))^*_z|^{s'}d\xi dz=\int_{\mathbb{R}^{n+m}}|\mathcal{F}_x (u^*_x)|^{s'}d\xi dz<\infty.
\end{align} By Fubini Theorem, for a.e. $z\in\mathbb{R}^m$, it holds $(\mathcal{F}_x (u^*_x))^*_z(\cdot,z)\in L^{s'}(\mathbb{R}^n)$. Besides, using \eqref{f-inverse}, \eqref{s'-bd} and Hausdorff-Young inequality (Lemma \ref{hy}), we obtain
\begin{align*}
\|w\|_{L^s(\mathbb{R}^{n+m})}\leq \|(\mathcal{F}_x(u^*_x))^*_z\|_{L^{s'}(\mathbb{R}^{n+m})}=\|\mathcal{F}_x(u^*_x)\|_{L^{s'}(\mathbb{R}^{n+m})}<\infty.
\end{align*}

Note that for fixed $z$, $\mathcal{F}_x((u^*_x)^s) (\cdot,z)$ is continuous due to $(u^*_x)^s(\cdot,z)\in L^1(\mathbb{R}^n)$, so $\mathcal{F}_x((u^*_x)^s)(0,z)=\int_{\mathbb{R}^n}(u^*_x)^s(x,z)dx$ is well-defined.  Since $u^*_x$ is nonnegative, by Lemma \ref{F-co}, it holds
\begin{align}\label{co-s}
	\int_{\mathbb{R}^{n+m}}|u^*_x|^s(x,z) dxdz=&\int_{\mathbb{R}^m}\mathcal{F}_x((u^*_x)^s)(0,z)dz\nonumber\\
	=&\int_{\mathbb{R}^m}(\mathcal{F}_x(u^*_z))\ast\cdots\ast(\mathcal{F}_x(u^*_x))(0,z)dz,
\end{align}
with $s-1$ convolutions with respect to $x\in\mathbb{R}^n$. By \eqref{*-ineq} and $\mathcal{F}_x w=(\mathcal{F}_x(u^*_x))^*_z$, 
it follows
\begin{align*}
	\int_{\mathbb{R}^m}(\mathcal{F}_x(u^*_x))\ast\cdots\ast(\mathcal{F}_x(u^*_x))(0,z)dz\leq& \int_{\mathbb{R}^m}(\mathcal{F}_x(u^*_x))^*_z\ast\cdots\ast(\mathcal{F}_x(u^*_x))^*_z(0,z)dz\\=&\int_{\mathbb{R}^m}(\mathcal{F}_x w)\ast\cdots\ast(\mathcal{F}_x w)(0,z)dz.
\end{align*}
Back to \eqref{co-s}, since $w$ is real-valued, invoking the definition of Fourier transformation and Lemma \ref{F-co} again,  we have
\begin{align*}
	\int_{\mathbb{R}^{n+m}}|u^*_x(x,z) |^sdxdz\leq \int_{\mathbb{R}^{n+m}}w^s(x,z)dxdz\leq \int_{\mathbb{R}^{n+m}}|w(x,z)|^sdxdz.
\end{align*}
Hence, we obtain $G[w]\leq G[u^*_x]\leq G[u]$.

\textbf{ Step $3$.} We prove that $w\in \cal{D}^1_{\tau,z}(\mathbb{R}^{n+m})$. By now, we know $w\in W^{1,1}_{loc}(\mathbb{R}^{n+m})$ is radially symmetric with respect to $x$ and $z$, respectively,  with
\begin{align*}
		\int_{\mathbb{R}^{n+m}}(|\nabla_z w|^2+|z|^{2\tau}|\nabla_x w|^2)dxdz<\infty,
\end{align*}
and
 \begin{align}\label{s-bd}
\int_{\mathbb{R}^{n+m}}|w|^{\frac{2Q}{Q-2}}dxdz<\infty. 
\end{align}
Set
\begin{align}\label{sob-gru-tr}
	\tilde{w}(y,t)=w(x,z),\ y=x\in\mathbb{R}^n,\ t=|z|^{\tau+1}\geq 0, \ z\in\mathbb{R}^m
\end{align}
and
\begin{align*}
\alpha=\frac{m+\tau-1}{\tau+1},\ \beta=\frac{m}{\tau+1}-1.
\end{align*}
Similar to the proof of Proposition\ref{bec-1}, we have $\tilde{w}\in \cal{D}^{1,2}_{\alpha}(\mathbb{R}^{n+1}_+)$. Suppose ${\tilde{w}_j}\subset
C_0^\infty (\mathbb{R}^{n+1}_+)$ is the approximation of $\tilde{w}$, 
 then for $\tau>0$, $w_j\in C_0^1(\mathbb{R}^{n+m})$ is the
approximation of $w$, 
where $\tilde{w}_j$ and $w_j$ satisfy the relation in \eqref{sob-gru-tr}. It implies  $w\in \cal{D}^{1}_{\tau,z}(\mathbb{R}^{n+m})$.

\hfill$\Box$

\section{Appendix}We provide proofs for some technical lemmas in this appendix.
Define
\begin{align}\label{K_ab}
{\cal K}_{\alpha,\beta}=\{u\in W^{1,1}_{loc}(\mathbb{R}^{n+1}_+)\big| \nabla u\in L^{2}_{\alpha}(\mathbb{R}^{n+1}_+),\  u\in L^{p^*}_{\beta}(\mathbb{R}^{n+1}_+)\}.
\end{align}
Then we have
\begin{lemma}\label{density} If $\alpha,\ \beta$ satisfy \eqref{beta-1}, then
	$	{\cal D}^{1,2}_{\alpha}(\mathbb{R}^{n+1}_+)={\cal K}_{\alpha,\beta}.$
\end{lemma}

\begin{proof}
$(1)$ First, we prove that $	{\cal D}^{1,2}_{\alpha}(\mathbb{R}^{n+1}_+)\subset {\cal K}_{\alpha,\beta}.$
Assume that $\{u_j\}\subset C^{\infty}_{0}(\overline{\mathbb{R}^{n+1}_+})$ satisfies
	$$\int_{\mathbb{R}^{n+1}_+} t^\alpha |\nabla u_i-\nabla u_j|^2dydt\to 0,\text{ as }i,j\to\infty.$$
	Then by \eqref{GGN-2}, as $i,j\to\infty$,
		$$\big(\int_{\mathbb{R}^{n+1}_+}t^\beta |u_i-u_j|^{p^*} dydt\big)^{\frac{2}{p^*}}\leq C \int_{\mathbb{R}^{n+1}_+} t^\alpha |\nabla u_i-\nabla u_j|^2dydt\to 0.$$
	Hence there exist $u\in L^{p^*}_\beta(\mathbb{R}^{n+1}_+)$ and $g_k\in L^2_\alpha(\mathbb{R}^{n+1}_+),\ k=1,\cdots, n+1$, such that as $ j\to \infty,$
	$$u_j\to u,\  \text{ in }L^{p^*}_\beta(\mathbb{R}^{n+1}_+),$$
	$$\partial_k u_j\to g_k,\ \text{ in }L^2_\alpha(\mathbb{R}^{n+1}_+),$$
where $ \partial_{n+1}=\partial_t, \partial_{k}=\partial_{y_k},\ k=1,\cdots, n$.

For any $K\subset\subset\mathbb{R}^{n+1}_+$, we have $dist(K,\partial \mathbb{R}^{n+1}_+)>0$. Using  H\"{o}lder inequality, we obtain
	\begin{align*}
		\int_{K} |g_k| dydt
		\leq (\int_{\mathbb{R}^{n+1}_+}t^\alpha | g_k|^2dydt)^{\frac12}(\int_{K}t^{-\alpha} dydt)^{\frac12}
			\leq C(K).
	\end{align*}That is, $g_k\in L^1_{loc}(\mathbb{R}^{n+1}_+)$.

We claim that  the weak derivative in distribution $\nabla u$ of $u$  in $\mathbb{R}^{n+1}_+$ is $(g_1,\cdots,g_{n+1})$.
In fact, for any $\phi\in C^{\infty}_0(\mathbb{R}^{n+1}_+)$, we have  $dist(supp\phi,\partial \mathbb{R}^{n+1}_+)>0$,  and
\begin{eqnarray*}
|\int_{\mathbb{R}^{n+1}_+} (\partial_k u_j-g_k)\phi dydt|
&\leq &(\int_{\mathbb{R}^{n+1}_+}t^\alpha |\partial_k u_j- g_k|^2dydt)^{\frac12}(\int_{\text{supp}\phi}t^{-\alpha} |\phi|^2dydt)^{\frac12}\\
&	\leq& C(\phi)(\int_{\mathbb{R}^{n+1}_+}t^\alpha |\partial_k u_j- g_k|^2dydt)^{\frac12}\to 0,
\end{eqnarray*}as $j\to\infty$.
Similarly,
\begin{eqnarray*}
|\int_{\mathbb{R}^{n+1}_+} ( u_j-u)\partial_k \phi dydt|
	\leq C(\phi)(\int_{\mathbb{R}^{n+1}_+}t^\beta | u_j- u|^{p^*}dydt)^{\frac12}\to 0,
\end{eqnarray*}as $j\to\infty$.
Noting that
$$\int_{\mathbb{R}^{n+1}_+} (\partial_k u_j )\phi dydt=-\int_{\mathbb{R}^{n+1}_+} u_j\partial_k \phi dydt,$$
we have
$$\int_{\mathbb{R}^{n+1}_+} g_k \phi dydt=-\int_{\mathbb{R}^{n+1}_+} u \partial_k \phi dydt,$$
as $j\to \infty $. This implies that in $\mathbb{R}^{n+1}_+$, the weak derivative $\nabla u$ of $u$ is $(g_1,\cdots,g_{n+1})$. Therefore the limit $u$ of $u_j$ under norm $\|\cdot\|_{{\cal D}^{1,2}_{\alpha}(\mathbb{R}^{n+1}_+)}$ is in $ {\cal K}_{\alpha,\beta}$.
$	{\cal D}^{1,2}_{\alpha}(\mathbb{R}^{n+1}_+)\subset {\cal K}_{\alpha,\beta}$ is proved.

\smallskip

$(2)$ Next, we show that for any $u\in W^{1,1}_{loc}(\mathbb{R}^{n+1}_+)$ satisfying
	\begin{align}\label{2-1}\int_{\mathbb{R}^{n+1}_+}t^\alpha|\nabla u|^2 dydt\leq C, \int_{\mathbb{R}^{n+1}_+}t^\beta |u|^{p^*} dydt\leq C,
	\end{align}
and any $\varepsilon>0$, there is $\tilde{u}\in C^{\infty}_0(\overline{\mathbb{R}^{n+1}_+})$, such that
	$$\int_{\mathbb{R}^{n+1}_+}t^\alpha|\nabla\tilde{u}-\nabla u|^2 dydt<\varepsilon.$$
	

We complete our proof through the following three steps.

\textbf{Step $1$}. We first show that  $u$ can be approximated by a function with compact support.

By \eqref{2-1}, there is $R>0$, such that
\begin{align}\label{2-2}\int_{\mathbb{R}^{n+1}_+\backslash B_R(0)}t^\alpha|\nabla u|^2 dydt< \varepsilon,\ \ \int_{\mathbb{R}^{n+1}_+\backslash B_R(0)}t^\beta |u|^{p^*} dydt< \varepsilon^{\frac{p^*}{2}}.
\end{align}
We claim that
\begin{equation}\label{a-b-holder}
\int_{\Omega}t^\alpha|u|^2dydt\leq C r^2(\int_{\Omega}t^\beta|u|^{p^*}dydt)^{\frac{2}{p^*}}
\end{equation}for $\alpha\leq \beta+2$, $\forall r>1$ and  $\Omega\subset \overline{B^+_r(0)}$.
Moreover, taking $\Omega=B_{2R}^+(0)\backslash B^+_{R}(0)$, and by \eqref{2-2}, we have
\begin{equation}\label{2-3}
\int_{B_{2R}^+(0)\backslash B^+_{R}(0)}t^\alpha|u|^2dydt< C R^2\varepsilon.
\end{equation}
In fact, for $\alpha=\beta+2$, we have $p^*=2$, and it is easy to check that  \eqref{a-b-holder} holds. For $\alpha<\beta+2$,  we have $p^*>2$. By H\"{o}lder inequality and $	\big(\frac{\alpha p^*-2\beta}{p^*-2}+n+1\big)\cdot\frac{p^*-2}{p^*}=2$, we obtain
\begin{align*}
\int_{\Omega}t^\alpha|u|^2dydt
\leq (\int_{\Omega}t^\beta|u|^{p^*}dydt)^{\frac{2}{p^*}}(\int_{B_{r}^+(0)}t^{\frac{\alpha p^*-2\beta}{p^*-2}} dydt)^{\frac{p^*-2}{p^*}}
\leq C r^2(\int_{\Omega}t^\beta|u|^{p^*}dydt)^{\frac{2}{p^*}}.
\end{align*}

Define a cut-off function $\eta_R\in C^\infty_0(\overline{\mathbb{R}^{n+1}_+})$ satisfying $0\leq \eta_R\leq 1$, $\eta_R=1 $ in $B_R^+(0)$, $\eta_R=0$ in $\mathbb{R}^{n+1}_+\backslash B_{2R}^+(0)$ and $|\nabla \eta_R|\leq \frac{C}{R} $. By \eqref{2-2} and \eqref{2-3}, it yields
\begin{eqnarray*}
	&&	\int_{\mathbb{R}^{n+1}_+}t^\alpha|\nabla(\eta_R u)-\nabla u|^2 dydt\\
	&\leq&2\int_{\mathbb{R}^{n+1}_+\backslash B_{R}(0)}t^\alpha|\nabla u|^2 dydt
	+2\int_{B_{2R}^+(0)\backslash B_{R}(0)}t^\alpha u^2|\nabla\eta_R|^2dydt\\
	&<&C\varepsilon.
\end{eqnarray*}
Therefore, without loss of generality, we can assume that $u$ has a compact support and $K:=\text{supp}u\subset \overline{B_{R-1}^+(0)}\subset\overline{B_{R}^+(0)}$. Similar to \eqref{2-3}, we have
 \begin{equation}\label{2-4}
\int_{K}t^\alpha |u|^2dydt\leq C(R).
\end{equation}

\textbf{Step $2$}. We show that $u$ can be approximated by a continuous and piecewise smooth function.

If $u$ has a compact support, it follows from \eqref{2-1} and \eqref{2-4} that, for any $\varepsilon>0$, there is $\delta\in(0,\varepsilon)$ small enough, such that
\begin{align}\label{e-d-1}
		\int_{\{0<t<2\delta \}}t^\alpha |\nabla u|^2 dydt<\varepsilon,\ \int_{\{0<t<2\delta \}}t^\alpha |u|^2 dydt<\varepsilon,
	\end{align}
where $\{0< t<2\delta\}$ is denoted as $\{(y,t):\ y\in\mathbb{R}^n, \ 0< t< 2\delta\} $  for short. Define  the mollification of $u$ as
\begin{equation}\label{mollifier-1}
u_\tau(y,t)=(\zeta_\tau\ast u)(y,t)=\int_{B_\tau(0)}\zeta_{\tau}(x,s)u(y-x,t-s)dxds,\ y\in\mathbb{R}^{n},\ t>\delta, \tau<\delta,
\end{equation}
where $\zeta_\tau(Y)=\frac{1}{\tau^{n+1}}\zeta(\frac{Y}{\tau})\ $, $\zeta\in C^{\infty}_0(\mathbb{R}^{n+1})$, supp$\zeta=\overline{B_1(0)}$, and $\zeta$ is a radially symmetric function with  $\int_{B_1(0)}  \zeta(Y)dY=1$. Obviously, $u_\tau\in C^{\infty}(\{t>\delta\})$ for $\tau<\delta$ .

Claim: for $\tau$ small enough, it holds
\begin{align}\label{t-d-1}
\|\nabla u_\tau-\nabla u\|^2_{L^2_\alpha(\{t>\delta\})}=\int_{\{t>\delta\}}t^\alpha |\nabla u_\tau-\nabla u|^2 dydt<\varepsilon.
\end{align}
Noting that $u\in L^2_{\alpha}(\{t>\frac{\delta}{2}\})$ due to \eqref{2-4}, we know that  there exists $\bar{u}\in C_0(\{t>\frac{\delta}{4}\})$ such that
$$\|\bar{u}-u\|^2_{L^2_\alpha(\{t>\frac{\delta}{2}\})}<\varepsilon.$$
Denoting $\tilde{u}_\tau=\zeta_\tau\ast \tilde{u}$ and using the triangle inequality, we have
\begin{align}\label{tri}
\|u_\tau-u\|_{L^2_\alpha(\{t>\delta\})}\leq \|u_\tau-\bar{u}_\tau\|_{L^2_\alpha(\{t>\delta\})}+
\|\bar{u}_\tau-\bar{u}\|_{L^2_\alpha(\{t>\delta\})}+\|\bar{u}-u\|_{L^2_\alpha(\{t>\delta\})}.
\end{align}
For $0<\tau<\frac{\delta}{2}$,  by H\"{o}lder inequality, $\int_{B_\tau(0)}  \zeta_\tau(y,t)dydt=1$ and $\alpha>0$, one has
\begin{align*}
\int_{\{t>\delta\}} t^\alpha |u_\tau|^2 dydt
\leq &\int_{\{t>\delta\}}t^\alpha ( \int_{B_\tau(0)} \zeta_\tau(x,s)|u(y-x,t-s)|^2 dxds \big)dydt\nonumber\\
\leq&\int_{B_\tau(0)} \zeta_\tau(x,s)\big(\int_{\{t>\frac{\delta}{2}\}}(t+s)|^\alpha|u(y,t)|^2 dydt\big) dxds\nonumber\\
\leq& 2^\alpha \int_{\{t>\frac{\delta}{2}\}}t^\alpha |u|^2 dydt.
\end{align*}
It follows
$$\|u_\tau-\bar{u}_\tau\|^2_{L^2_\alpha(\{t>\delta\})}\leq 2^\alpha\|u-\bar{u}\|^2_{L^2_\alpha(\{t>\frac{\delta}{2}\})}< C\varepsilon.$$
However, since $\bar{u}$ is a continuous function with compact support, we have $\bar{u}_\tau\to\bar{u}$ in $C_0(\{t>\frac{\delta}{8}\})$ as $\tau\to0$. Then
$$\|\bar{u}_\tau-\bar{u}\|^2_{L^2_\alpha(\{t>\delta\})}<\varepsilon$$
 for $\tau$ small enough.
Combining the above with \eqref{tri} we have $$\|u_\tau-u\|^2_{L^2_\alpha(\{t>\delta\})}<C\varepsilon.$$
Noting that $\nabla (u_\tau)=(\nabla u)_\tau$ and  $\nabla u\in L^2_{\alpha}(\{t>\frac{\delta}{2}\})$, similar to the above argument, we obtain \eqref{t-d-1}.

Fix a sufficiently small $\tau$ such that \eqref{t-d-1} holds, and define
\begin{equation}
\hat{u}(y,t)=\begin{cases}
u_\tau(y,t),& y\in\mathbb{R}^n,\ t\geq\delta,\\
u_\tau(y,2\delta-t),& y\in\mathbb{R}^n,\ 0\leq t<\delta.
\end{cases}
\end{equation}
Obviously, $\hat{u}$ is a continuous and smooth piecewise function with supp$\hat{u}\subset\subset\overline{\mathbb{R}^{n+1}_+}$. For $\alpha> 0$, we know from \eqref{e-d-1} and \eqref{t-d-1} that
\begin{align*}
\int_{\{0<t<\delta\}}t^\alpha|\nabla \hat u|^2dydt
=&\int_{\{\delta<t<2\delta\}}(2\delta-t)^\alpha|\nabla u_\tau |^2dydt
\leq \int_{\{\delta<t<2\delta\} } t^\alpha|\nabla u_\tau |^2dydt\\
\leq &2\int_{\{\delta<t<2\delta\} } t^\alpha|\nabla u_\tau -\nabla u|^2dydt+2\int_{\{\delta<t<2\delta\} } t^\alpha|\nabla u|^2dydt\\
< &C\varepsilon,
\end{align*}
and then
\begin{align*}
&\int_{\mathbb{R}^{n+1}_+}t^\alpha|\nabla \hat u-\nabla u|^2dydt\\
\leq &\int_{\{t>\delta\}}t^\alpha|\nabla  u_\tau-\nabla u|^2dydt+2\int_{\{0<t<\delta\}}t^\alpha|\nabla \hat u|^2dydt+2\int_{\{0<t<\delta\}}t^\alpha|\nabla  u|^2dydt\\
< &C\varepsilon.
\end{align*}

\textbf{Step $3$}.  We prove that $u$ can be approximated by a $C^{\infty}_0(\mathbb{R}^{n+1})$ function which is  symmetric with respect to $t$.

Let $\hat{u}$ be defined in Step $2$. Assuming $\text{supp}\hat{u}\subset \overline{B_{R-1}^+(0)}\subset\overline{B_{R}^+(0)}$, we have
\begin{align}\label{hat-1}
\int_{B_{R}^+(0)}t^\alpha|\nabla \hat{u}|^2 dydt\leq C,\  \int_{B_{R}^+(0)}t^\alpha|\hat{u}|^{2} dydt\leq C.
\end{align}
According to Fubini theorem, there exists $0<\delta_1<\delta<\varepsilon$ small enough, such that
\begin{align}\label{fubini}
\int_{B_R^n(0)}\delta_1^\alpha |\hat{u}(y,\delta_1)|^2 dy\leq C,\,\int_{B_R^n(0)}\delta_1^\alpha |\nabla \hat{u}(y,\delta_1)|^2 dy\leq C,
\end{align}
and
\begin{align}\label{fubini-1}
\int_{\{0\leq t\leq\delta_1\}}t^\alpha |\hat{u}|^2dydt<\varepsilon,\, \int_{\{0\leq t\leq\delta_1\}}t^\alpha |\nabla\hat{ u}|^2dydt<\varepsilon.
\end{align}
Set
\begin{equation}\label{modify}
v(y,t)=
\begin{cases}
\hat{u}(y,t), &y\in\mathbb{R}^n,\  t>\delta_1,\\
\hat{u}(y,\delta_1),&y\in\mathbb{R}^n,\  0\leq t\leq\delta_1.
\end{cases}
\end{equation}
Obviously, $v$ is well defined since $\hat{u}$ is a continuous and smooth piecewise function. It follows from \eqref{hat-1}, \eqref{fubini} and \eqref{fubini-1} that
\begin{align}\label{v-cut}
\int_{B_R^+(0)}t^\alpha|\nabla v-\nabla \hat{u}|^2dydt
\leq&2\int_{\{0\leq t\leq\delta_1\}}t^\alpha |\nabla \hat{u}|^2dydt+C\delta_1^{\alpha+1}\int_{B_R^n(0)} |\nabla \hat{u}(y,\delta_1)|^2 dy\nonumber\\
< & C\varepsilon,
\end{align}
and
$$\int_{B_R^+(0)}t^\alpha |\nabla v|^2dydt\leq C,\,\int_{B_R^+(0)}t^\alpha |v|^2dydt\leq C.$$
Extend $v$ evenly with respect to $t$, then supp$v\subset\subset B_R(0)\subset B_R^n(0)\times [-R, R].$
Define  the mollification of $v$ as
\begin{equation}\label{mollifier}
v_\rho(y,t)=(\zeta_\rho\ast v)(y,t)
,\ (y,t)\in\mathbb{R}^{n+1},
\end{equation}
For $\tau$ small enough, we have $v_\tau\in C_0^{\infty}(B_R(0))$, and $v_\rho$ is even with respect to  $t$ .

Claim:  for sufficiently small $\rho$, it holds
\begin{align}\label{v-rho}
\int_{B_R(0)} |t|^\alpha |\nabla v_\rho-\nabla v|^2dydt< C\varepsilon.
\end{align}
Similar to the proof of \eqref{t-d-1}, we only need to show
\begin{align}\label{mo}
\int_{B_R(0)}|t|^\alpha |v_\rho|^2dydt\leq C\int_{B_R(0)}|t|^\alpha |v|^2dydt.
\end{align}
In fact, for $0<\tau<\delta_1$, by H\"{o}lder inequality we have
\begin{align}\label{v-rho-1}
	\int_{B_R(0)}|t|^\alpha |v_\rho(y,t)|^2dydt
\leq&\int_{B_R(0)}|t|^\alpha\big( \int_{B_\rho(0)} \zeta_\rho(x,s)|v(y-x,t-s)|^2 dxds \big)dydt\nonumber\\
=&\int_{B_\rho(0)} \zeta_\rho(x,s)\big(\int_{B_R(0)}|t+s|^\alpha|v(y,t)|^2 dydt\big) dxds\nonumber\\
\leq&2^{\alpha+1}\int_{B_\rho(0)}  \zeta_\rho(x,s)\Big( \int_{B_R^n(0)\times [\delta_1, R]}|t|^\alpha|v(y,t)|^2 dydt\nonumber\\
&+ \delta_1^{\alpha}\int_{B_R^n(0)\times[0,\delta_1]}|\hat{u}(y,\delta_1)|^2dydt\Big)dxds,
\end{align}
where
\begin{align*}
\delta_1^{\alpha}\int_{B_R^n(0)\times[0,\delta_1]}|\hat{u}(y,\delta_1)|^2dydt
=&\delta_1^{\alpha+1}\int_{B_R^n(0)}|\hat{u}(y,\delta_1)|^2dy\\
=& (\alpha+1)\int_{B_R^n(0)\times[0,\delta_1]}|t|^\alpha |v(y,t)|^2dydt.
\end{align*}
Plugging the above into \eqref{v-rho-1} we obtain \eqref{mo}.

Noting that $\nabla v_\rho=\nabla(\zeta_\rho \ast v)=\zeta_\rho\ast \nabla v$, $\nabla v\in L^2_{\alpha}(B_R^+(0))$, and  for $0\leq t\leq \delta_1$,
	$$(\partial_{y_k})v(y,t)=(\partial_{y_k})v(y,\delta_1),\,(\partial_{t})v(y,t)=0,\ k=1,2,\cdots,n,$$
similar to \eqref{t-d-1}, we can obtain \eqref{v-rho}.

 From the above three steps, we can  choose $\tilde{u}=v_\rho$  for $\rho$ small enough, and hereby complete the proof.
\end{proof}

\medskip

For $\alpha\geq 1$, we also have the following density lemma.
\begin{lemma}\label{appro}
If $\alpha\geq 1$,  then ${\cal D}^{1,2}_{\alpha}(\mathbb{R}^{n+1}_+)$ is  the completion of the space $C^{\infty}_0(\mathbb{R}^{n+1}_+)$ under the norm $\|\cdot\|_{{\cal D}^{1,2}_{\alpha}(\mathbb{R}^{n+1}_+)}$.
\end{lemma}
\begin{proof}
By the definition of ${\cal D}^{1,2}_{\alpha}(\mathbb{R}^{n+1}_+)$, we  only need to show: for any $u\in C^{\infty}_0(\overline{\mathbb{R}^{n+1}_+})$ and  any $\varepsilon>0$, there exists $\tilde{u}\in C^{\infty}_0(\mathbb{R}^{n+1}_+)$ such that $\int_{\mathbb{R}^{n+1}_+}t^\alpha |\nabla \tilde{u}-\nabla u|^2 dydyt<\varepsilon$.
	
For $0<\delta<1$, define
\begin{equation*}
f_\delta (t)=\begin{cases}
0,& 0\leq t\leq \delta^2,\\
\frac{\ln(t/\delta^2)}{\ln(1/\delta)},& \delta^2<t<\delta,\\
1,& t\geq \delta.
\end{cases}
\end{equation*}
Since $u\in C^{\infty}_0(\overline{\mathbb{R}^{n+1}_+})$, there exists some $M>0$ such that $|u|\leq M$ in $\overline{\mathbb{R}^{n+1}_+}$, and  for any $\varepsilon>0$, there exists $0<\delta_0<1$ such that for $0<\delta<\delta_0$, it holds
$$\int_{\{0\leq t\leq \delta\}} t^\alpha |\nabla u|^2 dydt<\varepsilon.$$
 It follows that for $0<\delta<\delta_0$,
\begin{align*}
\int_{\mathbb{R}^{n+1}_+}t^\alpha |\nabla(f_\delta u)-\nabla u|^2 dydt
=&\int_{\mathbb{R}^{n+1}_+}t^\alpha |(f_\delta-1) \nabla u+u\nabla f_\delta|^2 dydt\\
\le&2\int_{\mathbb{R}^{n+1}_+}t^\alpha\big( |(f_\delta-1)|^2| \nabla u|^2+|f'_\delta(t)|^2  |u|^2\big) dydt\\
\leq &2\int_{\{0\leq t\leq \delta\}} t^\alpha |\nabla u|^2 dydt + 2\int_{\{\delta^2\leq t\leq \delta\}}t^\alpha |f_{\delta}'(t)|^2 u^2 dydt\\
\leq &2\varepsilon+M\int_{\{\delta^2\leq t\leq \delta\}\cap\text{supp u}}t^\alpha\frac{1}{t^2\ln^2\delta}dydt\\
\leq&\begin{cases}
2 \varepsilon+C\frac{\delta^{\alpha-1}}{\ln^2\delta},\ &\alpha>1,\\
2\varepsilon+\frac{C}{-\ln\delta},\ & \alpha=1.
\end{cases}
\end{align*}
Since $$\lim_{\delta\to0}\frac{\delta^{\alpha-1}}{\ln^2\delta}=0\ (\alpha>1),\quad \lim_{\delta\to 0}\frac{1}{-\ln\delta}=0,$$
there exists $0<\delta_1<\delta_0$ such that for $0<\delta<\delta_1$, $$\frac{\delta^{\alpha-1}}{\ln^2\delta}<\varepsilon (\alpha>1),\quad \frac{1}{-\ln\delta}<\varepsilon.$$
We then get
\begin{align}\label{appro-1}
\int_{\mathbb{R}^{n+1}_+}t^\alpha |\nabla(f_\delta u)-\nabla u|^2 dydt< C\varepsilon.
\end{align}
Noting that $\text{supp}(f_\delta u)\subset\{(y,t): \ y\in\mathbb{R}^n,\ t\geq \delta^2\}$, we can define the mollification  of $f_\delta u$ as
$$(f_\delta u)_\tau=\zeta_{\tau}\ast (f_\delta u),$$
where $\zeta_\tau$ is the  convolution kernel defined in \eqref{mollifier-1}.
Obviously, $(f_\delta u)_\tau\in C^{\infty}_0(\mathbb{R}^{n+1}_+)$ for $0<\tau<\delta^2$, and  similar to Lemma \ref{density}, it is easy to verify
$$\int_{\mathbb{R}^{n+1}_+}t^\alpha |\nabla ((f_\delta u)_{\tau})-\nabla (f_\delta u)|^2 dydyt<\varepsilon$$
 for $\tau$ small enough.
Combining the above and \eqref{appro-1} we  can choose  $\tilde{u}=(f_\delta u)_{\tau}$ to get the desired result.
\end{proof}

\bigskip

For $\Omega\subset\overline{\mathbb{R}^{n+1}_+},$  define
$$\|u\|_{\alpha,\Omega}:=(\int_{\Omega}t^\alpha(|\nabla u|^2+|u|^2)dydt)^{\frac{1}{2}},$$
then $${\cal D}^{1,2}_{\alpha,loc}(\mathbb{R}^{n+1}_+)=\{u\in W^{1,1}_{loc}(\mathbb{R}^{n+1}_+)\,:\,\|u\|_{\alpha,\Omega}<\infty\ \forall \Omega\subset\subset\overline{\mathbb{R}^{n+1}_+}\}.$$

\begin{remark}\label{p-loc}
	Assume that $\alpha,\ \beta$ satisfy \eqref{beta-1}, and $u\in{\cal D}^{1,2}_{\alpha,loc}(\mathbb{R}^{n+1}_+)$, then we have $u\in L^{p}_{\beta,loc}(\overline{\mathbb{R}^{n+1}_+})$ where $p\leq p^{*}$. In fact, for any $\Omega\subset\subset\overline{\mathbb{R}^{n+1}_+}$, there exists $R>0$ such that $\Omega\subset\subset B_R(0)$. Choose a cut-off function $0\leq \eta\leq 1$ with $\eta|_\Omega=1$, $\eta|_{\mathbb{R}^{n+1}_+\backslash B_R(0)}=0$ and $|\nabla \eta|\leq C$. By weighted Sobolev inequality \eqref{GGN-2}, we have
	\begin{equation}\label{H1}
	(\int_{\Omega}t^\beta |u|^{p^*}dydt)^{\frac{2}{p^*}}\leq C\int_{B_R^+(0)}t^\alpha (|\nabla u|^2+|u|^2)dydt,
	\end{equation}
	then $u\in L^{p^*}_{\beta,loc}(\overline{\mathbb{R}^{n+1}_+})$. For $p\leq p^*$, by H\"{o}lder inequality  we have
	$u\in L^{p}_{\beta,loc}(\overline{\mathbb{R}^{n+1}_+})$.
\end{remark}

\begin{lemma}\label{CompE}
	Assume that $\alpha$ and $ \beta$ satisfy \eqref{beta-2}.
Then, for $1\leq p<p^*$, the embedding ${\cal D}^{1,2}_{\alpha,loc}(\mathbb{R}^{n+1}_+) \hookrightarrow L^{p}_{\beta,loc}(\overline{\mathbb{R}^{n+1}_+})$ is compact.
\end{lemma}
\begin{proof}
Assume that $\{u_i\}\subset {\cal D}^{1,2}_{\alpha,loc}(\mathbb{R}^{n+1}_+)$ satisfies
$$\int_{B_R^+(0)}t^\alpha \big(|\nabla u_i|^2+|u_i|^2\big)dydt\leq C(R),$$
for any $R>0$. For $1\leq p<p^*$, we will show that there exists a subsequence, still denoting it as $\{u_i\}$, such that
\begin{align}\label{converge}
\int_{B_R^+(0)}t^\beta|u_i-u_j|^{p}dydt\to 0, \quad\text{ as }i,j\rightarrow\infty
\end{align}
for any $R>0$. We shall consider two cases.
	
$(1)$ The case of $\beta=\alpha>0$.	

$(i)$ If $\alpha=m$ is a positive integer, we study the convergence on  $B_R^n(0)\times[0,R]$.
For any $i\in\mathbb{N}^*$, similar to Lemma \ref{density}, there exists $\tilde{u}_i\in C^\infty(B_{2R}(0))$ which is even with respect to $t$, such that
$$\int_{B_{2R}(0)}t^m \big(|\nabla \tilde{u}_i-\nabla u_i|^2+|\tilde{u}_i-u_i|^2\big)dydt<\frac{1}{i}.$$
By Remark \ref{p-loc}, we have
\begin{align}\label{converge-1}
\big(\int_{B_R^n(0)\times[0,R]}t^m |\tilde{u}_i-u_i|^p dydt\big)^{\frac{2}{p}}\leq C\int_{B_{2R}(0)}t^m \big(|\nabla \tilde{u}_i-\nabla u_i|^2+|\tilde{u}_i-u_i|^2\big)dydt<\frac{C}{i}.
\end{align}
It is also easy to see that $\|\tilde{u}_i\|_{m, B_R^n(0)\times[0,R]}\leq C(R)$. 
 Define
$$v_i(y,z)=\tilde{u}_i(y,t),\ \ y\in B_R^n(0),\
	 t\in [0,R],\ z\in \overline{B^{m+1}_R(0)},\ |z|=t.$$
Noting that $dz=t^mdtdS_{m}$ in $B^{m+1}_R(0)$, we have
$$\int_{B_R^n(0)\times  B^{m+1}_R(0)}|v_i(y,z)|^2 dydz=(m+1)\omega_{m+1}\int_{B_R^n(0)\times[0,R]}t^{m}|\tilde{u}_i(y,t)|^2 dydt,$$
and
$$\int_{B_R^n(0)\times  B^{m+1}_R(0)}|\nabla v_i(y,z)|^2 dydz=(m+1)\omega_{m+1}\int_{B_R^n(0)\times[0,R]}t^{m}|\nabla \tilde{u}_i(y,t)|^2 dydt,$$
where $w_{m+1}$ is the volume of unit ball in  $\mathbb{R}^{m+1}$.
Thus $\{v_i\}$  is a bounded sequence in $H^{1,2}(B_R^n(0)\times B^{m+1}_R(0))$. Since $1\leq p<p^*=\frac{2(n+m+1)}{n+m-1}$, by classical compact embedding  in Sobolev space, there is a convergent subsequence (still denoting it as $\{v_i\}$) in  $L^p(B_R^n(0)\times  B^{m+1}_R(0))$. That is
$$\int_{B_R^n(0)\times  B^{m+1}_R(0)}|v_i(y,z)-v_j(y,z)|^p dydz\rightarrow 0, \quad\text{ as }i,j\rightarrow\infty,$$
which implies,
$$\int_{B_R^n(0)\times[0,R]}t^{m}|\tilde{u}_i-\tilde{u}_j|^p dydt\rightarrow0,\quad\text{as }i,j\rightarrow\infty.$$
Therefore by \eqref{converge-1}, we can obtain \eqref{converge} with $\beta=\alpha=m$.

$(ii)$ If $\alpha>0$ is not an integer, then there is a positive integer $m$, such that $m-1< \alpha<m$. Then we have
$$\|u_i\|^2_{m,B_R^+(0)}\leq R^{m-\alpha}\|u_i\|^2_{\alpha,B_R^+(0)}\leq C(R).$$
For $1\leq p_1<\frac{2(n+m+1)}{n+m-1}\big(<p^*=\frac{2(n+\alpha+1)}{n+\alpha-1}\big)$, we know from $(i)$ that there is a convergent subsequence (still denoting it as $\{u_i\}$)  in $L^{p_1}_m(B_R^+(0))$.	
For $1\le p<p^*$, choosing $\alpha_2\in(m-1,\alpha), \ \theta=\frac{\alpha-\alpha_2}{m-\alpha_2},$ and $  p_2=p-\frac{\alpha-\alpha_2}{m-\alpha}(p_1-p)$, we have $$\alpha=\theta m+(1-\theta)\alpha_2,\  p=\theta p_1+(1-\theta)p_2.$$
Using interpolation  inequality we have
\begin{align}\label{holder-1}
\int_{B_R^+(0)}t^{\alpha}|u_i-u_j|^p dydt
\leq \big(\int_{B_R^+(0)}t^{m}|u_i-u_j|^{p_1} dydt \big)^{\theta}\big(\int_{B_R^+(0)}t^{\alpha_2}|u_i-u_j|^{p_2} dydt \big)^{1-\theta}.
\end{align}
If
\begin{align}\label{alpha_2}
\int_{B_R^+(0)}t^{\alpha_2}|u_i-u_j|^{p_2} dydt\leq C(R),
\end{align}
then $\{u_i\}$ converges in $L^{p}_\alpha(B_R^+(0))$.
Next, we choose a suitable $\alpha_2$ such that \eqref{alpha_2} holds. Since	 $$\lim_{\alpha_2\to\alpha}p_2=p<p^*=\lim_{\alpha_2\to\alpha}\frac{2(n+\alpha_2+1)}{n+\alpha-1},$$
 we can choose  an  $\alpha_2$ sufficiently close to $\alpha$ such that $p_2<\frac{2(n+\alpha_2+1)}{n+\alpha-1}$.
Noting that $B_R^+(0)$ is bounded, invoking H\"older inequality  and \eqref{H1} we have
\begin{eqnarray*}
\big(\int_{B_R^+(0)}t^{\alpha_2}|u_i-u_j|^{p_2} dydt\big)^{\frac2{p_2}}
&\leq&C(R)\big(\int_{B_R^+(0)}t^{\alpha_2}|u_i-u_j|^{\frac{2(n+\alpha_2+1)}{n+\alpha-1}} dydt\big)^{\frac{n+\alpha-1}{n+\alpha_2+1}}\nonumber\\
&\leq &C(R)\|u_i-u_j\|^2_{\alpha, B_{R+1}^+(0)}
\leq C(R).
\end{eqnarray*}
Back to \eqref{holder-1}, we obtain \eqref{converge} with $\beta=\alpha$.

\medskip

$(2)$ We consider the general case. If $\beta>\alpha$, we choose $\gamma$ with $n\geq2,\ \beta<\gamma<\frac{n+1}{n-1}\alpha$ or $n=1, \ \gamma>\beta$. If $\beta<\alpha$, we choose $\gamma$ with $\max\{-1,\alpha-2,-\alpha\}<\gamma<\beta$.
It follows from \eqref{beta-2} that
$$\gamma>-1,\ \alpha+\gamma>0,\ \frac{n-1}{n+1}\gamma<\alpha<\gamma+2.$$
Setting $\beta=\theta\alpha+(1-\theta)\gamma$ with $\theta\in(0,1)$, we can write
$$p^*=\frac{2(n+\beta+1)}{n+\alpha-1}=\theta \frac{2(n+\alpha+1)}{n+\alpha-1}+(1-\theta)\frac{2(n+\gamma+1)}{n+\alpha-1}.$$
For $1\leq p<p^*$, take $p_3$ satisfying $p=\theta p_3+(1-\theta)\frac{2(n+\gamma+1)}{n+\alpha-1}$, then $1\leq p_3<\frac{2(n+\alpha+1)}{n+\alpha-1}$. By interpolation inequality and \eqref{H1} we have
\begin{align*}
\int_{B_R^+(0)}t^\beta |u|^p dydt \leq & \big(\int_{B_R^+(0)}t^\alpha |u|^{p_3} dydt\big)^\theta \big(\int_{B_R^+(0)}t^\gamma |u|^{\frac{2(n+\gamma+1)}{n+\alpha-1}} dydt\big)^{1-\theta} \\
\leq& C\big(\int_{B_R^+(0)}t^\alpha |u|^{p_3} dydt\big)^\theta \|u\|^{\frac{2(n+\gamma+1)(1-\theta)}{n+\alpha-1}}_{\alpha,\tilde{B_{R+1}^+(0)}}.
\end{align*}
Since the embedding  from ${\cal D}^{1,2}_{\alpha,loc}(\mathbb{R}^{n+1}_+)$ to $L^{p_3}_{\alpha,loc}(\overline{\mathbb{R}^{n+1}_+})$ is compact, we obtain that $D^{1,2}_{\alpha,loc}(\mathbb{R}^{n+1}_+)\hookrightarrow L^{p}_{\beta,loc}(\overline{\mathbb{R}^{n+1}_+})$ is  a compact embedding.
\end{proof}

\medskip

Finally, we show that for a weak solution $u$ to \eqref{genequ-1}, its Kelvin transformation $u_{\lambda,b}$ defined in Section 5, is still a "weak solution".
\begin{lemma}\label{weak solution}
Assume that $\alpha,\ \beta$ satisfy \eqref{beta-1}, and  $u\in {\cal D}^{1,2}_\alpha(\mathbb{R}^{n+1}_+)$ is a weak solution to \eqref{genequ-1}. Then $u_{\lambda,b}$ satisfies
 \begin{align}\label{weak-sol-v}
\int_{\mathbb{R}^{n+1}_+} t^\alpha\nabla u_{\lambda,b} \nabla \psi dydt=\int_{\mathbb{R}^{n+1}_+} t^\beta (u_{\lambda,b})^{p^*-1}\psi dydt
\end{align}
for any $\psi\in {\cal D}^{1,2}_{\alpha}(\mathbb{R}^{n+1}_+)$ with $supp\psi\subset\overline{\mathbb{R}^{n+1}_+}\backslash\{0\}$.
\end{lemma}

\begin{proof} 
	 Without loss of generality,  we can assume $\lambda=1, \ b=0$. 

$ 1)$ We first show that \eqref{weak-sol-v} holds for any $\psi\in C^\infty_0(\overline{\mathbb{R}^{n+1}_+}\backslash\{0\})$. Assume $ supp \psi\subset\overline{B_{R}^+(0)}\backslash B_{1/R}(0)$ with $R>1.$ A direct computation yields
\begin{align}\label{T-1}
\nabla  u_{1,0}(y,t)=&\frac{1}{|(y,t)|^{n+\alpha+1}}(\nabla  u)(\frac{(y,t)}{|y,t|^2})-\frac{2}{|(y,t)|^{n+\alpha+1}}\big[(\nabla u)(\frac{(y,t)}{|y,t|^2})\cdot\frac{(y,t)}{|y,t|^2}\big](y,t)\nonumber\\
&-\frac{n+\alpha-1}{|(y,t)|^{n+\alpha+1}}u(\frac{(y,t)}{|y,t|^2})(y,t).
\end{align}
Write $(x,s)=\frac{(y,t)}{|(y,t)|^2}$. Then by \eqref{T-1} and $\frac{t^\alpha}{|(y,t)|^{n+\alpha+1}}dydt=\frac{s^\alpha}{|(x,s)|^{n+\alpha+1}}dxds$, the left side of \eqref{weak-sol-v} can be written as
\begin{eqnarray*}
&&\int_{\mathbb{R}^{n+1}_+} t^{\alpha}\nabla  u_{1,0}(y,t)\nabla\psi(y,t)dydt\nonumber\\
&=&\int_{\mathbb{R}^{n+1}_+} \frac{s^\alpha}{|(x,s)|^{n+\alpha+1}}[(\nabla u)(x,s)\cdot(\nabla\psi)(\frac{(x,s)}{|(x,s)|^2})] dxds\nonumber\\
&&-\int_{\mathbb{R}^{n+1}_+}\frac{2s^\alpha}{|(x,s)|^{n+\alpha+1}}[(\nabla u)(x,s)\cdot(x,s)][(\nabla\psi)(\frac{(x,s)}{|(x,s)|^2})\cdot\frac{(x,s)}{|(x,s)|^2}]dxds\nonumber\\
&&	-\int_{\mathbb{R}^{n+1}_+}\frac{(n+\alpha-1)s^\alpha}{|(x,s)|^{n+\alpha+1}}u(x,s)
[(\nabla\psi)(\frac{(x,s)}{|(x,s)|^2})\cdot\frac{(x,s)}{|(x,s)|^2}]dxds.\nonumber\\
\end{eqnarray*}
Applying \eqref{T-1} with $u$ replaced by $\psi$, we have 
\begin{eqnarray*}
&&\int_{\mathbb{R}^{n+1}_+} t^{\alpha}\nabla  u_{1,0}(y,t)\nabla\psi(y,t)dydt\nonumber\\
&=&\int_{\mathbb{R}^{n+1}_+}s^\alpha\nabla u(x,s)\cdot \nabla\Big(\frac{1}{|(x,s)|^{n+\alpha-1}} \psi(\frac{(x,s)}{|(x,s)|^2})\Big)dxds\nonumber\\
&&+\int_{\mathbb{R}^{n+1}_+}\frac{(n+\alpha-1)s^\alpha}{|(x,s)|^{n+\alpha+1}}\big[\nabla \big(u(x,s)\psi(\frac{(x,s)}{|(x,s)|^2})\big)\cdot (x,s)\big]dxds.
\end{eqnarray*}
We claim that
\begin{align}\label{div}
	\int_{\mathbb{R}^{n+1}_+}\frac{s^\alpha}{|(x,s)|^{n+\alpha+1}}\big[\nabla \big(u(x,s)\psi(\frac{(x,s)}{|(x,s)|^2})\big)\cdot (x,s)\big]dxds=0.
\end{align}
Then since $u\in {\cal D}^{1,2}_\alpha(\mathbb{R}^{n+1}_+)$ is a weak solution of \eqref{genequ-1}, it follows that
\begin{align*}
	\int_{\mathbb{R}^{n+1}_+} t^{\alpha}\nabla  u_{1,0}(y,t)\nabla\psi(y,t)dydt
	=&\int_{\mathbb{R}^{n+1}_+} s^\beta u^{p^*-1}(x,s)\frac{1}{|(x,s)|^{n+\alpha-1}} \psi(\frac{(x,s)}{|(x,s)|^2})dxds\\
	=&\int_{\mathbb{R}^{n+1}_+} t^\beta (u_{1,0})^{p^*-1}(y,t) \psi(y,t)dydt.
\end{align*}

In fact, since $ supp \psi\subset\overline{B_{R}^+(0)}\backslash B_{1/R}(0)$ and  $$div\big(\frac{s^\alpha}{|(x,s)|^{n+\alpha+1}}(x,s)\big)=0,$$
it holds
\begin{align*}
	LHS \text{ of \eqref{div}}=\int_{B_{R}^+(0)\backslash B_{1/R}(0)} div\Big(\frac{s^\alpha}{|(x,s)|^{n+\alpha+1}}u(x,s)\psi(\frac{(x,s)}{|(x,s)|^2})(x,s)\Big)dxds.
\end{align*}
Note that $\psi =0 $ on $\partial \big(B_{R}^+(0)\backslash B_{1/R}(0)\big)\cap\mathbb{R}^{n+1}_+$ and $\nu$ denotes the outward unit normal vector field on $\partial \big(B_{R}^+(0)\backslash B_{1/R}(0)\big)\cap\partial \mathbb{R}^{n+1}_+$ which satisfies $\nu_i=0,\ i=1,\cdots,n,\ \nu_{n+1}=-1$. By the divergence theorem, for any $u\in C^1_0(\overline{\mathbb{R}^{n+1}_+})$,
it holds
\begin{align*}
	LHS \text{ of \eqref{div}}
=&-\int_{\partial \big(B_{R}^+(0)\backslash B_{1/R}(0)\big)\cap\partial \mathbb{R}^{n+1}_+}\lim_{s\to 0^+}\frac{s^{\alpha+1} }{|(x,s)|^{n+\alpha+1}}u(x,s)\psi(\frac{(x,s)}{|(x,s)|^2})dS\\
=&0.
\end{align*}
For a general $u\in {\cal D}^{1,2}_{\alpha}(\mathbb{R}^{n+1}_+)$, similar to the proof of Lemma \ref{density},   we can show that for any $\varepsilon>0$, there exists $\tilde{u}\in C^\infty_0(\overline{ \mathbb{R}^{n+1}_+})$ such that
\begin{align}\label{loc-density}
\int_{B_R^+(0)} t^{\alpha}\big(|\nabla \tilde{u}-\nabla u|^2+|\tilde{u}-u|^2\big)dydt< \varepsilon.
\end{align}
We have shown \eqref{div} holds for $\tilde{u}$, then by H\"{o}lder inequality and \eqref{loc-density}, we know that \eqref{div} holds for $u\in {\cal D}^{1,2}_{\alpha}(\mathbb{R}^{n+1}_+)$. And hence \eqref{weak-sol-v} holds for $\psi\in C^\infty_0(\overline{\mathbb{R}^{n+1}_+}\backslash\{0\})$.

\medskip

$2)$ For $u\in {\cal D}^{1,2}_{\alpha}(\mathbb{R}^{n+1}_+)$, it is easy to check that
\begin{align}\label{weak-id-1}
		\int_{\mathbb{R}^{n+1}_+} t^\beta | u_{1,0}|^{p^*}dydt=\int_{\mathbb{R}^{n+1}_+} s^\beta |u|^{p^*}dxds<\infty,
\end{align}
and by \eqref{T-1},
\begin{align*}
		\int_{B_{R}^+(0)\backslash B_{1/R}(0)} t^\alpha|\nabla u_{1,0}|^2dydt\leq C(R),\ \forall R>1.
\end{align*}
Then by approximation, \eqref{weak-sol-v} holds for $\psi\in {\cal D}^{1,2}_\alpha(\mathbb{R}^{n+1}_+)$ with $ supp \psi\subset\subset\overline{\mathbb{R}^{n+1}_+}\backslash \{0\}$. We then use \eqref{weak-sol-v} to prove that $ u_{1,0}\in{\cal D}^{1,2}_{\alpha}(\mathbb{R}^{n+1}_+)$. Hence by approximation, \eqref{weak-sol-v} holds for any $\psi\in{\cal D}^{1,2}_{\alpha}(\mathbb{R}^{n+1}_+)$  with $supp \,\psi\subset\overline{\mathbb{R}^{n+1}_+}\backslash\{0\}$.
	
		For any  $R>1$, choose cut-off function $\eta_R\in C^\infty_0(\overline{\mathbb{R}^{n+1}_+})$ satisfying $supp \eta_R\subset \overline{B_{2R}^+(0)}\backslash B_{1/(2R)}(0)$, $0\leq\eta_R\leq 1$, $\eta_R|_{B_R^+(0)\backslash B_{1/R}(0)}=1$ and $|\nabla \eta_R|\leq\frac{C}{R}$ in $B_{2R}^+(0)\backslash B_R(0)$, $|\nabla \eta_R|\leq CR$ in $B_{1/R}^+(0)\backslash B_{1/(2R)}(0)$. It is easy to check that $\eta_R^2u_{1,0}\in {\cal D}^{1,2}_\alpha(\mathbb{R}^{n+1}_+)$ and $supp\big(\eta_R^2 u_{1,0}\big)\subset\subset\overline{\mathbb{R}^{n+1}_+}\backslash\{0\}$. Take $\psi=\eta_R^2u_{1,0}$ in \eqref{weak-sol-v} with $\lambda=1,\ b=0$, then
		\begin{align*}
		\int_{\mathbb{R}^{n+1}_+} t^{\alpha}\nabla u_{1,0}\cdot \nabla\big(\eta_R^2 u_{1,0}\big)dydt=\int_{\mathbb{R}^{n+1}_+} t^{\beta}\eta_R^2 (u_{1,0})^{p^*} dydt.
		\end{align*}
		Since $\nabla u_{1,0}\cdot \nabla\big(\eta_R^2u_{1,0}\big)\geq \frac{1}{2}\eta_R^2|\nabla u_{1,0}|^2-2|\nabla \eta_R|^2|u_{1,0}|^2$, it holds
		\begin{align}\label{weak-in-1}
		\int_{\mathbb{R}^{n+1}_+} t^{\alpha}\eta_R^2|\nabla u_{1,0}|^2 dydt\leq C\int_{\mathbb{R}^{n+1}_+} t^{\alpha}|\nabla \eta_R|^2|u_{1,0}|^2 dydt+C\int_{\mathbb{R}^{n+1}_+} t^{\beta}\eta_R^2|u_{1,0}|^{p^*} dydt.
		\end{align}
		By the upper bound of $|\nabla \eta_R|$ and \eqref{a-b-holder}, it holds
		\begin{align*}
		&	\int_{\mathbb{R}^{n+1}_+} t^{\alpha}|\nabla \eta_R|^2|u_{1,0}|^2 dydt\\
		\leq& \frac{C}{R^2}\int_{B_{2R}^+(0)\backslash B_R(0)} t^\alpha|u_{1,0}|^2dydt+CR^2\int_{B_{1/R}^+(0)\backslash B_{1/(2R)}(0)}t^\alpha|u_{1,0}|^2dydt\\
		\leq & C\int_{\mathbb{R}^{n+1}_+} t^\beta |u_{1,0}|^{p^*}dydt.
		\end{align*}
		Back to \eqref{weak-in-1} and letting $R\to\infty$, we have
		$$	\int_{\mathbb{R}^{n+1}_+} t^{\alpha}|\nabla u_{1,0}|^2 dydt\leq C\int_{\mathbb{R}^{n+1}_+} t^\beta |u_{1,0}|^{p^*}dydt.$$
		Combining above inequality with \eqref{weak-id-1}, and by Lemma \ref{density}, we have $u_{1,0}\in{\cal D}^{1,2}_{\alpha}(\mathbb{R}^{n+1}_+)$,  we hereby obtain the desired result.
\end{proof}

In the proof of Theorem \ref{L-1}, for $\psi$ defined in \eqref{transform-2}, we have
\begin{lemma}\label{n-bd}
	For $\alpha>0$,  $\lim_{r\rightarrow(\frac{1}{2})^-}(\frac{1}{4}-r^2)^\alpha\frac{\partial\psi}{\partial r}=0.$
\end{lemma}

\begin{proof} From the  regularity of $u$, we have
	$$\psi\in C^2(B_\frac{1}{2}(-\frac{e_{n+1}}{2})) \cap C^0(\overline{B_\frac{1}{2}(-\frac{e_{n+1}}{2})}).$$
	For  $r\in(0,\frac{1}{2})$  close to $\frac{1}{2}$ and   $\delta>0$  small, we take a small arc $$\Gamma=\{x\in \partial B_r(-\frac{e_{n+1}}{2}): r-\frac{1}{2}-\delta<x_{n+1}\leq r-\frac{1}{2}\}.$$
	Let $x$ and $(y,t)$ satisfy the relation \eqref{transform-1}, and let
	$$\tilde{\Gamma}=\{(y,t)\in\mathbb{R}^{n+1}_+: x=\frac{(y,t)+e_{n+1}}{|(y,t)+e_{n+1}|^2}-e_{n+1}\in\Gamma\}.$$

	\begin{figure}[!htbp]\small
		\center
		\includegraphics[height=1.8in,width=2.6in]{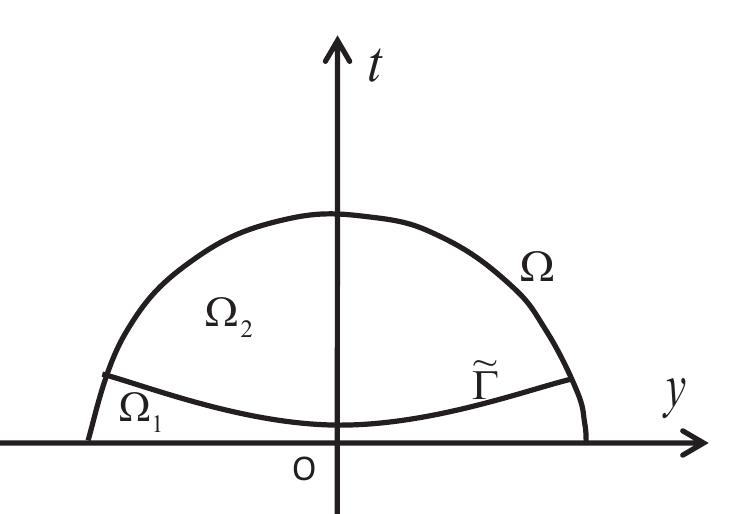}
		\vspace{1em}
		\caption{ Domain of $\Omega$}
	\end{figure}
	For a bounded domain  $\Omega$  in the upper half space (see figure 2). $\tilde{\Gamma}$ divides $\Omega$ into two parts $\Omega_1$ and $\Omega_2$. Take a test function $\phi\in C_0^\infty(\overline{\mathbb{R}^{n+1}_+})$ satisfying $0\leq\phi\leq 1$, $\phi\equiv 0$ in $\mathbb{R}^{n+1}_+\backslash\overline{\Omega}$ and $\phi\equiv1$ in $\Omega_\rho:=\{(y,t)\in\overline{\Omega}: dist((y,t),\partial\Omega\backslash\partial\mathbb{R}^{n+1}_+)>\rho\}$ for some small $\rho>0$.
	We have that
	\begin{equation}\label{equality}
	\int_\Omega t^\alpha\nabla u\cdot\nabla \phi dydt=\int_{\Omega}t^\beta u^{p^*-1}\phi dydt.
	\end{equation}
	When $r$ is sufficiently close to $\frac{1}{2}$, we have that $|\Omega_1|$ is small enough, such that
	$$|\int_{\Omega_1} t^\alpha\nabla u\cdot\nabla \phi dydt|\leq(\int_{\Omega_1}t^\alpha|\nabla u|^2)^{\frac{1}{2}}(\int_{\Omega_1}t^\alpha|\nabla\phi|^2)^{\frac{1}{2}}=o(1),$$
	and
		\begin{align*}
	|\int_{\Omega_1}t^\beta u^{p^*-1}\phi dydt|
	\leq (\int_{\Omega_1}t^\beta u^{p^*} dydt)^{\frac{p^*-1}{p^*}}(\int_{\Omega_1}t^\beta \phi^{p^*}dydt)^{\frac{1}{p^*}}=o(1).
	\end{align*}
	Then back to \eqref{equality},
	\begin{eqnarray*}
		LHS \ of \ \eqref{equality} &=&\int_{\Omega_1} t^\alpha\nabla u\cdot\nabla \phi dydt+\int_{\Omega_2} t^\alpha\nabla u\cdot\nabla \phi dydt\\
	&=&o(1)+\int_{\Omega_2}div(t^\alpha\nabla u\phi)dydt-\int_{\Omega_2}div(t^\alpha\nabla u)\phi dydt\\
		&=&\int_{\tilde{\Gamma}}t^\alpha \frac{\partial u}{\partial \nu}\phi dydt+\int_{\Omega_2}t^\beta u^{p^*-1}\phi dydt+o(1),
	\end{eqnarray*}
	and
	$$RHS \ of \ \eqref{equality}=\int_{\Omega_2}t^\beta u^{p^*-1}\phi dydt+o(1),$$
	where $\nu$ is the unit outer normal vector on $\tilde{\Gamma}$ with respect to $\Omega_2$.
	Therefore, we have that
	\begin{align}\label{o(1)}
	\int_{\tilde{\Gamma}}t^\alpha \frac{\partial u}{\partial \nu}\phi dydt=o(1), \text{  as } r\rightarrow(\frac{1}{2})^-.\end{align}
	Since $$|x+\frac{e_{n+1}}{2}|^2=\frac{1}{4}-\frac{t}{|y|^2+(t+1)^2},$$ we have that
   on $\tilde{\Gamma}$, $r^2=\frac{1}{4}-\frac{t}{|y|^2+(t+1)^2},$ i.e.
	\begin{align}\label{t}
	t=\frac{1+4r^2-\sqrt{16r^2-|y|^2(1-4r^2)^2}}{1-4r^2}.
	\end{align}
	The unit normal vector $\nu$ on $\tilde{\Gamma}$ is
	\begin{eqnarray*}
		\nu=\Big(\frac{1-4r^2}{4r}y,-\frac{\sqrt{16r^2-|y|^2(1-4r^2)^2}}{4r}\Big)\big|_{\tilde{\Gamma}}
		=\Big(\frac{1-4r^2}{4r}y,\frac{1-4r^2}{4r}(t+1)-\frac{1}{2r}\Big)\big|_{\tilde{\Gamma}}.
	\end{eqnarray*}
	Thus
	\begin{eqnarray}\label{partial-nu}
	\frac{\partial u}{\partial \nu}\big|_{\tilde{\Gamma}}=(\nabla u \cdot \nu)\big|_{\tilde{\Gamma}}=\frac{1-4r^2}{4r}\big[\nabla u \cdot(y,t+1)-\frac{2}{1-4r^2}\partial_t u\big]\big|_{\tilde{\Gamma}}.
	\end{eqnarray}
	On the other hand, by \eqref{transform-1}, \eqref{transform-2} and \eqref{t}, the normal derivative of $\psi$ on $\Gamma$ is
	\begin{equation*}
	\begin{split}
	&\frac{\partial\psi}{\partial r}\big|_{\Gamma}= \big(\nabla_x\psi \cdot \frac{x+\frac{e_{n+1}}{2}}{r}\big)\big|_{\Gamma}\\
	=&\frac{[|y|^2+(t+1)^2]^\frac{n+\alpha-1}{2}}{2r}\big(-(n+\alpha-1)(1-t)u+2t\nabla u\cdot (y,t+1)-(|y|^2+(t+1)^2)\partial_t u\big)\big|_{\tilde{\Gamma}}\\
	=&\frac{1}{2r}\big(\frac{4t}{1-4r^2}\big)^{\frac{n+\alpha-1}{2}}\big(-(n+\alpha-1)(1-t)u+2t\nabla u\cdot (y,t+1)-\frac{4t}{1-4r^2}\partial_t u\big)\big|_{\tilde{\Gamma}}.
	\end{split}
	\end{equation*}
	Comparing with \eqref{partial-nu}, we have
	\begin{align*}
	\frac{\partial u}{\partial \nu}\big|_{\tilde{\Gamma}}=(\frac{1-4r^2}{4t})^{\frac{n+\alpha+1}{2}}\big|_{\tilde{\Gamma}}\frac{\partial \psi}{\partial r}\big|_{\Gamma}+\frac{(n+\alpha-1)(1-4r^2)(1-t)}{8rt}u\big|_{\tilde{\Gamma}}.
	\end{align*}
	 Back to \eqref{o(1)}, and by \eqref{transform-2}, $dS|_{\tilde{\Gamma}}=\frac{1}{|x+e_{n+1}|^{2n}}d\tilde{S}|_{\Gamma}$ and $t|_{\tilde{\Gamma}}=\frac{\frac{1}{4}-r^2}{|x+e_{n+1}|^2}|_{\Gamma}$, we have
	\begin{equation}\label{o(1)-1}
	\begin{split}
	o(1)=&\int_{\tilde{\Gamma}}t^\alpha\frac{\partial u}{\partial\nu}\phi dydt\\
	=&\int_\Gamma\frac{(\frac{1}{4}-r^2)^\alpha \frac{\partial\psi}{\partial r}}{|x+e_{n+1}|^{n+\alpha+1}}\phi dx+(n+\alpha-1)\int_\Gamma \frac{(\frac{1}{4}-r^2)^\alpha(2r^2+x_{n+1}+\frac12)}{2r|x+e_{n+1}|^{n+\alpha+3}}\psi\phi dx.
	\end{split}
	\end{equation}
	Since $\psi$ is bounded and $\Gamma$ is far away from the point $-e_{n+1}$, we know that the integral $\int_\Gamma \frac{2r^2+x_{n+1}+\frac12}{2r|x+e_{n+1}|^{n+\alpha+3}}\psi\phi dx$ is bounded.  Then for $\alpha>0$, the second term of \eqref{o(1)-1} is $o(1)$ as $r\to (\frac{1}{2})^-$.
	Since $\psi$ is radially symmetric with respect to $-\frac{e_{n+1}}{2}$, we have that
	$$(\frac{1}{4}-r^2)^\alpha \frac{\partial\psi}{\partial r}\big|_{\Gamma}\int_\Gamma\frac{1}{|x+e_{n+1}|^{n+\alpha+1}}\phi dx=o(1),$$
	which imples
	\[
	\lim_{r\rightarrow(\frac{1}{2})^-}(\frac{1}{4}-r^2)^\alpha\frac{\partial\psi}{\partial r}=0.
	\]

\end{proof}

\bigskip

 \noindent {\bf Acknowledgements}\\
 \noindent  We would like to thank Yazhou Han for bringing our attention to Monti and Morbidelli's work while J. Dou and M. Zhu were working on the extension operators, and thank Xiaodong Wang for bringing Maz'ya book to our attention. We thank Ha\"im Brezis for his personal remark on the discovery of Gagliado-Nirenberg inequality and his interest in this paper. L. Wang especially thanks her supervisor Min Ji for guiding her PhD thesis, which involved the content of this paper. M. Ji helped her to catch the key points and to develop a deep, clear and rigorous logical thinking. M. Ji pointed out some errors in the proof of density lemma (Lemma \ref{density}) and also the need of globally bounded integrals for solutions in the classification results, and provided many valuable inputs, which helped us to make the paper more concise and precise. Besides, it is because of M. Ji's enlightenment that L. Wang found some other errors, such as the respective range of $\alpha$ and $\beta$ in different lemmas, and the rearrangement property in section 6.

  L. Wang is supported by the China Scholarship Council for her study/research at the University of Oklahoma.  L. Wang  and L. Sun would like to thank Department of Mathematics at  the  University of Oklahoma  for its hospitality,  where this work has been done.

 The project is supported by  the National Natural Science Foundation of China (Grant No. 12071269), Natural Science Basic Research Plan for Distinguished Young Scholars  in Shaanxi Province of China (Grant No. 2019JC-19), and the Fundamental Research Funds for the Central Universities (Grant No. GK202101008). The work of L. Wang  is supported by the National Natural Science Foundation of China (Grant No. 11571344).

\end{document}